\documentclass[12pt,reqno]{amsart}
\usepackage[utf8]{inputenc}

\title[Plane partitions and rowmotion]{Plane partitions and rowmotion on rectangular and trapezoidal posets}
\author{Joseph Johnson}
\address{Department of Mathematics, KTH Royal Institute of Technology, Stockholm, Sweden}
\email{josjohn@kth.se}

\author{Ricky Ini Liu}
\address{Department of Mathematics, University of Washington, Seattle, WA 98195}
\email{riliu@uw.edu}

\thanks{The second author was partially supported by grants from the National Science Foundation (DMS 1700302/2204415 and CCF-1900460).}
\date{\today}

\usepackage{amssymb}
\usepackage{amsfonts}
\usepackage{amsthm}
\usepackage{amsmath}
\usepackage{amscd}
\usepackage{enumerate}
\usepackage{t1enc}
\usepackage{mathrsfs}
\usepackage{indentfirst}
\usepackage{graphicx}
\usepackage{graphics}
\usepackage{hyperref}
\usepackage{pict2e}
\usepackage{epic}
\numberwithin{equation}{section}
\usepackage[margin=2.9cm]{geometry}
\usepackage{epstopdf} 
\usepackage{transparent}
\usepackage{wasysym}

\usepackage{caption}

\usepackage{mathtools}

\usepackage{tikz-cd}
\usepackage{tikz}
\usetikzlibrary{decorations.pathreplacing,calligraphy}

\usepackage{mathdots}
\usepackage{todonotes}
\usepackage{verbatim}

 \colorlet{myGreen}{green!50!gray!120!}

\theoremstyle{plain}
\newtheorem{Th}{Theorem}[section]
\newtheorem{Lemma}[Th]{Lemma}
\newtheorem{Cor}[Th]{Corollary}
\newtheorem{Prop}[Th]{Proposition}

 \theoremstyle{definition}
\newtheorem{Def}[Th]{Definition}

\newtheorem{Prob}[Th]{Problem}
\newtheorem{Ex}[Th]{Example}

\theoremstyle{remark}
\newtheorem{Rem}[Th]{Remark}

\DeclareMathOperator{\incorner}{In}
\DeclareMathOperator{\outcorner}{Out}
\DeclareMathOperator{\updeg}{udeg}
\DeclareMathOperator{\downdeg}{ddeg}

\newcommand{\op}{\mathcal{O}(P)}
\newcommand{\cp}{\mathcal{C}(P)}

\newcommand{\spicyrho}{\tilde \varrho}

\renewcommand{\subset}{\subseteq}

\newcommand{\RR}{\mathbb R}
\newcommand{\ZZ}{\mathbb Z}

\newcommand{\phibar}{\psi}
\newcommand{\themap}{\zeta}

\tikzstyle{wW}=[circle, draw, fill=white, inner sep=0pt, minimum width=4pt]
\tikzstyle{wB}=[circle, draw, fill=black, inner sep=0pt, minimum width=4pt]
\tikzstyle{wR}=[circle, draw, fill=red, inner sep=0pt, minimum width=4pt]
\tikzstyle{wBlue}=[circle, draw, fill=blue, inner sep=0pt, minimum width=4pt]

\tikzstyle{bigW}=[circle, draw, fill=white, inner sep=0pt, minimum width=6pt]
\tikzstyle{bigB}=[circle, draw, fill=black, inner sep=0pt, minimum width=6pt]
\tikzstyle{bigR}=[circle, draw, fill=red, inner sep=0pt, minimum width=6pt]
\tikzstyle{bigBlue}=[circle, draw, fill=blue, inner sep=0pt, minimum width=6pt]

\tikzstyle{v}=[circle, fill=black, inner sep=2pt]

\begin{document}

\maketitle

\begin{abstract} 
We define a birational map between labelings of a rectangular poset and its associated trapezoidal poset. This map tropicalizes to a bijection between the plane partitions of these posets of fixed height, giving a new bijective proof of a result by Proctor. We also show that this map is equivariant with respect to birational rowmotion, resolving a conjecture of Williams and implying that birational rowmotion on trapezoidal posets has finite order. 

%
\end{abstract}

\section{Introduction}


For a finite poset $P$, a \emph{plane partition} of $P$ (also known as a \emph{$P$-partition}) is an order-preserving labeling of $P$ with nonnegative integers. Plane partitions and their variations arise in a number of important contexts in combinatorics, representation theory, and related areas; see, for instance, \cite{stanley4} for an overview. 

When $P$ is the \emph{rectangular poset} $R_{r,s}$, the Cartesian product of two chains of $r$ and $s$ elements, an elegant product formula for the number of plane partitions of $P$ with maximum label at most $\ell$ was given by MacMahon \cite{macmahon}. Surprisingly, Proctor \cite{proctor} showed that there is another poset, namely the \emph{trapezoidal poset} $T_{r,s}$, that has the same number of plane partitions with maximum label at most $\ell$ for all $\ell$. (See Figure~\ref{fig:rectangleandtrapezoid} for a depiction of $R_{4,3}$ and $T_{4,3}$.)

Proctor's proof relies on a branching rule for Lie algebra representations and is not bijective---he describes the question of giving a combinatorial correspondence between these two sets of plane partitions as ``a complete mystery.'' Partial bijections were later constructed by Stembridge \cite{stembridge} and Reiner \cite{reiner} for $\ell=1$, and Elizalde \cite{elizalde} for $\ell=2$, but a full bijection for all $\ell$ was not given until work of Hamaker, Patrias, Pechenik, and Williams \cite{hamakerpatriaspechenikwilliams} using a tool from Schubert calculus known as \emph{$K$-theoretic jeu de taquin}. 

Although the bijection given in \cite{hamakerpatriaspechenikwilliams} has many nice properties, it also has a few shortcomings. First, it cannot be extended in a natural way to a continuous piecewise-linear map on real-valued labelings of the rectangle and trapezoid. As a result, it cannot be written using expressions in the tropical semiring (that is, using the operations addition, subtraction, and maximum). Second, it does not appear to be generally well-behaved with respect to a certain map on labelings of posets studied in dynamical algebraic combinatorics called \emph{rowmotion}.

\medskip

(Combinatorial) rowmotion is a term coined by Striker and Williams \cite{strikerwilliams} to describe a map first studied by Brouwer and Schrijver \cite{brouwerschrijver} that permutes the set of order ideals (downward-closed subsets) of a poset. Specifically, rowmotion sends an order ideal $I \subset P$ to the order ideal generated by the minimal elements of $P \setminus I$. It was shown in \cite{brouwerschrijver} that the action of rowmotion on order ideals of the rectangle $R_{r,s}$ has order exactly $r+s$. Since then, rowmotion has received much study; see, for instance, \cite{cameronfonderflaass, fonderflaass, joseph, propproby, striker, strikerwilliams, thomaswilliams}. One important observation by Einstein and Propp \cite{einsteinpropp1} is that one can generalize combinatorial rowmotion to a piecewise-linear map on real labelings of posets, and even further (via ``detropicalization'') to a birational map. Results about birational rowmotion typically then descend to results for piecewise-linear rowmotion (via tropicalization) and further to combinatorial rowmotion. Birational rowmotion on rectangular posets is closely related to the \emph{birational Robinson-Schensted-Knuth (RSK) correspondence}, also known as \emph{tropical} or \emph{geometric RSK}---see \cite{dauvergne, noumiyamada, osz} for some discussion. 

Although the generalized versions of rowmotion no longer act on finite sets, it turns out that for rectangular posets, they still retain many of the important dynamical properties of combinatorial rowmotion. For example, it was shown by Grinberg and Roby \cite{grinbergroby2} that birational rowmotion on $R_{r,s}$ still has order $r+s$. This result was observed by Glick and Grinberg (as noted in \cite{musikerroby, grinbergroby3}) to be equivalent to a phenomenon in discrete dynamics known as \emph{type AA Zamolodchikov periodicity}, first proved by Volkov \cite{volkov}. However, the class of posets for which birational rowmotion is known to have finite order is very small \cite{grinbergroby2, grinbergroby1}. Grinberg and Roby conjecture that birational rowmotion on $T_{r,s}$ also has order $r+s$. (For more on the conjectural good behavior of birational rowmotion and the related \emph{$R$-systems} with respect to singularity confinement and algebraic entropy, see Galashin-Pylyavskyy~\cite{galashinpylyavskyy}.)

Given the apparent close relationship between the rectangular and trapezoidal posets, Williams conjectures (as noted in \cite{grinbergroby2}, based on work in \cite{williams}) that there should exist a birational map between labelings of $R_{r,s}$ and $T_{r,s}$ that intertwines with the action of rowmotion (see also Hopkins \cite{hopkins2} for further discussion). In particular, such a map would prove that birational rowmotion on $T_{r,s}$ has finite order $r+s$.  In work of Dao, Wellman, Yost-Wolff, and Zhang \cite{daowellmanyostwolffzhang}, it was shown that the bijection given in Hamaker-Patrias-Pechenik-Williams \cite{hamakerpatriaspechenikwilliams} does intertwine with combinatorial rowmotion on plane partitions of height 1, thereby showing that combinatorial rowmotion on $T_{r,s}$ has the correct order. However, they also note that it does not respect piecewise-linear or birational rowmotion, so that it cannot be used to prove periodicity on $T_{r,s}$ in these cases.

\medskip

Our main result is to settle this question. We construct a birational map between labelings of $R_{r,s}$ and $T_{r,s}$. We then show that this map:
\begin{itemize}
    \item tropicalizes to a continuous, piecewise-linear map that restricts to a bijection between plane partitions of $R_{r,s}$ and $T_{r,s}$ of height at most $\ell$, and
    \item is equivariant with respect to rowmotion on $R_{r,s}$ and $T_{r,s}$, implying that birational (as well as piecewise-linear and combinatorial) rowmotion on $T_{r,s}$ has order $r+s$.
\end{itemize}

Our construction uses tools from the study of rowmotion, namely the \emph{toggles} of Cameron and Fon-der-Flaass \cite{cameronfonderflaass}, generalized to birational involutions by Einstein and Propp \cite{einsteinpropp1}. The proof that our map gives a bijection between plane partitions utilizes a generalized version of the \emph{chain shifting lemma} proved by the current authors in \cite{johnsonliu} (also closely related to the conversion lemma by Grinberg and Roby \cite{grinbergroby3}). For the rectangle, this lemma shows that the action of rowmotion shifts certain chain statistics in the poset in a predictable way---it was shown by the current authors in \cite{johnsonliu2} to be closely related to properties of Sch\"utzenberger promotion on semistandard Young tableaux, and also appears in a different form in the study of the hidden invariance of last passage percolation by Dauvergne \cite{dauvergne}. We derive a new, simple proof of this lemma based on the duality of plane trees. 

To prove that our birational map respects rowmotion, we utilize an alternate construction of rowmotion known as \emph{birational antichain rowmotion}, introduced by Joseph and Roby \cite{josephroby2}. Our construction deforms labelings of the trapezoid into labelings of the rectangle via certain intermediate posets, and this alternate construction allows us to define a version of rowmotion on these intermediate posets that respects these deformations.

\medskip

The outline of this paper is as follows. In Section~\ref{sec:paper3Background}, we introduce some background and preliminary results about posets and rowmotion. In Section~\ref{section:skewChainShifting}, we show how the action of rowmotion on chains in the poset can be described in terms of weighted arborescences. We then define a key bijection $\aleph$ on arborescences and show how it can be used to give a straightforward combinatorial proof of the chain shifting lemma from \cite{johnsonliu} as well as a number of important generalizations to other skew shapes, including trapezoids. In Section~\ref{section:equivariantMap}, we use rowmotion to construct a birational map between labelings of the rectangle and trapezoid and use the chain shifting lemma to prove that, when tropicalized, it can be used to give a bijection between plane partitions of the same height. (A detailed example of the bijection on plane partitions is given in Section~\ref{sec:polytopes}.) In Section~\ref{sec:rowmotionequivariance}, we prove that the map intertwines with rowmotion on the rectangle and trapezoid, which in particular implies that birational rowmotion on the trapezoid has finite order. We conclude in Section~\ref{sec:conclusion} with some possible directions for future study.

\section{Background}
\label{sec:paper3Background}

\subsection{Posets and labelings}

We will assume all posets to be finite. Given a poset $P$ and a set $K$, a \emph{$K$-labeling} (or just a \emph{labeling} if $K$ is clear) of $P$ is an element $x = (x_p)_{p \in P} \in K^P$, that is, an assignment of an element of $K$ to each element of $P$.

\begin{Rem} \label{rem:semiring}
Typically $K$ can be a field or a semifield (commutative semiring with multiplicative inverses) such as the positive real numbers $(\RR_+, +, \cdot)$. For simplicity, we will state most of our results using $K = \RR_+$, except when discussing plane partitions or polytopes in which case we will let $K$ be the \emph{tropical semiring} $(\RR, \max, +)$. Regardless, most of our results will hold in general for fields (provided one takes the necessary precautions such as using rational maps) and semirings. We refer the reader to Grinberg-Roby \cite{grinbergroby1} for further discussion, as well as Noumi-Yamada \cite{noumiyamada} for discussion of tropicalization of rational expressions.

In the case of semirings, all of our formulas will be subtraction-free until Section~\ref{sec:rowmotionequivariance} so the same proofs will apply. However, even in Section~\ref{sec:rowmotionequivariance}, as noted in \cite[Section 3]{grinbergroby1}, any equality between subtraction-free rational expressions that holds over every field  also holds over semirings, so in particular Theorem~\ref{thm:equivariance} will hold for semirings as well. 
\end{Rem}

A poset $P$ is called \emph{bounded} if it has unique minimum and maximum elements. It will be useful to augment $P$ to a bounded poset $\widehat P$ by adding a minimum element $\hat 0$ and a maximum element $\hat 1$ such that $\hat 0 < p < \hat 1$ for all $p \in P$. We will naturally identify $P$ with the induced subposet $\widehat P \setminus \{\hat 0, \hat 1\}$. Thus any labeling of $\widehat P$ can be restricted to a labeling of $P$, and any labeling of $P$ yields a labeling of $\widehat P$ if we specify the labels at $\hat 0$ and $\hat 1$.

We will call an induced subposet $I$ of $P$ \emph{saturated} if all of its cover relations $p \lessdot q$ are cover relations of $P$. In other words, $I$ is saturated if the Hasse diagram of $I$ is a subgraph of the Hasse diagram of $P$. Note that any maximal chain of a saturated subposet $I$ is a saturated chain in $P$. By default, we will assume that all subposets of interest are saturated, and we will use the word \emph{chain} to refer to a saturated chain.

\subsubsection{Chain and order polytopes}

Following Stanley \cite{stanley2}, we can associate with a poset $P$ two polytopes called the chain and order polytopes.

\begin{Def}
The \emph{chain polytope} $\cp \subset \RR^P$ is the set of all $\RR$-labelings $x = (x_p)_{p \in P}$ such that $x_p \geq 0$ for all $p \in P$, and $\sum_{p \in C} x_p \leq 1$ for all (maximal) chains $C \subset P$.

The \emph{order polytope} $\op \subset \RR^P$ is the set of all $\RR$-labelings $y = (y_p)_{p \in P}$ such that $0 \leq y_p \leq 1$ for all $p \in P$, and $y_p \leq y_q$ if $p \leq q$ in $P$.
\end{Def}

\begin{Def}
\label{Def:planePartition}
A \emph{plane partition} of $P$ (or \emph{$P$-partition}) is a $\mathbb Z_{\geq 0}$-labeling $x \in \mathbb{Z}_{\geq 0}^{P}$ such that $x_p \leq x_q$ whenever $p \leq q$ in $P$.

We say a plane partition $x$ has \emph{height $\ell$} if the maximum label appearing in $x$ is $\ell$.
\end{Def}

Note that plane partitions of height at most $\ell$ are exactly the lattice points in $\ell \mathcal O(P)$. As shown by Stanley \cite{stanley2}, there exists a continuous, piecewise-linear, volume-preserving map between the chain and order polytopes called the \emph{transfer map}, which we will define in Section~\ref{sec:rowmotion}.


\subsection{Rectangles and Trapezoids}

Throughout this paper, fix positive integers $r \geq s$. We will use $[n]$ to denote the set $\{1, 2, \dots, n\}$ with the usual total order.

Given two posets $P$ and $Q$, the \emph{Cartesian product} $P \times Q$ is the poset 
of pairs $(p,q)$ for $p \in P$ and $q \in Q$ with $(p,q) \leq (p', q')$ if $p \leq p'$ in $P$ and $q \leq q'$ in $Q$.

\begin{Def}
The \emph{rectangle poset} $R_{r,s}$ is the poset \[R_{r,s}=[r] \times [s].\]

The \emph{right trapezoid poset} $RT_{r,s}$ is the induced subposet
\[RT_{r,s}=\{(i,j) \mid i-j < r\} \subset R_{r+s-1,s}.\]

The \emph{trapezoid poset} $T_{r,s}$ is the induced subposet
\[T_{r,s}=\{(i,j) \mid i+j > s\} \subset RT_{r,s}.\]
\end{Def}

Note that $R_{r,s}$ is also an induced subposet of $RT_{r,s}$, and $|R_{r,s}| = |T_{r,s}| = rs$.
See Figure~\ref{fig:rectangleandtrapezoid} for an example of the Hasse diagrams of these posets. We will typically draw our posets oriented in the plane so that the first coordinate increases to the northwest and the second coordinate increases to the northeast.

\begin{figure}
    \centering
\begin{tikzpicture}[scale = 0.6]
\foreach \x in {0,...,2}
\draw (\x,\x)--(-3,3+2*\x);

\foreach \x in {0,...,3}
\draw (-\x,\x)--(2-\x,2+\x);

\foreach \x in {0,...,1}
\draw (-3,5+2*\x)--(-2-\x,6+\x);

\foreach [count=\i,
          evaluate=\i as \xmin using int(2-\y)] \y in {0,...,2} {
    \foreach \x in {\xmin,...,3} {
        \draw [fill=black] (- \x + \y, \x + \y) circle [radius=0.15cm];
    }
}

\foreach \x in {0,...,1}
\foreach \y in {\x,...,1}
\draw [fill=black] (-\x + \y-3,\x + \y+5) circle [radius=0.15cm];

\draw (0,0) [fill=black] circle [radius=0.15cm];
\draw (-1,1) [fill=black] circle [radius=0.15cm];
\draw (1,1) [fill=black] circle [radius=0.15cm];

\node at (-3.9,3) {(4,1)};
\node at (-3.9,5) {(5,2)};
\node at (-3.9,7) {(6,3)};

\begin{scope}[shift = {(7,0)}]
\foreach \x in {0,...,2}{
\draw (\x,\x)--(-3+\x,3+\x);
\draw [dashed] (-3+\x,3+\x)--(-3,3+2*\x);
};

\foreach \x in {0,...,3}
\draw (-\x,\x)--(2-\x,2+\x);

\foreach \x in {0,...,1}
\draw [dashed] (-3,5+2*\x)--(-2-\x,6+\x);

\foreach \x in {0,...,3}
\foreach \y in {0,...,2}
\draw [fill=black] (-\x + \y,\x + \y) circle [radius=0.15cm];

\foreach \x in {0,...,1}
\foreach \y in {\x,...,1}
\draw [fill=white] (-\x + \y-3,\x + \y+5) circle [radius=0.15cm];
\end{scope}

\begin{scope}[shift = {(14,0)}]
\foreach \x in {0,...,2}{
\draw (-2+2*\x,2)--(-3,3+2*\x);
\draw [dashed] (\x,\x)--(-2+2*\x,2);
};

\foreach \x in {0,...,2}{
\draw (2-2*\x,2)--(2-\x,2+\x);
\draw [dashed] (-\x,\x)--(2-2*\x,2);
};

\foreach \x in {-1,...,1}
\draw (-3,5+2*\x)--(-2-\x,6+\x);

\foreach \x in {0,...,1}
\draw [fill=white] (-\x ,\x) circle [radius=0.15cm];

\foreach \x in {2,...,3}
\draw [fill=black] (-\x ,\x) circle [radius=0.15cm];

\draw [fill=white] (1,1) circle [radius=0.15cm];

\foreach \x in {1,...,4}
\draw [fill=black] (-\x+1 ,\x+1) circle [radius=0.15cm];

\foreach \x in {0,...,5}
\draw [fill=black] (-\x+2 ,\x+2) circle [radius=0.15cm];
\end{scope}

\end{tikzpicture}
    \caption{The right trapezoid $RT_{4,3}$ and its subposets: the rectangle $R_{4,3}$, and the trapezoid $T_{4,3}$. The labeled points $(i,j)$ on the left border of the right trapezoid satisfy $i-j = r-1$.}
    \label{fig:rectangleandtrapezoid}
\end{figure}
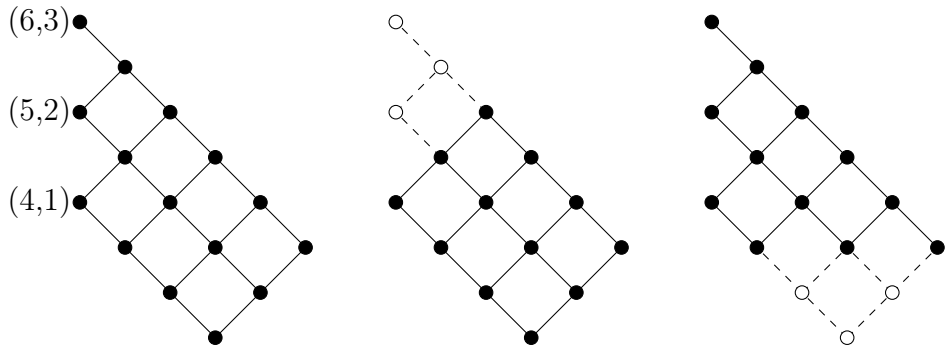

For the rectangle poset $R_{r,s}$, the number of plane partitions is given by the following formula due to MacMahon.
\begin{Th}[MacMahon \cite{macmahon}]
The number of plane partitions of $R_{r,s}$ of height at most $\ell$ is
\[N(r,s,\ell) = \prod_{i=1}^r \prod_{j=1}^s \prod_{k=1}^\ell \frac{i+j+k-1}{i+j+k-2}.\]
\end{Th}

In 1983, Proctor used a branching rule for the inclusion of Lie algebras $\mathfrak{sp}_{2n} (\mathbb C) \hookrightarrow \mathfrak{sl}_{2n}(\mathbb C)$ to prove the following result.

\begin{Th}[Proctor \cite{proctor}]
\label{Th:ppsamenumber}
The posets $R_{r,s}$ and $T_{r,s}$ have the same number of plane partitions of height $\ell$ for all $\ell \in \mathbb{Z}_{\geq 0}$.
\end{Th}

Since Proctor's proof is not bijective, this raises the following natural question.

\begin{Prob}
\label{Prob:pps}
Find a bijective proof that $R_{r,s}$ and $T_{r,s}$ have the same number of plane partitions of height $\ell$ for all $\ell \in \mathbb Z_{\geq 0}$.
\end{Prob}

Understanding the relationship between the plane partitions of the rectangle and trapezoid has been the subject of a number of papers over the past 40 years. Partial results giving bijections for low heights were given by Stembridge \cite{stembridge}, Reiner \cite{reiner}, and Elizalde \cite{elizalde} before a bijective proof for all heights was found in 2016 by Hamaker, Patrias, Pechenik, and Williams \cite{hamakerpatriaspechenikwilliams} in terms of \emph{$K$-theoretic jeu de taquin}. 

\subsubsection{$K$-theoretic jeu de taquin}
We give a brief description of the bijection between plane partitions of $R_{r,s}$ and $T_{r,s}$ given in \cite{hamakerpatriaspechenikwilliams}. This description differs slightly from the one given in \cite{hamakerpatriaspechenikwilliams} but is equivalent under taking the inverse and poset duality; we describe it in this form to be more similar to the bijection we construct later.
(This subsection is provided for informative purposes only and is not related to the remainder of the paper.)

A \emph{strictly increasing labeling} of $P$ is a plane partition $x$ for which $x_p < x_q$ whenever $p < q$. For a graded poset $P$, there is a bijection from plane partitions of $P$ to strictly increasing nonnegative labelings of $P$ obtained by adding the rank of $p$ to $x_p$ for all $p \in P$. Therefore it suffices to give a bijection between strictly increasing nonnegative labelings of $R_{r,s}$ and $T_{r,s}$ of the same height.

The main tool in the bijection is a sliding procedure called \emph{$K$-theoretic jeu de taquin} ($K$-jdt). See Figure~\ref{fig:kjdt} for a depiction of possible slides. Black dots denote unlabeled nodes into which we wish to perform a $K$-jdt slide. We slide the smallest number covering an unlabeled node into that node. Unlike classical jeu de taquin, whenever we have a tie, both numbers slide into the new position, and slides are performed simultaneously at each black dot. The process is then iterated until no more slides are possible.

To perform the bijection from strictly increasing labelings of $T_{r,s}$ to $R_{r,s}$, we embed $T_{r,s}$ in $RT_{r,s}$ and perform a sequence of $K$-jdt slides to fill in the labels of the elements $(i,j) \in R_{r,s} \setminus T_{r,s}$ in decreasing order of $2i+j$. See Figure~\ref{fig:kjdtBijection} for a simple example of this map. It is shown in \cite{hamakerpatriaspechenikwilliams} that the result will be a strictly increasing labeling of $R_{r,s}$ and that this is a bijection.

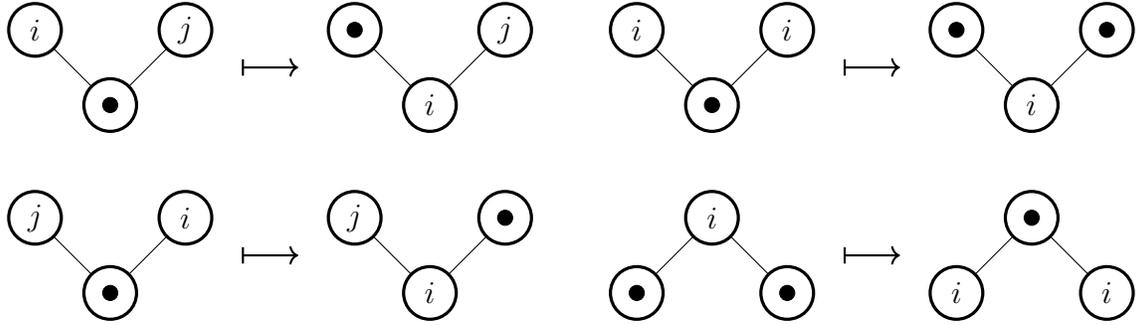
\begin{figure}
    \centering
\begin{tikzpicture}

\begin{scope}[shift={(0,0)}]
\draw (0,0)--(-1,1);
\draw (0,0)--(1,1);

\draw [fill=white, very thick] (-1,1) circle [radius=0.35cm];
\draw [fill=white, very thick] (0,0) circle [radius=0.35cm];
\draw [fill=white, very thick] (1,1) circle [radius=0.35cm];

\node at (-1,1) {$i$};
\node at (1,1) {$j$};
\draw [fill=black] (0,0) circle [radius=0.1cm];

\draw [|->,thick] (1.75,0.5)--(2.5,0.5);

\begin{scope}[shift={(4.25,0)}]
\draw (0,0)--(-1,1);
\draw (0,0)--(1,1);

\draw [fill=white, very thick] (-1,1) circle [radius=0.35cm];
\draw [fill=white, very thick] (0,0) circle [radius=0.35cm];
\draw [fill=white, very thick] (1,1) circle [radius=0.35cm];

\node at (0,0) {$i$};
\node at (1,1) {$j$};
\draw [fill=black] (-1,1) circle [radius=0.1cm];
\end{scope}
\end{scope}

\begin{scope}[shift={(8,0)}]
\draw (0,0)--(-1,1);
\draw (0,0)--(1,1);

\draw [fill=white, very thick] (-1,1) circle [radius=0.35cm];
\draw [fill=white, very thick] (0,0) circle [radius=0.35cm];
\draw [fill=white, very thick] (1,1) circle [radius=0.35cm];

\node at (-1,1) {$i$};
\node at (1,1) {$i$};
\draw [fill=black] (0,0) circle [radius=0.1cm];

\draw [|->,thick] (1.75,0.5)--(2.5,0.5);

\begin{scope}[shift={(4.25,0)}]
\draw (0,0)--(-1,1);
\draw (0,0)--(1,1);

\draw [fill=white, very thick] (-1,1) circle [radius=0.35cm];
\draw [fill=white, very thick] (0,0) circle [radius=0.35cm];
\draw [fill=white, very thick] (1,1) circle [radius=0.35cm];

\node at (0,0) {$i$};
\draw [fill=black] (-1,1) circle [radius=0.1cm];
\draw [fill=black] (1,1) circle [radius=0.1cm];
\end{scope}
\end{scope}

\begin{scope}[shift={(0,-2.5)}]
\draw (0,0)--(-1,1);
\draw (0,0)--(1,1);

\draw [fill=white, very thick] (-1,1) circle [radius=0.35cm];
\draw [fill=white, very thick] (0,0) circle [radius=0.35cm];
\draw [fill=white, very thick] (1,1) circle [radius=0.35cm];

\node at (-1,1) {$j$};
\node at (1,1) {$i$};
\draw [fill=black] (0,0) circle [radius=0.1cm];

\draw [|->,thick] (1.75,0.5)--(2.5,0.5);

\begin{scope}[shift={(4.25,0)}]
\draw (0,0)--(-1,1);
\draw (0,0)--(1,1);

\draw [fill=white, very thick] (-1,1) circle [radius=0.35cm];
\draw [fill=white, very thick] (0,0) circle [radius=0.35cm];
\draw [fill=white, very thick] (1,1) circle [radius=0.35cm];

\node at (0,0) {$i$};
\node at (-1,1) {$j$};
\draw [fill=black] (1,1) circle [radius=0.1cm];
\end{scope}
\end{scope}

\begin{scope}[shift={(8,-2.5)}]
\draw (0,1)--(-1,0);
\draw (0,1)--(1,0);

\draw [fill=white, very thick] (-1,0) circle [radius=0.35cm];
\draw [fill=white, very thick] (0,1) circle [radius=0.35cm];
\draw [fill=white, very thick] (1,0) circle [radius=0.35cm];

\node at (0,1) {$i$};
\draw [fill=black] (-1,0) circle [radius=0.1cm];
\draw [fill=black] (1,0) circle [radius=0.1cm];

\draw [|->,thick] (1.75,0.5)--(2.5,0.5);

\begin{scope}[shift={(4.25,0)}]
\draw (0,1)--(-1,0);
\draw (0,1)--(1,0);

\draw [fill=white, very thick] (-1,0) circle [radius=0.35cm];
\draw [fill=white, very thick] (0,1) circle [radius=0.35cm];
\draw [fill=white, very thick] (1,0) circle [radius=0.35cm];

\node at (-1,0) {$i$};
\node at (1,0) {$i$};
\draw [fill=black] (0,1) circle [radius=0.1cm];
\end{scope}
\end{scope}
\end{tikzpicture}
    \caption{K-theoretic jeu de taquin slides. If $i$ is the smallest label covering an unlabeled node (so $i < j$), then $i$ slides to the unlabeled node.}
    \label{fig:kjdt}
\end{figure}


\begin{Th}[\cite{hamakerpatriaspechenikwilliams}]
    The $K$-theoretic jeu de taquin moves in $RT_{r,s}$ described above bijectively rectify strictly increasing labelings of $T_{r,s}$ to strictly increasing labelings of $R_{r,s}$. In particular, $R_{r,s}$ and $T_{r,s}$ have the same number of plane partitions of each height.
\end{Th}

\begin{figure}
    \begin{tikzpicture}[scale=.7, transform shape]
        \begin{scope}
            \draw (0,2)--(1,3) (-1,3)--(0,4) (-2,4)--(-1,5) (0,2)--(-2,4) (1,3)--(-2,6);
            \foreach \i/\j/\l in {1/1/\bullet,1/2/1,2/1/0,2/2/2,3/1/2,3/2/3,4/2/4}
                \node[draw=black,fill=white,circle,inner sep=0pt,minimum size=14pt] (\i\j) at (\j-\i,\i+\j){$\l$};
        \end{scope}
        \begin{scope}[shift={(4.5,0)}]
            \draw (0,2)--(1,3) (-1,3)--(0,4) (-2,4)--(-1,5) (0,2)--(-2,4) (1,3)--(-2,6);
            \foreach \i/\j/\l in {1/1/0,1/2/1,2/1/\bullet,2/2/2,3/1/2,3/2/3,4/2/4}
                \node[draw=black,fill=white,circle,inner sep=0pt,minimum size=14pt] (\i\j) at (\j-\i,\i+\j){$\l$};
        \end{scope}
        \begin{scope}[shift={(9,0)}]
            \draw (0,2)--(1,3) (-1,3)--(0,4) (-2,4)--(-1,5) (0,2)--(-2,4) (1,3)--(-2,6);
            \foreach \i/\j/\l in {1/1/0,1/2/1,2/1/2,2/2/\bullet,3/1/\bullet,3/2/3,4/2/4}
                \node[draw=black,fill=white,circle,inner sep=0pt,minimum size=14pt] (\i\j) at (\j-\i,\i+\j){$\l$};
        \end{scope}
        \begin{scope}[shift={(13.5,0)}]
            \draw (0,2)--(1,3) (-1,3)--(0,4) (-2,4)--(-1,5) (0,2)--(-2,4) (1,3)--(-2,6);
            \foreach \i/\j/\l in {1/1/0,1/2/1,2/1/2,2/2/3,3/1/3,3/2/\bullet,4/2/4}
                \node[draw=black,fill=white,circle,inner sep=0pt,minimum size=14pt] (\i\j) at (\j-\i,\i+\j){$\l$};
        \end{scope}
        \begin{scope}[shift={(18,0)}]
            \draw (0,2)--(1,3) (-1,3)--(0,4) (-2,4)--(-1,5) (0,2)--(-2,4) (1,3)--(-2,6);
            \foreach \i/\j/\l in {1/1/0,1/2/1,2/1/2,2/2/3,3/1/3,3/2/4,4/2/\bullet}
                \node[draw=black,fill=white,circle,inner sep=0pt,minimum size=14pt] (\i\j) at (\j-\i,\i+\j){$\l$};
        \end{scope}
        \node at (1.5,4){$\to$};
        \node at (6,4){$\to$};
        \node at (10.5,4){$\to$};
        \node at (15,4){$\to$};
    \end{tikzpicture}
    \caption{An example of the $K$-jeu de taquin bijection \cite{hamakerpatriaspechenikwilliams} from strictly increasing labelings of $T_{3,2}$ to strictly increasing labelings of $R_{3,2}$.}
    \label{fig:kjdtBijection}
\end{figure}
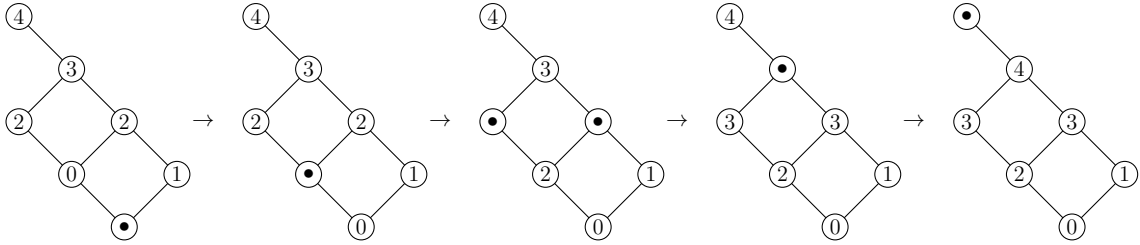

It should be noted that by dividing the labels by $k$, $K$-jdt gives a bijection between rational points in the order polytopes $\mathcal O(T_{r,s})$ and $\mathcal O(R_{r,s})$ with denominator dividing $k$. However, it is not hard to see that this map cannot be extended to a continuous map on the real points of the order polytopes in a natural way (compatible with all $k$). For instance, for the two slides on the left of Figure~\ref{fig:kjdt}, the outcomes are very different (and will generally continue to diverge on future slides) even if $i$ and $j$ are close in value.

\subsection{Rowmotion} \label{sec:rowmotion}
In this section we define birational toggles and rowmotion. Note that we define these as maps on $\RR_+$-labelings of $P$, but one can also define them as rational maps on $\RR$-labelings of $P$.  (See Remark~\ref{rem:semiring}.) Throughout, we will identify the labelings in $\RR_+^P$ with the corresponding labelings in $\RR_+^{\widehat P}$, where the labels at $\hat 0$ and $\hat 1$ are both $1$. (Although rowmotion and related notions are sometimes defined with different labels at $\hat 0$ and $\hat 1$, this does not have any significant effect on the generality of results when $P$ is graded since one can renormalize as discussed in \cite{einsteinpropp1}.) For any subset $S \subset P$ and $x \in \RR_+^P$, we will define the \emph{weight} of $S$ in $x$ to be the product $\prod_{p \in S} x_p$.

\begin{Def}
For any $p \in P$, the \emph{toggle} $t_p \colon \RR_+^{P} \to \RR_+^{P}$ is the map that changes the $p$-coordinate of $x \in \RR_+^{P}$ by
\[x_p \mapsto \left(\sum\limits_{q \gtrdot p} \frac{1}{x_q} \right)^{-1} \left( \sum\limits_{q \lessdot p} x_q \right) \frac{1}{x_p}\]
while keeping all other coordinates fixed.
\end{Def}

\begin{Def}
The \emph{rowmotion map} $\rho \colon \RR_+^P \to \RR_+^P$ is the composition
\[\rho = t_{L^{-1}(1)} \circ t_{L^{-1}(2)} \circ \cdots \circ t_{L^{-1}(n)}\] for any linear extension (order-preserving bijection) $L \colon P \to [n]$.
\end{Def}

It is easy to verify that toggles $t_p$ and $t_q$ commute if and only if $p$ and $q$ are not adjacent in the Hasse diagram of $P$. From this, one can deduce that $\rho$ does not depend on the choice of linear extension.

There is an alternate definition of rowmotion in terms of the following transfer map. 

\begin{Def}
The \emph{transfer map} $\psi^{-1} \colon \RR_+^P \to \RR_+^P$ is defined coordinatewise by, for $p \in P$,
\[x_p \mapsto \frac{x_p}{\sum_{q \lessdot p} x_q}.\]
Its inverse $\phibar$ acts by
\[x_p \mapsto \sum_{\hat 0 \lessdot q_1 \lessdot \dots \lessdot q_n = p} \prod_{i=1}^n x_{q_i}.\]
Thus $\phibar(x)_p$ is the total weight of all maximal chains in the interval $[\hat 0, p]$ in $x$.
\end{Def}

If one works over the tropical semiring (and sets $x_{\hat 0} = 0$ and $x_{\hat 1}=1$), then the transfer map defines a piecewise-linear, continuous bijection from the order polytope of $P$ to the chain polytope of $P$, as shown by Stanley \cite{stanley2}.

One can similarly define a dual transfer map.

\begin{Def}
The \emph{dual transfer map} ${\psi^*}^{-1} \colon \RR_+^P \to \RR_+^P$ is defined coordinatewise by, for $p \in P$,
\[x_p \mapsto \frac{x_p}{\sum_{q \gtrdot p} x_q}.\]
Its inverse $\phibar^*$ acts by
\[x_p \mapsto \sum_{\hat 1 \gtrdot q_1 \gtrdot \dots \gtrdot q_n = p} \prod_{i=1}^n x_{q_i}.\]
Thus $\phibar^*(x)_p$ is the total weight of all maximal chains in the interval $[p,\hat 1]$ in $x$.
\end{Def}

The following lemma is originally due to Einstein and Propp in a slightly different form (see \cite{einsteinpropp2,josephroby2}).


\begin{Lemma} \label{lemma:dualtransfer}
    Let $x \in \RR_+^P$, $z=\psi^{-1} \circ \rho^{-1} \circ \psi(x)$, and $p \in P$. Then $\psi(x)_p\psi^*(z)_p = 1$.
\end{Lemma}
\begin{proof}
    Assume by induction that the claim holds for all elements of $P$ above $p$. Since $\psi(x) = \rho(\psi(z))$, the value of $\psi(x)_p$ is determined when toggling $p$ while performing rowmotion on $\psi(z)$. At this point, the labels above $p$ agree with $\psi(x)$, while the labels below $p$ agree with $\psi(z)$. Therefore we find that
    \begin{align*}
      \psi(x)_p &= \left(\sum_{q \gtrdot p} \frac{1}{\psi(x)_q}\right)^{-1} \cdot \left(\sum_{q \lessdot p} \psi(z)_q\right)\cdot \psi(z)_p^{-1}\\
      &= \left(\sum_{q \gtrdot p} \psi^*(z)_q\right)^{-1} \cdot z_p^{-1}\\
      &= \psi^*(z)_p^{-1},
    \end{align*}
    where we used the induction hypothesis for $q \gtrdot p$ and the fact that $\psi(z)_p = z_p \cdot \sum_{q \lessdot p} \psi(z)_q$ in the first step, and the analogous expression for $\psi^*(z)_p$ in the second.
\end{proof}

\begin{Ex}
    Let $P = RT_{3,2}$, and let $x \in \RR_+^P$ be labeled as in the leftmost diagram of Figure~\ref{fig:r32}. Then we can compute $y=\psi(x)$, $\rho^{-1}(y)$, and $z = \psi^{-1}(\rho^{-1}(y))$ as shown. Note that the values of $\psi^*(z)_p$ agree with the values of $y_p^{-1} = \psi(x)_p^{-1}$. For instance, when $p = (2,2)$, \[\psi^*(z)_{22} = \frac{1}{a(bd+be+ce)fg} \cdot g \cdot \frac{(bd+be+ce)f}{(b+c)e} = \frac{1}{a(b+c)e} = y_{22}^{-1}.\]
    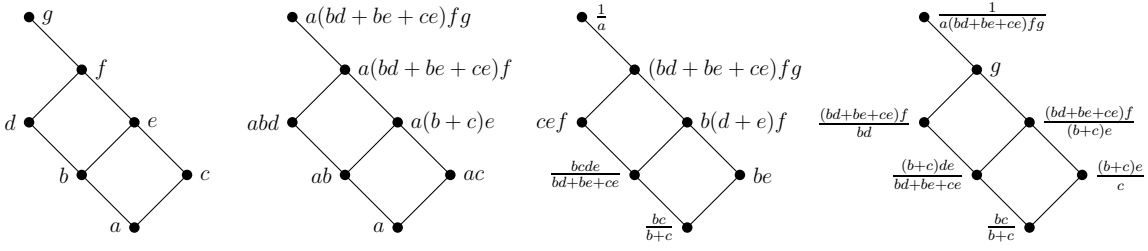
\begin{figure}
        \centering
        \begin{tikzpicture}[scale=.7, transform shape]
            \begin{scope}[shift={(-.5,0)}]
                \draw (0,0)--(-2,2) (1,1)--(-2,4) (0,0)--(1,1) (-1,1)--(0,2) (-2,2)--(-1,3);
                \node[circle, fill=black, inner sep=2pt, label=left:$a$] at (0,0){};
                \node[circle, fill=black, inner sep=2pt, label=left:$b$] at (-1,1){};
                \node[circle, fill=black, inner sep=2pt, label=left:$d$] at (-2,2){};
                \node[circle, fill=black, inner sep=2pt, label=right:$c$] at (1,1){};
                \node[circle, fill=black, inner sep=2pt, label=right:$e$] at (0,2){};
                \node[circle, fill=black, inner sep=2pt, label=right:$f$] at (-1,3){};
                \node[circle, fill=black, inner sep=2pt, label=right:$g$] at (-2,4){};
            \end{scope}
            \begin{scope}[shift={(4.5,0)}]
                \draw (0,0)--(-2,2) (1,1)--(-2,4) (0,0)--(1,1) (-1,1)--(0,2) (-2,2)--(-1,3);
                \node[circle, fill=black, inner sep=2pt, label=left:$a$] at (0,0){};
                \node[circle, fill=black, inner sep=2pt, label=left:$ab$] at (-1,1){};
                \node[circle, fill=black, inner sep=2pt, label=left:$abd$] at (-2,2){};
                \node[circle, fill=black, inner sep=2pt, label=right:$ac$] at (1,1){};
                \node[circle, fill=black, inner sep=2pt, label=right:$a(b+c)e$] at (0,2){};
                \node[circle, fill=black, inner sep=2pt, label=right:$a(bd+be+ce)f$] at (-1,3){};
                \node[circle, fill=black, inner sep=2pt, label=right:$a(bd+be+ce)fg$] at (-2,4){};
            \end{scope}
            \begin{scope}[shift={(10,0)}]
                \draw (0,0)--(-2,2) (1,1)--(-2,4) (0,0)--(1,1) (-1,1)--(0,2) (-2,2)--(-1,3);
                \node[circle, fill=black, inner sep=2pt, label=left:$\frac{bc}{b+c}$] at (0,0){};
                \node[circle, fill=black, inner sep=2pt, label=left:$\frac{bcde}{bd+be+ce}$] at (-1,1){};
                \node[circle, fill=black, inner sep=2pt, label=left:$cef$] at (-2,2){};
                \node[circle, fill=black, inner sep=2pt, label=right:$be$] at (1,1){};
                \node[circle, fill=black, inner sep=2pt, label=right:$b(d+e)f$] at (0,2){};
                \node[circle, fill=black, inner sep=2pt, label=right:$(bd+be+ce)fg$] at (-1,3){};
                \node[circle, fill=black, inner sep=2pt, label=right:$\frac{1}{a}$] at (-2,4){};
            \end{scope}
            \begin{scope}[shift={(16.5,0)}]
                \draw (0,0)--(-2,2) (1,1)--(-2,4) (0,0)--(1,1) (-1,1)--(0,2) (-2,2)--(-1,3);
                \node[circle, fill=black, inner sep=2pt, label=left:$\frac{bc}{b+c}$] at (0,0){};
                \node[circle, fill=black, inner sep=2pt, label=left:$\frac{(b+c)de}{bd+be+ce}$] at (-1,1){};
                \node[circle, fill=black, inner sep=2pt, label=left:$\frac{(bd+be+ce)f}{bd}$] at (-2,2){};
                \node[circle, fill=black, inner sep=2pt, label=right:$\frac{(b+c)e}{c}$] at (1,1){};
                \node[circle, fill=black, inner sep=2pt, label=right:$\frac{(bd+be+ce)f}{(b+c)e}$] at (0,2){};
                \node[circle, fill=black, inner sep=2pt, label=right:$g$] at (-1,3){};
                \node[circle, fill=black, inner sep=2pt, label=right:$\frac{1}{a(bd+be+ce)fg}$] at (-2,4){};
            \end{scope}
        \end{tikzpicture}
        \caption{Labelings of $RT_{3,2}$. From left to right: $x$, $y=\psi(x)$, $\rho^{-1}(y)$, and $z =\tilde\rho^{-1}(x) = \psi^{-1}(\rho^{-1}(y))$.}
        \label{fig:r32}
    \end{figure}
\end{Ex}

\subsubsection{Antichain rowmotion}

Lemma~\ref{lemma:dualtransfer} allows us to give an alternate formulation of rowmotion via its conjugation under the transfer map. (This map was called \emph{birational antichain rowmotion} or \emph{barmotion} by Joseph and Roby \cite{josephroby2}.)

\begin{Def}
For any $p \in P$, the \emph{(antichain) toggle} $\tau_p \colon \RR_+^{P} \to \RR_+^{P}$ is the map that changes the $p$-coordinate of $x \in \RR_+^{P}$ by
\[x_p \mapsto \left(\sum_{C} \prod_{q \in C} x_q\right)^{-1},\]
where $C$ ranges over all maximal chains of $P$ containing $p$, while keeping all other coordinates fixed.
\end{Def}

\begin{Def}
The \emph{(antichain) rowmotion map} $\tilde\rho \colon \RR_+^P \to \RR_+^P$ is the composition
\[\tilde\rho = \tau_{L^{-1}(n)} \circ \cdots \circ \tau_{L^{-1}(2)}\circ \tau_{L^{-1}(1)}\] for any linear extension (order-preserving bijection) $L \colon P \to [n]$.
\end{Def}
It is easy to verify that toggles $\tau_p$ and $\tau_q$ commute if and only if $p$ and $q$ are incomparable in $P$, which implies that $\tilde \rho$ does not depend on the choice of linear extension. Note that the $\tau_p$ here are applied from bottom to top as opposed to the $t_p$ in the definition of $\rho$, which were applied from top to bottom.

In fact, $\tilde\rho$ and $\rho$ are conjugate to one another, as shown in the following proposition due to Joseph and Roby \cite{josephroby2}.
\begin{Prop} \label{prop:rhotilde}
The equality $\tilde \rho = \psi^{-1} \circ \rho \circ \phibar$ holds.
\end{Prop}
\begin{proof}
    Suppose $\tilde \rho(z) = x$. During the application of $\tilde \rho$ to $z$, when applying $\tau_p$ to find $x_p$ for some $p \in P$, the labels below $p$ agree with $x$ while the labels above $p$ agree with $z$. Each maximal chain in $\widehat P$ through $p$ can be written as a union of a maximal chain from $[\hat 0, p]$ and one from $[p, \hat 1]$ (that overlap at $p$), so before applying $\tau_p$, the weight of the maximal chains through $p$ is
    $\frac{\psi(x)_p \psi^*(z)_p}{x_p}$. But this equals $x_p^{-1}$ by the definition of $\tau_p$, so we find that $\psi(x)_p \psi^*(z)_p = 1$. Comparing with Lemma~\ref{lemma:dualtransfer} then gives the result. 
\end{proof}

    
    

One important property of $\tilde\rho$ is the following identity.

\begin{Cor}\label{cor:edgeweight}
    Let $x \in \RR_+^P$, $y = \phibar(x)$, and $z = \tilde \rho^{-1}(x)$. Then for $p \in P$,
    \begin{align*}
        x_p^{-1} &= \sum_{q \lessdot p} \frac{y_q}{y_p},\\
        z_p^{-1} &= \sum_{q \gtrdot p} \frac{y_p}{y_q}.
    \end{align*}
\end{Cor}
\begin{proof}
The first equation follows by applying the definition of $\psi^{-1}$ to $x = \psi^{-1}(y)$. For the second, by Proposition~\ref{prop:rhotilde} and Lemma~\ref{lemma:dualtransfer} we have
    $\phibar^*(z)_p = \frac{1}{y_p}$.
    Therefore the equation $z = {\psi^*}^{-1}(\phibar^*(z))$ gives
    \[z_p = \frac{\frac{1}{y_p}}{\sum_{q \gtrdot p} \frac{1}{y_q}} = \frac{1}{\sum_{q \gtrdot p} \frac{y_p}{y_q}}.\]
    Inverting gives the desired result.
\end{proof}

\begin{Ex}
    Referring again to Figure~\ref{fig:r32}, one can compute $z$ directly from $y$ using Corollary~\ref{cor:edgeweight}. For example, 
    \[z_{21}^{-1} = \frac{y_{21}}{y_{31}}+\frac{y_{21}}{y_{22}} = \frac{ab}{abd} + \frac{ab}{a(b+c)e} = \frac{bd+be+ce}{(b+c)de}.\]
\end{Ex}

We will also need the following easy fact about $\tilde\rho^{-1}$ for later.
\begin{Prop} \label{prop:arbitrarya}
    Suppose $P$ has a unique minimum element $a$. Let $x \in \RR_+^P$ and $z = \tilde \rho^{-1}(x)$. Then for any nonmaximal $p \in P$, $z_p$ does not depend on $x_a$.

    Similarly, if $P$ has a unique maximum element $b$, $z \in \RR_+^P$, and $x = \tilde \rho(z)$, then for any nonminimal $p \in P$, $x_p$ does not depend on $z_b$.
\end{Prop}
\begin{proof}
    By the definition of $\phibar$, since $a$ is a minimum of $P$, $y_p=\phibar(x)_p$ is a multiple of $x_a$ for all $p \in P$, but $\frac{y_p}{x_a}$ is independent of $x_a$. This factor of $x_a$ then cancels in every term in the expression for $z_p^{-1}$ in Corollary~\ref{cor:edgeweight}. 

    The second statement is proved similarly, using the fact that $\frac{1}{y_p} = \phibar^*(z)_p$ for $p \in P$.
\end{proof}

\section{Chain Shifting in Skew Shapes}
\label{section:skewChainShifting}

In this section we prove a chain shifting lemma for rowmotion on skew shapes. This lemma is a generalization of the chain shifting lemma for rectangles proved by the current authors in \cite{johnsonliu}, which was also proved in the noncommutative setting by Grinberg and Roby \cite{grinbergroby3}. As we will see, the proof given here is simpler than those earlier proofs and relies mainly on a duality between plane trees.

\subsection{Arborescences}

Let $P$ be a poset and $x \in \RR^P_+$. We first describe how to express the weight $w_S(x)=\prod_{p \in S} x_p$ for particular subsets $S \subseteq P$ in terms of $y=\phibar(x)$. (Recall that we set $y_{\hat 0} = y_{\hat 1} = 1$.)


\begin{Def}
\label{Def:upanddowndegree}
Let $G$ be a (not necessarily induced) subgraph of the Hasse diagram of $P$. (The edges of $G$ are therefore cover relations of $P$).

The \emph{up degree} $\updeg_G(p)$ of a vertex $p \in G$ is the number of edges in $G$ that connect $p$ to an element that covers $p$ in $P$. Similarly, the \emph{down degree} $\downdeg_G(p)$ is the number of edges connecting $p$ to an element covered by $p$.
\end{Def}

\begin{Def}
\label{Def:spanningTree}
An \emph{upward arborescence of $P$} is a subgraph of $\widehat P \setminus \{\hat 1\}$ such that every element of $P$ has down degree $1$. Similarly, a \emph{downward arborescence of $P$} is a subgraph of $\widehat P \setminus \{\hat 0\}$ such that every element of $P$ has up degree $1$.

We denote the set of upward and downward arborescences of $P$ by $U_P$ and $D_P$, respectively.
\end{Def}


Define the \emph{weight} (with respect to $y$) of the edge $e$ corresponding to the cover relation $p \lessdot q$ to be $\omega_e(y) = \frac{y_p}{y_q}$ and the \emph{weight} of an arborescence $T$ to be $\omega_T(y) = \prod_{e \in E(T)} \omega_e(y)$. (Note that the notion of the weight $\omega_T(y)$ of an arborescence $T$ is different from the notion of the weight $w_S(x)$ of a subset $S$ of the poset.)

\begin{Ex}
Consider the right trapezoid $RT_{3,2}$ in Figure~\ref{fig:upwardSpanningTree} with its upward and downward arborescences shown. The weight of the first upward arborescence can be computed as
\[\frac{1}{y_{11}} \cdot \frac{y_{11}}{y_{21}} \cdot \frac{y_{11}}{y_{12}} \cdot \frac{y_{21}}{y_{31}} \cdot \frac{y_{21}}{y_{22}} \cdot \frac{y_{31}}{y_{32}} \cdot \frac{y_{32}}{y_{42}} = \frac{y_{11}y_{21}}{y_{12}y_{22}y_{42}}.\]
The weights of the other arborescences can be computed similarly. 

Let us take $y=\psi(x)$ as in Figure~\ref{fig:r32}. Then the total weight of all four upward arborescences is\[\frac{y_{11}(y_{21}+y_{12})(y_{31}+y_{22})}{y_{12}y_{22}y_{31}y_{42}}=\frac{a(ab+ac)(abd+a(b+c)e)}{ac \cdot a(b+c)e \cdot abd \cdot a(bd+be+ce)f} = \frac{1}{abcdef},\] which is $w_P(x)^{-1}$.
\end{Ex}
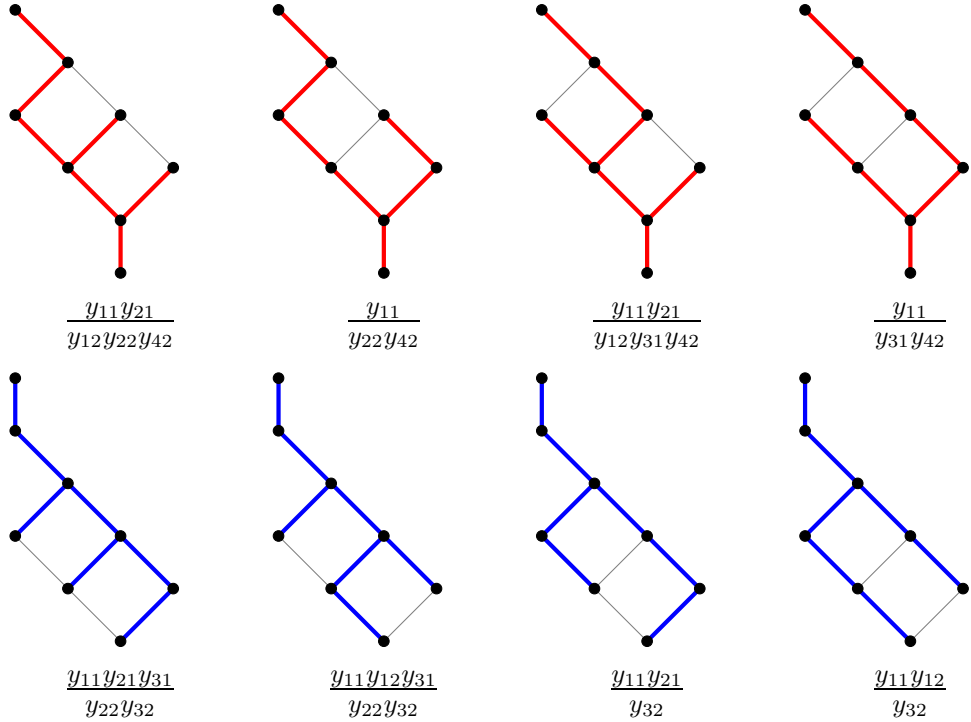
\begin{figure}
    \centering
\begin{tikzpicture}[scale = 0.7]
\begin{scope}[shift={(0,0)}]
\draw[gray] (0,0)--(1,1);
\draw[gray] (-1,1)--(0,2);
\draw[gray] (-2,2)--(-1,3);

\draw[gray] (0,0)--(-2,2);
\draw[gray] (1,1)--(-2,4);

\draw [red, ultra thick] (0,0)--(1,1);
\draw [red, ultra thick] (0,0)--(-2,2);
\draw [red, ultra thick] (-1,3)--(-2,4);

\draw [red, ultra thick] (-2,2)--(-1,3);
\draw [red, ultra thick] (-1,1)--(0,2);

\draw [red, ultra thick] (0,0)--(0,-1);

\node[wB] at (0,0) {};
\node[wB] at (1,1) {};
\node[wB] at (-1,1) {};
\node[wB] at (0,2) {};
\node[wB] at (-2,2) {};
\node[wB] at (-1,3) {};
\node[wB] at (-2,4) {};
\node[wB] at (0,-1) {};

\node at (0,-2) {\large$\frac{y_{11}y_{21}}{y_{12}y_{22}y_{42}}$};
\end{scope}

\begin{scope}[shift={(5,0)}]
\draw[gray] (0,0)--(1,1);
\draw[gray] (-1,1)--(0,2);
\draw[gray] (-2,2)--(-1,3);

\draw[gray] (0,0)--(-2,2);
\draw[gray] (1,1)--(-2,4);

\draw [red, ultra thick] (0,0)--(1,1);
\draw [red, ultra thick] (0,0)--(-2,2);
\draw [red, ultra thick] (-1,3)--(-2,4);

\draw [red, ultra thick] (-2,2)--(-1,3);
\draw [red, ultra thick] (1,1)--(0,2);

\draw [red, ultra thick] (0,0)--(0,-1);

\node[wB] at (0,0) {};
\node[wB] at (1,1) {};
\node[wB] at (-1,1) {};
\node[wB] at (0,2) {};
\node[wB] at (-2,2) {};
\node[wB] at (-1,3) {};
\node[wB] at (-2,4) {};
\node[wB] at (0,-1) {};

\node at (0,-2) {\large$\frac{y_{11}}{y_{22}y_{42}}$};
\end{scope}

\begin{scope}[shift={(10,0)}]
\draw[gray] (0,0)--(1,1);
\draw[gray] (-1,1)--(0,2);
\draw[gray] (-2,2)--(-1,3);

\draw[gray] (0,0)--(-2,2);
\draw[gray] (1,1)--(-2,4);

\draw [red, ultra thick] (0,0)--(1,1);
\draw [red, ultra thick] (0,0)--(-2,2);
\draw [red, ultra thick] (-1,3)--(-2,4);

\draw [red, ultra thick] (0,2)--(-1,3);
\draw [red, ultra thick] (-1,1)--(0,2);

\draw [red, ultra thick] (0,0)--(0,-1);

\node[wB] at (0,0) {};
\node[wB] at (1,1) {};
\node[wB] at (-1,1) {};
\node[wB] at (0,2) {};
\node[wB] at (-2,2) {};
\node[wB] at (-1,3) {};
\node[wB] at (-2,4) {};
\node[wB] at (0,-1) {};

\node at (0,-2) {\large$\frac{y_{11}y_{21}}{y_{12}y_{31}y_{42}}$};
\end{scope}

\begin{scope}[shift={(15,0)}]
\draw[gray] (0,0)--(1,1);
\draw[gray] (-1,1)--(0,2);
\draw[gray] (-2,2)--(-1,3);

\draw[gray] (0,0)--(-2,2);
\draw[gray] (1,1)--(-2,4);

\draw [red, ultra thick] (0,0)--(1,1);
\draw [red, ultra thick] (0,0)--(-2,2);
\draw [red, ultra thick] (-1,3)--(-2,4);
\draw [red, ultra thick] (0,2)--(-1,3);
\draw [red, ultra thick] (1,1)--(0,2);

\draw [red, ultra thick] (0,0)--(0,-1);

\node[wB] at (0,0) {};
\node[wB] at (1,1) {};
\node[wB] at (-1,1) {};
\node[wB] at (0,2) {};
\node[wB] at (-2,2) {};
\node[wB] at (-1,3) {};
\node[wB] at (-2,4) {};
\node[wB] at (0,-1) {};

\node at (0,-2) {\large$\frac{y_{11}}{y_{31}y_{42}}$};
\end{scope}

\begin{scope}[shift={(0,-8)}]
\draw[gray] (0,0)--(1,1);
\draw[gray] (-1,1)--(0,2);
\draw[gray] (-2,2)--(-1,3);

\draw[gray] (0,0)--(-2,2);
\draw[gray] (1,1)--(-2,4);

\draw [blue, ultra thick] (0,0)--(1,1);
\draw [blue, ultra thick] (1,1)--(-1,3);
\draw [blue, ultra thick] (-1,3)--(-2,4);
\draw [blue, ultra thick] (-2,2)--(-1,3);
\draw [blue, ultra thick] (-1,1)--(0,2);
\draw [blue, ultra thick] (-2,4)--(-2,5);

\node[wB] at (0,0) {};
\node[wB] at (1,1) {};
\node[wB] at (-1,1) {};
\node[wB] at (0,2) {};
\node[wB] at (-2,2) {};
\node[wB] at (-1,3) {};
\node[wB] at (-2,4) {};
\node[wB] at (-2,5) {};

\node at (0,-1) {\large$\frac{y_{11}y_{21}y_{31}}{y_{22}y_{32}}$};
\end{scope}

\begin{scope}[shift={(5,-8)}]
\draw[gray] (0,0)--(1,1);
\draw[gray] (-1,1)--(0,2);
\draw[gray] (-2,2)--(-1,3);

\draw[gray] (0,0)--(-2,2);
\draw[gray] (1,1)--(-2,4);

\draw [blue, ultra thick] (0,0)--(-1,1);
\draw [blue, ultra thick] (1,1)--(-1,3);
\draw [blue, ultra thick] (-1,3)--(-2,4);
\draw [blue, ultra thick] (-2,2)--(-1,3);
\draw [blue, ultra thick] (-1,1)--(0,2);

\draw [blue, ultra thick] (-2,4)--(-2,5);

\node[wB] at (0,0) {};
\node[wB] at (1,1) {};
\node[wB] at (-1,1) {};
\node[wB] at (0,2) {};
\node[wB] at (-2,2) {};
\node[wB] at (-1,3) {};
\node[wB] at (-2,4) {};
\node[wB] at (-2,5) {};

\node at (0,-1) {\large$\frac{y_{11}y_{12}y_{31}}{y_{22}y_{32}}$};
\end{scope}

\begin{scope}[shift={(10,-8)}]
\draw[gray] (0,0)--(1,1);
\draw[gray] (-1,1)--(0,2);
\draw[gray] (-2,2)--(-1,3);

\draw[gray] (0,0)--(-2,2);
\draw[gray] (1,1)--(-2,4);

\draw [blue, ultra thick] (0,0)--(1,1);
\draw [blue, ultra thick] (1,1)--(-1,3);
\draw [blue, ultra thick] (-1,3)--(-2,4);
\draw [blue, ultra thick] (-2,2)--(-1,3);
\draw [blue, ultra thick] (-1,1)--(-2,2);

\draw [blue, ultra thick] (-2,4)--(-2,5);

\node[wB] at (0,0) {};
\node[wB] at (1,1) {};
\node[wB] at (-1,1) {};
\node[wB] at (0,2) {};
\node[wB] at (-2,2) {};
\node[wB] at (-1,3) {};
\node[wB] at (-2,4) {};
\node[wB] at (-2,5) {};

\node at (0,-1) {\large$\frac{y_{11}y_{21}}{y_{32}}$};
\end{scope}

\begin{scope}[shift={(15,-8)}]
\draw[gray] (0,0)--(1,1);
\draw[gray] (-1,1)--(0,2);
\draw[gray] (-2,2)--(-1,3);

\draw[gray] (0,0)--(-2,2);
\draw[gray] (1,1)--(-2,4);

\draw [blue, ultra thick] (-2,2)--(-1,3);
\draw [blue, ultra thick] (0,0)--(-2,2);
\draw [blue, ultra thick] (-1,3)--(-2,4);
\draw [blue, ultra thick] (0,2)--(-1,3);
\draw [blue, ultra thick] (1,1)--(0,2);

\draw [blue, ultra thick] (-2,4)--(-2,5);

\node[wB] at (0,0) {};
\node[wB] at (1,1) {};
\node[wB] at (-1,1) {};
\node[wB] at (0,2) {};
\node[wB] at (-2,2) {};
\node[wB] at (-1,3) {};
\node[wB] at (-2,4) {};
\node[wB] at (-2,5) {};

\node at (0,-1) {\large$\frac{y_{11}y_{12}}{y_{32}}$};
\end{scope}

\end{tikzpicture}
    \caption{The four upward arborescences in $U_{RT_{3,2}}$ and the four downward arborescences in $D_{RT_{3,2}}$, together with their weights. 
    }
    \label{fig:upwardSpanningTree}
\end{figure}

The following proposition relates the weight of $P$ with arborescences.

\begin{Prop} \label{prop:weightofP}
    Let $P$ be a poset and $x \in \RR_+^P$, and let $y=\psi(x)$ and $z = \tilde \rho^{-1}(x)$. Then
    \begin{align*}
        \frac{1}{w_P(x)} &= \sum_{T \in U_P} \omega_T(y),\\
        \frac{1}{w_P(z)} &= \sum_{T \in D_P} \omega_T(y).
    \end{align*}
\end{Prop}
\begin{proof}
    For the first equation, by Corollary~\ref{cor:edgeweight},
    \begin{align*}
        \frac{1}{w_P(x)} = \prod_{p \in P} x_p^{-1} = \prod_{p \in P} \sum_{q \lessdot p} \frac{y_q}{y_p}. 
    \end{align*}
    Each upward arborescence contains exactly one edge downward from each element of $P$, so the right hand side is the total weight of all upward arborescences in $U_P$, which gives the first equation. The proof of the second equation is similar.
\end{proof}

For any $y \in \RR_+^P$, we can use $\omega_T(y)$ to define probability measures on $U_P$ and $D_P$: for any subsets $U \subset U_P$ and $D \subset D_P$, define
\begin{align*}
    \mu_y(U) &= \frac{\sum_{T \in U} \omega_T(y)}{\sum_{T \in U_P} \omega_T(y)}, &
    \mu_y(D) &= \frac{\sum_{T \in D} \omega_T(y)}{\sum_{T \in D_P} \omega_T(y)}.
\end{align*}
We will show that these are closely related to the weights of chains in $P$.

\subsubsection{Weights of chains}


For any saturated chain $C \subset P$, we say that an arborescence \emph{contains} $C$ if it contains all of the edges representing cover relations in $C$. Let $U_P(C)$ and $D_P(C)$ denote the sets of all upward and downward arborescences that contain $C$. The following proposition relates the weight of $C$ (as a subset of $P$) to the weights of the arborescences containing $C$.
 \begin{Prop} \label{prop:chainarborescence}
     Let $P$ be a poset, $x \in \RR_+^P$, $y = \phibar(x)$, and $z = \tilde\rho^{-1}(x)$. Let $C$ be a saturated chain in $P$ with minimum $a$ and maximum $b$. Then
     \begin{align*}
     w_C(x) &= \left(\sum_{a' \lessdot a} \frac{y_{a'}}{y_b}\right)^{-1} \cdot \mu_y(U_P(C)),\tag{$*$}\label{eq:wcx}\\
     w_C(z) &= \left(\sum_{b' \gtrdot b} \frac{y_{a}}{y_{b'}}\right)^{-1} \cdot \mu_y(D_P(C)).\tag{$**$}\label{eq:wcz}
     \end{align*}
 \end{Prop}
\begin{proof}
    The product of the weights of the edges in $C$ is $\frac{y_{a}}{y_{b}}$. The elements of $U_P(C)$ are obtained from $C$ by adding one downward edge from each element of $(P \setminus C) \cup \{a\}$. It follows from Corollary~\ref{cor:edgeweight} that
\begin{align*}
\sum_{T \in U_P(C)} \omega_T(y) &= \left(\frac{y_{a}}{y_{b}} \cdot \sum_{a' \lessdot a} \frac{y_{a'}}{y_{a}} \right) \cdot \prod_{p \in P \setminus C} \sum_{q \lessdot p} \frac{y_q}{y_p}\\
&=\sum_{a' \lessdot a} \frac{y_{a'}}{y_{b}} \cdot \prod_{p \in P \setminus C} x_p^{-1}\\
&=\sum_{a' \lessdot a} \frac{y_{a'}}{y_{b}} \cdot \frac{w_{C}(x)}{w_P(x)}.
\end{align*}
Subsitituting for $w_P(x)$ using Proposition~\ref{prop:weightofP} and
rearranging gives the result.
The other equation is proved similarly.
\end{proof}

\begin{Ex}
    Consider again the four upward arborescences in $P=RT_{3,2}$ as depicted in Figure~\ref{fig:upwardSpanningTree}, and let $x$ and $y = \psi(x)$ be as shown in Figure~\ref{fig:r32}. 
    
    Suppose $C$ is the chain $(2,1) \lessdot (3,1) \lessdot (3,2)$. Only the first two arborescences shown contain $C$, and their weights are $\frac{y_{31}}{y_{22}}$ times the weights of the last two arborescences. Thus the fraction of upward arborescence weight covered by the two containing $C$ is
    \[\mu_y(U_P(C)) = \frac{y_{31}}{y_{31}+y_{22}}.\]
    Plugging in the coordinates from Figure~\ref{fig:r32} gives
    \[\frac{y_{31}}{y_{31}+y_{22}} = \frac{abd}{abd+a(b+c)e} = bdf \cdot \frac{a}{a(bd+be+ce)f} = w_C(x) \cdot \frac{y_{11}}{y_{32}}\]
    as required by Proposition~\ref{prop:chainarborescence}.
\end{Ex}

Given a collection $\mathscr C$ of (saturated) chains, define
\[U_P(\mathscr C) = \bigsqcup_{C \in \mathscr C} U_P(C) \quad \text{and} \quad D_P(\mathscr C) = \bigsqcup_{C \in \mathscr C} D_P(C).\]
We will typically be interested in finding the total weight \[w_{\mathscr C}(x) = \sum_{C \in \mathscr C} w_C(x).\] 
Note that if $a = \min(C)$ and $b=\max(C)$ are fixed for all $C \in \mathscr C$, then it is easy to sum \eqref{eq:wcx} (and similarly \eqref{eq:wcz}) over all $C \in \mathscr C$ since the first factor on the right hand side will be fixed. Thus the result will be the same equation \eqref{eq:wcx} but with $C$ replaced by $\mathscr C$. In addition, each arborescence will contain at most one $C \in \mathscr C$, so there would be no need to worry about multiplicities in $U_P(\mathscr C)$.

The most important consequence is the following corollary, which we will use in Section~\ref{sec:chainshifting}.
\begin{Cor} \label{cor:bijectionimpliesweight}
    Let $P$ be a poset, $x \in \RR_+^P$, and $z = \tilde \rho^{-1}(x)$. Let $m,m',M,M' \in P$ such that $m'$ is the unique element covered by $m$ and $M$ is the unique element covering $M'$.
    
    Suppose $\mathscr C$ and $\mathscr C'$ are collections of saturated chains such that $(\min(C),\max(C)) = (m,M)$ for all $C \in \mathscr C$, and $(\min(C'), \max(C')) = (m',M')$ for all $C' \in \mathscr C'$. Then $\mu_y(U_P(\mathscr C)) = \mu_y(D_P(\mathscr C'))$ implies $w_{\mathscr C}(x) = w_{\mathscr C'}(z)$.
\end{Cor}
\begin{proof}
    Summing \eqref{eq:wcx} and \eqref{eq:wcz} over all elements of $\mathscr C$ and $\mathscr C'$, respectively, gives
    \[w_{\mathscr C}(x) = \left(\frac{y_{m'}}{y_{M}}\right)^{-1} \cdot \mu_y(U_P(\mathscr C)) \quad \text{and} \quad w_{\mathscr C'}(z) = \left(\frac{y_{m'}}{y_{M}}\right)^{-1} \cdot \mu_y(D_P(\mathscr C')).\qedhere\]
\end{proof}

\subsection{Skew shapes}

For the rest of this section, we will mainly be concerned with a particular type of poset called a \emph{skew shape}.

\begin{Def}
    A \emph{skew shape poset} $S$ is a saturated subposet of a rectangular poset $R_{r,s}$ containing $(1,1)$ and $(r,s)$ such that if $(i-1,j),(i,j-1) \in S$, then $(i,j) \in S$.
\end{Def}

Since $S$ is a saturated subposet containing the minimum and maximum of $R_{r,s}$, the leftmost elements at each rank form a maximal chain, as do the rightmost elements. The final condition guarantees that all elements in between are also contained in $S$. See Figure~\ref{fig:cornerPoints} for an example. (A skew shape can equivalently be described via the boxes of a connected skew Young diagram $\lambda/\mu$, or as the distributive lattice of order ideals of width 2 posets.) Note that the right trapezoid $RT_{r,s}$ is an example of a skew shape.

\begin{figure}
    \centering
\begin{tikzpicture}[scale = 0.5]
\draw (0,0)--(-2,2);
\draw (1,1)--(-3,5);
\draw (2,2)--(-2,6);
\draw (3,3)--(-2,8);
\draw (4,4)--(-1,9);
\draw (3,7)--(0,10);

\draw (0,0)--(4,4);
\draw (-1,1)--(3,5);
\draw (-2,2)--(3,7);
\draw (-2,4)--(2,8);
\draw (-3,5)--(1,9);
\draw (-2,8)--(0,10);

\node[bigB] at (4,4) {};
\node[bigB] at (3,7) {};
\node[bigB] at (-2,2) {};
\node[bigB] at (-3,5) {};
\node[bigB] at (-2,8) {};

\node[bigW] at (2,6) {};
\node[bigW] at (-1,3) {};
\node[bigW] at (-1,7) {};
\end{tikzpicture}
    \caption{A skew shape and its corner points. Outward corner points are marked in black and inward corner points are marked in white.}
    \label{fig:cornerPoints}
\end{figure}
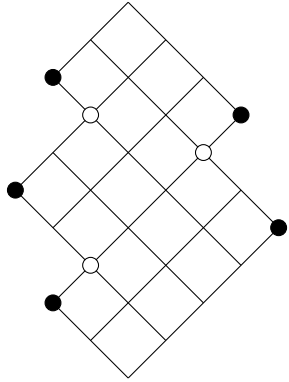

To discuss the boundary of $S$, we introduce the notion of \emph{corner points}.

\begin{Def}
Let $S$ be a skew shape and let $(i,j) \in S$. We say that $(i,j)$ is:
\begin{itemize}
    \item a \emph{left outward corner point} if $(i-1,j),(i,j+1) \in S$ but $(i+1,j),(i,j-1) \not\in S$;
    \item a \emph{right outward corner point} if $(i+1,j),(i,j-1) \in S$ but $(i-1,j),(i,j+1) \not\in S$;
    \item a \emph{left inward corner point} if $(i+1,j),(i,j-1) \in S$ but $(i+1,j-1) \not\in S$; and
    \item a \emph{right inward corner point} if $(i-1,j),(i,j+1) \in S$ but $(i-1,j+1) \not\in S$.
\end{itemize}
We denote the set of outward (resp.\ inward) corner points by $\outcorner(S)$ (resp.\ $\incorner(S)$). 
\end{Def}
See Figure~\ref{fig:cornerPoints} for an illustration.
As a warning, an element can be a corner point in more than one way. For example, in the skew shape with vertices $\{(1,1), (2,1), (2,2)\}$, the vertex $(2,1)$ is both a left outward corner point and a right inward corner point. In the case that $p$ is both a left inward and right inward corner point, then we take $\incorner(S)$ to be a multiset containing $p$ with multiplicity $2$. Finally, note that the minimum and maximum of $S$ are never corner points.

\subsection{The main bijection}

In this section, we will exhibit a bijection between $U_P$ and $D_P$ in the case that $P$ is a skew shape and show that this bijection multiplies weight by a constant factor. (Such a bijection does not exist for most posets $P$, so this is a special feature of skew shapes.) This will allow us to use Propositions~\ref{prop:weightofP} and \ref{prop:chainarborescence} to relate the chain statistics $w_{\mathscr C}(x)$ before and after rowmotion.

Throughout this section, fix a skew shape $S \subset R_{r,s}$. Note that if $q \in S$ only covers a single element $p$, then any element of $U_S$ must contain the edge $p \lessdot q$, so we call this edge \emph{forced} for any upward arborescence. Likewise, the edge $p \lessdot q$ is \emph{forced} for any downward arborescence if $q$ is the only element that covers $p \in S$.

We define a bijection $\aleph \colon U_S \to D_S$ as follows. Translate the Hasse diagram of $S$ in the plane by the vector $(-\frac12, -\frac12)$ (i.e., downward) to form a shifted diagram $\overline S$. Then any $T \in U_S$ has a corresponding translation $\overline T$. We then form $\aleph(T)$ by taking all edges of $S$ that do not intersect $\overline{T}$, together with all forced edges for downward arborescences. See Figure~\ref{fig:treeBijection1} for an example. (Aside from some special behavior along the boundary of $S$, $\aleph$ is the standard bijection between spanning trees of a planar graph and its dual graph.)

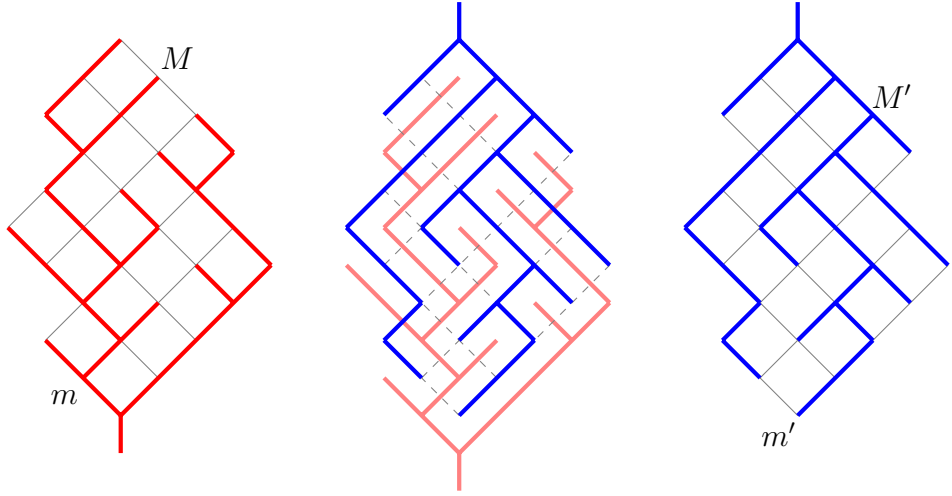
\begin{figure}
\begin{tikzpicture}[scale=0.5]
\begin{scope}
    \draw [gray] 
    (0,0)--(-2,2)
    (1,1)--(-3,5)
    (2,2)--(-2,6)
    (3,3)--(-2,8)
    (4,4)--(-1,9)
    (3,7)--(0,10)
    (0,0)--(4,4)
    (-1,1)--(3,5)
    (-2,2)--(3,7)
    (-2,4)--(2,8)
    (-3,5)--(1,9)
    (-2,8)--(0,10);
    \draw [ultra thick, red] 
    (0,-1)--(0,0)
    (0,0)--(4,4) (-1,1)--(1,3) (-1,3)--(1,5) (2,6)--(3,7) (-2,6)--(1,9) (-2,8)--(0,10)
    (0,0)--(-2,2) (0,2)--(-3,5) (0,4)--(-2,6) (3,3)--(2,4) (1,5)--(0,6) (-1,7)--(-2,8) (4,4)--(1,7) (3,7)--(2,8);
    \node at (-1.5,.5){$m$};
    \node at (1.5,9.5){$M$};
\end{scope}
\begin{scope}[shift={(9,-1)}]
    \draw [ultra thick, red!50] 
    (0,-1)--(0,0)
    (0,0)--(4,4) (-1,1)--(1,3) (-1,3)--(1,5) (2,6)--(3,7) (-2,6)--(1,9) (-2,8)--(0,10)
    (0,0)--(-2,2) (0,2)--(-3,5) (0,4)--(-2,6) (3,3)--(2,4) (1,5)--(0,6) (-1,7)--(-2,8) (4,4)--(1,7) (3,7)--(2,8);
\end{scope}
\begin{scope}[shift={(9,0)}]
    \draw [gray, dashed] 
    (0,0)--(-2,2)
    (1,1)--(-3,5)
    (2,2)--(-2,6)
    (3,3)--(-2,8)
    (4,4)--(-1,9)
    (3,7)--(0,10)
    (0,0)--(4,4)
    (-1,1)--(3,5)
    (-2,2)--(3,7)
    (-2,4)--(2,8)
    (-3,5)--(1,9)
    (-2,8)--(0,10);
    \draw [ultra thick, blue]
    (0,0)--(2,2) (0,2)--(2,4) (-2,2)--(-1,3) (-1,5)--(2,8) (-3,5)--(1,9) (-2,8)--(0,10)--(0,11)
    (-1,1)--(-2,2) (-1,3)--(-3,5) (2,2)--(1,3) (0,4)--(-1,5) (3,3)--(0,6) (4,4)--(1,7) (3,7)--(0,10);
\end{scope}
\begin{scope}[shift={(18,0)}]
    \draw [gray] 
    (0,0)--(-2,2)
    (1,1)--(-3,5)
    (2,2)--(-2,6)
    (3,3)--(-2,8)
    (4,4)--(-1,9)
    (3,7)--(0,10)
    (0,0)--(4,4)
    (-1,1)--(3,5)
    (-2,2)--(3,7)
    (-2,4)--(2,8)
    (-3,5)--(1,9)
    (-2,8)--(0,10);
    \draw [ultra thick, blue]
    (0,0)--(2,2) (0,2)--(2,4) (-2,2)--(-1,3) (-1,5)--(2,8) (-3,5)--(1,9) (-2,8)--(0,10)--(0,11)
    (-1,1)--(-2,2) (-1,3)--(-3,5) (2,2)--(1,3) (0,4)--(-1,5) (3,3)--(0,6) (4,4)--(1,7) (3,7)--(0,10);
    \node at (-0.5,-.5){$m'$};
    \node at (2.5,8.5){$M'$};
\end{scope}
\end{tikzpicture}
\caption{An upward arborescence $T$ (in red) and its image $\aleph(T)$ (in blue). After shifting $T$ down to $\overline T$, we form $\aleph(T)$ by taking all edges that do not intersect $\overline T$, together with all forced edges. Note that the only intersecting edges occur near the inner corners of $S$.}
\label{fig:treeBijection1}
\end{figure}

As an abuse of language, we say edges $e \in T$ and $e' \in \aleph(T)$ \emph{intersect} if the shifted edge $\overline e \in \overline T$ intersects $e'$. If $e$ is, say, the covering relation $(i,j) \lessdot (i,j+1)$, then it only has the potential to intersect $e'$ if $e'$ is the edge $(i-1,j) \lessdot (i,j)$. Moreover, these edges are either both forced (if $(i,j)$ is a right inward corner) or neither is forced. Therefore, $e \in T$ and $e' \in \aleph(T)$ can only intersect if they are both forced edges incident to an inward corner of $S$.

\begin{Prop} \label{prop:alephbijection}
    The map $\aleph \colon U_S \to D_S$ is a well-defined bijection.
\end{Prop}
\begin{proof}
    We need to check that for any $T \in U_S$, we have $\aleph(T) \in D_S$, that is, for any $a \in S$, there is exactly one edge $a \lessdot b$ in $\aleph(T)$.
    
    If $a$ is only covered by one element, then this is clear since all forced edges are in $\aleph(T)$, so suppose $a$ is covered by two elements $b$ and $c$. Then $b$ and $c$ are both covered by a unique element $d\in S$. By the nonintersecting condition,
    $\aleph(T)$ contains the edge $a \lessdot b$ if and only if $T$ does not contain the edge $b \lessdot d$, and similarly with the edges $a \lessdot c$ and $c \lessdot d$. Since $T \in U_S$, it contains exactly one of $b \lessdot d$ and $c \lessdot d$, so $\aleph(T)$ will contain exactly one of $a \lessdot b$ and $a \lessdot c$, as desired.

    To see that $\aleph$ is a bijection, one can define $\aleph^{-1}$ in an analogous manner to $\aleph$ by rotating the plane by a half turn.
\end{proof}

We now show how $\aleph$ affects the weight of an arborescence. 

\begin{Lemma}
\label{lemma:weightChange}
Let $T \in U_S$ and $y \in \mathbb{R}^{S}_+$. Then there exists a Laurent monomial $y^{\alpha(S)}$ depending only on $S$ such that $\omega_{\aleph(T)}(y) = \omega_T(y)\cdot  y^{\alpha(S)}$ for all $T \in U_S$.
Specifically,
\[\frac{ \omega_{\aleph(T)}(y) }{ \omega_T(y) }  = y_{11}^{2-\updeg_S(1,1)}y_{rs}^{2-\downdeg_S(r,s)} \cdot 
\frac{\prod_{p \in \outcorner(S)} y_p }{  \prod_{p \in \incorner(S)} y_p }.\]
\end{Lemma}

\begin{proof}
Note that $\omega_T(y)$ contains a factor of $y_p$ for each edge $p \lessdot q$ in $T$ and a factor of $y_p^{-1}$ for each edge $q \lessdot p$ in $T$. Therefore we can write the exponent of $y_p$ in $\omega_{\aleph(T)}(y)/\omega_T(y)$ as
\begin{align*}
\alpha_{p} &= (\updeg_{\aleph(T)}(p) - \downdeg_{\aleph(T)}(p))-(\updeg_T(p) - \downdeg_T(p))\\
&= 2-\downdeg_{\aleph(T)}(p)-\updeg_T(p).
\end{align*}

If $p = (1,1)$, then $\downdeg_{\aleph(T)}(1,1)=0$ while $\updeg_T(1,1)=\updeg_S(1,1)$ since the upward edges from $(1,1)$ are forced in $T$. Thus $\alpha_{11} = 2-\updeg_S(1,1)$. Similarly for the maximum element $(r,s)$, we have $\alpha_{rs} = 2-\downdeg_S(r,s)$.

For any other element $p = (i,j)$, we can consider the contributions to $2-\alpha_p= \downdeg_{\aleph(T)}(p)+\updeg_{T}(p)$ from the possible edges to the right, namely $(i,j) \lessdot (i,j+1)$ and $(i-1,j) \lessdot (i,j)$.
\begin{itemize}
    \item If $p$ is a right outward corner point, then clearly there is no contribution.
    \item If $p$ is a right inward corner point, then both edges are forced in their respective arborescences, so the contribution is $2$.
    \item If only one of the two edges, say $(i,j) \lessdot (i,j+1)$, exists in $S$, then it is a forced edge for $T$, so the contribution is $1$ (and similarly if the other edge is the only one present).
    \item If both edges exist in $S$ but $p$ is not a right outward corner point, then neither edge is forced. Then by the construction of $\aleph$, $(i,j) \lessdot (i,j+1)$ is an edge of $T$ if and only if $(i-1,j) \lessdot (i,j)$ is not an edge of $\aleph(T)$. It follows that the contribution is $1$ in this case as well.
\end{itemize}
Since the contributions from the possible edges to the left of $p$ can be obtained symmetrically, we find that $2-\alpha_p$ is usually $2$, but we must subtract $1$ if $p$ is an outward corner point and add $1$ if $p$ is an inward corner point (these adjustments cancel out if $p$ is both an inward and an outward corner point). The result follows easily.
\end{proof}

In other words, $\aleph$ scales the weight of every arborescence by the same amount. It follows that it must preserve the measure $\mu_y$.

\begin{Cor} \label{cor:measurepreserving}
   The bijection $\aleph \colon U_S \to D_S$ is measure-preserving: $\mu_y(U) = \mu_y(\aleph(U))$ for all $y \in \RR_+^S$ and subsets $U \subset U_S$.
\end{Cor}
\begin{proof}
    By Lemma~\ref{lemma:weightChange}, both the numerator and denominator of $\mu_y(\aleph(U))$ are obtained from those of $\mu_y(U)$ by scaling by $y^{\alpha(S)}$, so these factors cancel out.
\end{proof}

\begin{Ex}
    Consider again the arborescences for $RT_{3,2}$ in Figure~\ref{fig:upwardSpanningTree}. The bijection $\aleph$ sends each upward arborescence to the downward arborescence directly below it. In each case, $\aleph$ multiplies the weight by $y_{42} \cdot \frac{y_{12}y_{31}}{y_{32}}$, as predicted by Lemma~\ref{lemma:weightChange}. As a consequence, each upward arborescence occupies the same fraction of weight as the corresponding downward arborescence.
\end{Ex}

\subsection{Chain shifting} \label{sec:chainshifting}

We are now ready to prove a chain shifting lemma for skew shapes by combining Corollaries~\ref{cor:bijectionimpliesweight} and \ref{cor:measurepreserving}.

The following lemma summarizes the properties we need of $\aleph$. Given a saturated chain $C \subset S$, we say that $p \in S$ lies \emph{southeast} (resp. \emph{southwest}) of $C$ if it lies to the east (resp. west) of an element of the downward shift $\overline C$. We say that $p$ \emph{lies beyond} $C$ if it does not lie southeast or southwest of $C$, that is, when $\overline C$ does not contain an element at the same height as $p$. 

\begin{Lemma} \label{lemma:chaincross}
    Let $T \in U_S$ and $C$ a saturated chain contained in $T$. Let $p \in S$, and let $C'$ be the unique chain upwards from $p$ to $\hat 1$ in $\aleph(T) \in D_S$.

    If $p$ lies southeast of $C$, and $C$ contains no (forced) edge northeast from a right inward corner point, then every element of $C'$ lies to the southeast of or beyond $C$. In particular, if $C$ contains the rightmost element at any rank above $p$, then $C'$ must contain its southeast neighbor.

The analogous statement obtained by reflecting across the vertical axis also holds.
\end{Lemma}
\begin{proof}
    The only way that $C'$ can go from an element southeast of $C$ to an element southwest of $C$ is along an edge of $\aleph(T)$ to the northwest that intersects an edge of $C \subset T$ to the northeast. But this can only occur if the edge in $C$ is a forced edge of the given type.
\end{proof}

If $p = (i,j)$, denote by $se(p)$ the set of elements of $S$ of the form $(i-1-k,j+k)$ for $k \geq 0$. (These are the elements that lie east of the downward shift $\overline p$.) For instance, $se(p) \neq 0$ if and only if $p$ has a southeast neighbor.

Similarly denote by $sw(p)$ the set of elements $(i-1-k,j+k)$ with $k < 0$, and define $nw(p) = \{q \mid p \in se(q)\}$ and $ne(p) = \{q \mid p \in sw(q)\}$. 

\subsubsection{Simple chain shifting} \label{sec:simplechainshifting}

Our first form of chain shifting is a generalization of the chain shifting lemma for rectangles proven by the current authors in \cite{johnsonliu} (and in the noncommutative setting by Grinberg and Roby \cite{grinbergroby3}) to skew shapes $S$.

Given elements $p < q$ in $S$, let $\mathscr C_{p,q}$ denote the set of all saturated chains from $p$ to $q$, and let $\mathscr C^{se}_{p,q} \subset \mathscr C_{p,q}$ denote the subset consisting of those chains $C$ for which $se(r) \neq \varnothing$ for all $r \in C$. (Also define the analogous notation for the directions $sw$, $ne$, and $nw$.)

Recall that if $\mathscr C$ is a collection of chains, then $U_S(\mathscr C) = \bigsqcup_{C \in \mathscr C} U_S(C)$ is the set of all arborescences containing some chain in $\mathscr C$, and $D_S(\mathscr C)$ is defined similarly.

\begin{Lemma} \label{lemma:skewchainshifting}
Let $m' \lessdot m$ and $M' \lessdot M$ be elements of $S$ such that $sw(m) = ne(M') = \varnothing$.
\begin{enumerate}[(a)]
\item The bijection $\aleph$ restricts to a bijection from $U_S(\mathscr C^{se}_{m,M})$ to $D_S(\mathscr C^{nw}_{m',M'})$.
\item Let $x \in \RR_+^S$ and $z = \tilde\rho^{-1}(x)$. Then $w_{\mathscr C^{se}_{m,M}}(x) = w_{\mathscr C^{nw}_{m',M'}}(z)$.
\end{enumerate}
\end{Lemma}
\begin{proof}
Note that the conditions imply that $m'$ is the only element covered by $m$ (with $m'$ to the southeast) and $M$ is the only element covering $M'$ (with $M'$ to the southeast). Choose any $T \in U_S(\mathscr C^{se}_{m,M})$, and let $C$ be the chain in $T$ from $m$ to $M$. By definition of $\mathscr C^{se}_{m,M}$, $C$ does not contain any (forced) edge northeast from a right inward corner point. Hence by Lemma~\ref{lemma:chaincross}, $\aleph(T) \in D_S$ contains a saturated chain $C'$ from $m'$ to $M'$ that lies entirely southeast of $C$. Thus $C' \in \mathscr C^{nw}_{m',M'}$ and hence $\aleph(T) \in D_S(\mathscr C^{nw}_{m',M'})$. 
A similar argument shows that $\aleph^{-1}$ sends $D_S(\mathscr C^{nw}_{m',M'})$ to $U_S(\mathscr C^{se}_{m,M})$, completing part (a). Part (b) then follows from part (a), Corollary~\ref{cor:measurepreserving}, and Corollary~\ref{cor:bijectionimpliesweight}.
\end{proof}

\begin{Ex}
    As an illustration of Lemma~\ref{lemma:skewchainshifting}, consider Figure~\ref{fig:treeBijection1} with the elements $m$, $M$, $m'$, and $M'$ indicated. The upward arborescence $T$ depicted contains a chain $C \in \mathscr C^{se}_{m,M}$. When applying $\aleph$ as in the center picture, we find that the chain in $\aleph(T)$ upward from $m'$ must remain to the right of $\overline{C}$ and therefore pass through $M'$. Thus $\aleph(T)$ contains a chain $C' \in \mathscr C^{nw}_{m',M'}$ as in Lemma~\ref{lemma:skewchainshifting}.
\end{Ex}

\begin{Ex}
    Consider again $S=RT_{32}$ with arborescences shown in Figure~\ref{fig:upwardSpanningTree}. If we let $m = (2,1)$, $M = (3,2)$, $m'=(1,1)$, and $M' = (2,2)$, then the conditions of Lemma~\ref{lemma:skewchainshifting} are satisfied. Only the first three upward arborescences contain a chain from $m$ to $M$ (which are all in $\mathscr C^{se}_{m,M}$), and $\aleph$ sends these to the first three downward arborescences, which are the only ones that contain a chain in $\mathscr C^{nw}_{m',M'}$.

    We can verify that Lemma~\ref{lemma:skewchainshifting}(b) holds in this case using the labels in Figure~\ref{fig:r32}:
    \begin{align*}
        z_{11}z_{21}z_{22}+z_{11}z_{12}z_{22} &= \frac{bc}{b+c} \left( \frac{(b+c)de}{bd+be+ce}+\frac{(b+c)e}{c}\right) \frac{(bd+be+ce)f}{(b+c)e}\\
        &= \frac{bcdf}{b+c} + \frac{b(bd+be+ce)f}{b+c}\\
        &= b(d+e)f = x_{21}x_{31}x_{32}+x_{21}x_{22}x_{32}.
    \end{align*}
    This equality holds even though neither term on the left is equal to a term on the right. (This is because $\aleph$ does not restrict to a bijection between $U_S(C)$ and $D_S(C')$ for any single chains $C \in \mathscr C^{se}_{m,M}$ and $C'\in \mathscr C^{nw}_{m',M'}$).
\end{Ex}

As this result illustrates, the bijection $\aleph$ is a powerful tool for relating weights of subsets of $P$ with respect to $x$ and $z$. The general strategy is simple: relate the quantities of interest to the weights of certain subsets of $U_P$ and $D_P$, then show that these subsets are in bijection via $\aleph$. In this way, one can easily prove many previously established results about rowmotion on rectangles as well as further generalizations. We give a variety of examples of this in the rest of this section.

\subsubsection{Chains in trapezoids} \label{sec:chainsintrapezoids}

A particular case that will be important for us concerns chains in the right trapezoid $S=RT_{r,s}$.

\begin{Def}
    The \emph{left border} of the right trapezoid $RT_{r,s}$ is the set of elements $L = \{(\ell+r-1,\ell) \mid 1 \leq \ell \leq s\} \subset RT_{r,s}$.
\end{Def}
In other words, $L$ is the set of all left outer corners of $RT_{r,s}$ together with the maximum element.

For $p,q \in RT_{r,s}$, let $\mathscr C^L_{p,q}$ (resp.\ $\mathscr C^{\overline L}_{p,q}$) be the subset of $\mathscr C_{p,q}$ consisting of chains that intersect $L$ (resp.\ do not intersect $L$).

\begin{Lemma} \label{Lemma:rtchainshifting}
    Let $S=RT_{r,s}$, and let $m' \lessdot m$ and $M' \lessdot M$ be elements of $S$ such that $se(m) = ne(M') = \varnothing$.
    \begin{enumerate}[(a)]
        \item The bijection $\aleph$ restricts to a bijection from $U_S(\mathscr C^L_{m,M})$ to $D_S(\mathscr C^L_{m',M'})$.
        \item Let $x \in \RR_+^S$ and $z = \tilde \rho^{-1}(x)$. Then $w_{\mathscr C^L_{m,M}}(x) = w_{\mathscr C^L_{m',M'}}(z)$.
    \end{enumerate}
\end{Lemma}
\begin{proof}
    Choose any chain $C \in \mathscr C^L_{m,M}$, and let $c$ be the lowest element of $C \cap L$. Let $c'$ be the left inner corner covered by $c$, and let $c''$ be the element of $L$ covered by $c'$.
    By Lemma~\ref{lemma:chaincross}, $\aleph(T)$ contains a chain $C'$ that starts at $m'$, remains southwest of $C$ until it reaches $c''$, at which point it crosses over (using a forced edge) to $c'$, and then remains southeast of $C$ until reaching $M'$. It follows that $\aleph(T) \in D_S(\mathscr C^L_{m',M'})$, and the highest element of $C' \cap L$ is $c''$.
    A similar argument applied to $\aleph^{-1}$ shows that $\aleph$ gives a bijection from $U_S(\mathscr C^L_{m,M})$ to $D_S(\mathscr C^L_{m',M'})$, as desired.

    Part (b) then follows from part (a) by using Corollaries~\ref{cor:bijectionimpliesweight} and \ref{cor:measurepreserving}.
\end{proof}
See Figure~\ref{fig:rtchainshifting} for an illustration. We can alternatively visualize this by breaking the chains in $\mathscr C$ immediately before they first intersect $L$ and then shifting the lower portions to the southwest and the upper portions to the southeast. In particular, this gives an alternate proof of Lemma~\ref{Lemma:rtchainshifting}(b): for $c \in L$, define $c'$ and $c''$ as in the proof of Lemma~\ref{Lemma:rtchainshifting}. Then using Lemma~\ref{lemma:skewchainshifting},
\[w_{\mathscr C^L_{m,M}}(x) = \sum_c w_{\mathscr C^{sw}_{m,c'}}(x)w_{\mathscr C^{se}_{c,M}}(x) = \sum_c w_{\mathscr C^{ne}_{m',c''}}(z) w_{\mathscr C^{nw}_{c',M'}}(z) = w_{\mathscr C^L_{m',M'}}(z).\]
This also makes it clear that $\mathscr C^L_{m,M}$ and $\mathscr C^L_{m',M'}$ contain the same number of chains.

\begin{Ex}
    A simple example of Lemma~\ref{Lemma:rtchainshifting} can be seen in $S=RT_{3,2}$ using the arborescences in Figure~\ref{fig:upwardSpanningTree}. If $m = (1,2)$ and $M=(4,2)$, then there is a unique chain in $\mathscr C^L_{m,M}$, and the only upward arborescence containing this chain is the fourth one. Since $m' = (1,1)$ and $M' = (3,2)$, there is also a unique chain in $\mathscr C^L_{m',M'}$ (it must pass through $(3,1) \in L$), and the only downward arborescence containing this chain is again the fourth one, the image under $\aleph$ of the upward arborescence found above. Algebraically, we can check using the labelings given in Figure~\ref{fig:r32} that Lemma~\ref{Lemma:rtchainshifting}(b) holds here:
    \[z_{11}z_{21}z_{31}z_{32} = \frac{bc}{b+c} \cdot \frac{(b+c)de}{bd+be+ce} \cdot \frac{(bd+be+ce)f}{bd} \cdot g = cefg = x_{12}x_{22}x_{32}x_{42}.\]
\end{Ex}

\begin{Rem}Though we will not need it for our purposes, it is straightforward to apply this technique more generally to write down similar chain shifting lemmas for chains that touch the boundary of a general skew shape $S$, possibly in several places.
\end{Rem}

\begin{figure}
    \centering
\begin{tikzpicture}[scale = 0.35]
\begin{scope}[shift = {(0,0)}]
\draw (-3,3)--(-6,6)--(-5,7)--(-6,8)--(-5,9)--(-6,10)--(-5,11)--(-6,12)--(-5,13)--(-6,14)--(-5,15)--(-6,16)--(-5,17)--(-6,18)--(-5,19)--(-6,20)--(-5,21)--(-6,22);
\draw (-3,3)--(5,11)--(-6,22);
\draw[ultra thick] (-6,14)--(-5,13);
\node[label=left:$c$] at (-6,14) {};
\node[label={[label distance=-8pt]above right:$c'$}] at (-5,13) {};
\node[label={[label distance=-4pt]left:$c''$}] at (-6,12) {};

\filldraw [ultra thick, red, fill=red!40] (-1,5)--(-5,9)--(-4,10)--(-5,11)--(-4,12)--(-5,13)--(1,7)--cycle;
\node[label={[label distance=-6pt]below right:$m$}] at (-1,5) {};

\filldraw [ultra thick, blue, fill=blue!40] (-6,14)--(-5,15)--(-6,16)--(-5,17)--(-6,18)--(-4,20)--(-2,18)--cycle;
\node[label={[label distance=-6pt]above right:$M$}] at (-4,20) {};
\end{scope}

\begin{scope}[shift = {(17,0)}]
\draw (-3,3)--(-6,6)--(-5,7)--(-6,8)--(-5,9)--(-6,10)--(-5,11)--(-6,12)--(-5,13)--(-6,14)--(-5,15)--(-6,16)--(-5,17)--(-6,18)--(-5,19)--(-6,20)--(-5,21)--(-6,22);
\draw (-3,3)--(5,11)--(-6,22);
\draw[ultra thick] (-6,12)--(-5,13);
\node[label=left:$c$] at (-6,14) {};
\node[label={[label distance=-4pt]left:$c''$}] at (-6,12) {};
\node[label={[label distance=-10pt]below right:$c'$}] at (-5,13) {};

\begin{scope}[shift={(-1,-1)}]
\filldraw [ultra thick, red, fill=red!40] (-1,5)--(-5,9)--(-4,10)--(-5,11)--(-4,12)--(-5,13)--(1,7)--cycle;
\node[label={[label distance=-10pt]below right:$m'$}] at (-1,5) {};
\end{scope}

\begin{scope}[shift={(1,-1)}]
\filldraw [ultra thick, blue, fill=blue!40] (-6,14)--(-5,15)--(-6,16)--(-5,17)--(-6,18)--(-4,20)--(-2,18)--cycle;
\node[label={[label distance=-6pt]above right:$M'$}] at (-4,20) {};
\end{scope}
\end{scope}
\end{tikzpicture}
    \caption{Chain shifting in the right trapezoid. The total weight of all chains in the highlighted regions on the left (passing through $c'$ and $c$) with respect to $x$ equals the total weight of all chains in the highlighted regions on the right (passing through $c''$ and $c'$) with respect to $z = \tilde\rho^{-1}(x)$.}
    \label{fig:rtchainshifting}
\end{figure}
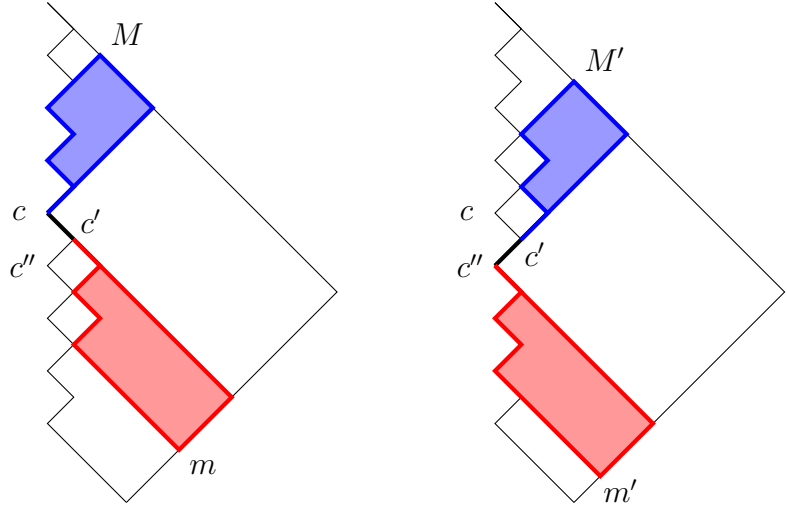

\medskip

The applications in the remainder of this section will not be needed until Section~\ref{sec:rowmotionequivariance}.

\subsubsection{Partial chains} \label{sec:partialchains}

We can also use $\aleph$ to study chains with an endpoint that does not lie on the boundary of $S$. 
The following result generalizes Lemma~\ref{lemma:skewchainshifting} when $p = M$.
\begin{Lemma} \label{lemma:partialchainshifting}
    Let $m' \lessdot m<p$ be elements of $S$ such that $sw(m) = \varnothing$.
    \begin{enumerate}[(a)]
        \item The bijection $\aleph$ restricts to a bijection from $U_S(\mathscr C^{se}_{m,p})$ to $\bigcup_{u \in se(p)} D_S(\mathscr C^{nw}_{m',u})$.
        \item Let $x \in \RR_+^S$, $y = \psi(x)$, and $z = \tilde\rho^{-1}(x)$. Then
    \[w_{\mathscr C^{se}_{m,p}}(x) = \sum_{u \in se(p)} \sum_{u' \gtrdot u} \frac{y_p}{y_{u'}} \cdot w_{\mathscr C^{nw}_{m',u}}(z).\]
    \end{enumerate}
\end{Lemma}
\begin{proof}
    Choose any $T \in U_S(\mathscr C^{se}_{m,p})$ and let $C$ be the chain from $m$ to $p$ in $T$. By Lemma~\ref{lemma:chaincross}, the chain upward from $m'$ in $\aleph(T)$ must contain some $u \in se(p)$, so $\aleph(T) \in D_S(\mathscr C^{nw}_{m',u})$. Similarly $\aleph^{-1}$ sends each $D_S(\mathscr C^{nw}_{m',u})$ into $U_S(\mathscr C^{se}_{m,p})$ since $nw(m') = \{m\}$.

    For part (b), part (a) and Corollary~\ref{cor:measurepreserving} imply $\mu_y(U_S(\mathscr C^{se}_{m,p})) = \sum_{u \in se(p)}\mu_y(D_S(\mathscr C^{nw}_{m',u}))$. By Proposition~\ref{prop:chainarborescence},
        \begin{align*}
            \mu_y(U_S(\mathscr C^{se}_{m,p})) &= \frac{y_{m'}}{y_p} \cdot w_{\mathscr C^{se}_{m,p}}(x),\\
            \mu_y(D_S(\mathscr C^{nw}_{m',u})) &= \sum_{u' \gtrdot u} \frac{y_{m'}}{y_{u'}} \cdot w_{\mathscr C^{nw}_{m',u}}(z).
        \end{align*}
        Summing the second equation over all $u \in se(p)$, equating with the first equation and rearranging gives the result.
\end{proof}

One can also apply a dual argument to instead consider chains whose maximum lies on the boundary of $S$.

\begin{Lemma} \label{lemma:partialchainshifting2}
    Let $p<M' \lessdot M$ be elements of $S$ such that $ne(M') = \varnothing$.
    \begin{enumerate}[(a)]
        \item The bijection $\aleph$ restricts to a bijection from $\bigcup_{v \in nw(p)}U_S(\mathscr C^{se}_{v,M})$ to $D_S(\mathscr C^{nw}_{p,M'})$.
        \item Let $x \in \RR_+^S$, $y = \psi(x)$, and $z = \tilde\rho^{-1}(x)$. Then
    \[w_{\mathscr C^{nw}_{p,M'}}(z) = \sum_{v \in nw(p)} \sum_{v' \lessdot v} \frac{y_{v'}}{y_{p}} \cdot w_{\mathscr C^{se}_{v,M}}(x).\]
    \end{enumerate}
\end{Lemma}
\begin{proof}
    Analogous to the proof of Lemma~\ref{lemma:partialchainshifting}. (Alternatively, apply Lemma~\ref{lemma:skewchainshifting} on the dual poset to $S$, switching the roles of $x$ and $z$.)
\end{proof}

\subsubsection{Partial chains in trapezoids} \label{sec:partialchainsintrapezoids}

We can combine the techniques of Sections~\ref{sec:chainsintrapezoids} and \ref{sec:partialchains} to obtain the following results about partial chains in trapezoids, which we will use in Section~\ref{sec:rowmotionequivariance}. As in Section~\ref{sec:chainsintrapezoids}, let $S = RT_{r,s}$ be a right trapezoid with left border $L$.

\begin{Lemma} \label{lemma:trapezoidshift}
    Let $S = RT_{r,s}$, and either let $m' \lessdot m < p$ be elements of $S$ such that $se(m)=\varnothing$, or let $n' \lessdot n < p$ such that $sw(n) = \varnothing$.
    \begin{enumerate}[(a)]
        \item The bijection $\aleph$ restricts to a bijection:
        \begin{enumerate}[(i)]
            \item from $U_S(\mathscr C^L_{m,p})$ to $\bigcup_{u \in se(p)} D_S(\mathscr C^L_{m',u})$;
            \item from $U_S(\mathscr C^{\overline L}_{m,p})$ to $\bigcup_{u \in sw(p)} D_S(\mathscr C_{m',u})$;
            \item from $U_S(\mathscr C_{n,p})$ to $\bigcup_{u \in se(p)} D_S(\mathscr C^{\overline L}_{n', u})$.
        \end{enumerate}
        \item Let $x \in \RR_+^S$, $y = \psi(x)$, and $z = \tilde\rho^{-1}(x)$. Then
        \begin{align*}
        w_{\mathscr C^{L}_{m,p}}(x) &= \sum_{u \in se(p)}\sum_{u' \gtrdot u} \frac{y_p}{y_{u'}} w_{\mathscr C^{L}_{m',u}}(z),\\
        w_{\mathscr C^{\overline{L}}_{m,p}}(x) &= \sum_{u \in sw(p)}\sum_{u' \gtrdot u} \frac{y_p}{y_{u'}} w_{\mathscr C_{m',u}}(z),\\
        w_{\mathscr C_{n,p}}(x) &= \sum_{u \in se(p)}\sum_{u' \gtrdot u} \frac{y_p}{y_{u'}} w_{\mathscr C^{\overline{L}}_{n',u}}(z).
    \end{align*}
    \end{enumerate}

\end{Lemma}
\begin{proof}
    For part (a), (i) follows by combining the arguments in Lemmas~\ref{Lemma:rtchainshifting} and \ref{lemma:partialchainshifting}, while (ii) and (iii) follow immediately from Lemma~\ref{lemma:partialchainshifting} (reflected over the vertical axis for (ii)). Part (b) then follows from part (a) as in Lemma~\ref{lemma:partialchainshifting}.
\end{proof}

We likewise have the following dual result.

\begin{Lemma} \label{lemma:trapezoidshiftdual}
    Let $S = RT_{r,s}$, and let $p < M' \lessdot M$ be elements of $S$ such that $ne(M') = \varnothing$.
    \begin{enumerate}[(a)]
        \item The bijection $\aleph$ restricts to a bijection:
        \begin{enumerate}[(i)]
            \item from $\bigcup_{v \in ne(p)} U_S(\mathscr C^L_{v,M})$ to $D_S(\mathscr C^L_{p,M'})$;
            \item from $\bigcup_{v \in nw(p)} U_S(\mathscr C_{v,M})$ to $D_S(\mathscr C^{\overline L}_{p,M'})$.
        \end{enumerate}
        \item Let $x \in \RR_+^S$, $y = \psi(x)$, and $z = \tilde\rho^{-1}(x)$. Then
        \begin{align*}
        w_{\mathscr C^{L}_{p,M'}}(z) &= \sum_{v \in ne(p)}\sum_{v' \lessdot v} \frac{y_{v'}}{y_{p}} w_{\mathscr C^{L}_{v,M}}(x),\\
        w_{\mathscr C^{\overline L}_{p,M}}(z) &= \sum_{v \in nw(p)}\sum_{v' \lessdot v} \frac{y_{v'}}{y_{p}} w_{\mathscr C_{v,M}}(x).
    \end{align*}
    \end{enumerate}
\end{Lemma}
\begin{proof}
    Analogous to the proof of Lemma~\ref{lemma:trapezoidshift}.
\end{proof}

\medskip

The remainder of this section will not be needed for the main results of this paper.

\subsubsection{Nonintersecting chains}

We can generalize Proposition~\ref{prop:chainarborescence} to apply not just to chains but also to disjoint unions of chains.

\begin{Prop} \label{prop:kchains}
    Let $P$ be a poset, $x \in \RR_+^P$, $y = \psi(x)$, and $z = \tilde\rho^{-1}(x)$. Let $C = C_1 \sqcup C_2 \sqcup \dots \sqcup C_k$ be a disjoint union of saturated chains $C_i$ with minima $a_i$ and maxima $b_i$. Then
        \begin{align*}
    w_C(x) &= \prod_{i=1}^k\left(\sum_{a' \lessdot a_i} \frac{y_{a'}}{y_{b_i}}\right)^{-1} \cdot \mu_y(U_P(C)),\\
    w_C(z) &= \prod_{i=1}^k\left(\sum_{b' \gtrdot b_i} \frac{y_{a_i}}{y_{b'}}\right)^{-1} \cdot \mu_y(D_P(C)).
    \end{align*}
\end{Prop}
\begin{proof}
    Analogous to the proof of Proposition~\ref{prop:chainarborescence}: the only difference is that the weight of the edges in $C$ is $\prod_i \frac{y_{a_i}}{y_{b_i}}$ instead of $\frac{y_a}{y_b}$, and we must add downward edges from each $a_i$, which have weight $\prod_i \sum_{a' \lessdot a_i} \frac{y_{a'}}{y_{a_i}}$ instead of just $\sum_{a' \lessdot a} \frac{y_{a'}}{y_a}$.
\end{proof}

Let us call a disjoint union $C = C_1 \sqcup \cdots \sqcup C_k$ of $k$ (nonintersecting) saturated chains a \emph{$k$-chain}. If $\mathscr C$ is a collection of $k$-chains, then we can define $w_{\mathscr C}(x)$, $U_P(\mathscr C)$, and $D_P(\mathscr C)$ as before. We can then prove chain shifting results using $\aleph$ as before to relate these sets.

As one example, we can easily prove a $k$-chain version of Lemma~\ref{lemma:skewchainshifting}. (While this result can also be proved from Lemma~\ref{lemma:skewchainshifting} by using the Lindstr\"om-Gessel-Viennot Lemma, the proof that we give here has the advantage that it does not require subtraction.)

\begin{Lemma}
    For $1 \leq i \leq k$, let $m'_i \lessdot m_i$ and $M'_i \lessdot M_i$ be elements of $S$ such that $sw(m_i) = ne(M'_i) = \varnothing$. Let $\mathscr C$ be the set of $k$-chains $C= C_1 \sqcup \cdots \sqcup C_k$ such that $C_i \in \mathscr C^{se}_{m_i,M_i}$, and let $\mathscr C'$ be the set of $k$-chains $C' = C'_1 \sqcup \cdots \sqcup C'_k$ such that $C'_i \in \mathscr C^{nw}_{m'_i,M'_i}$.
    \begin{enumerate}[(a)]
    \item The bijection $\aleph$ restricts to a bijection from $U_S(\mathscr C)$ to $D_S(\mathscr C')$.
    \item Let $x \in \RR^S_+$ and $z = \tilde\rho^{-1}(x)$. Then $w_{\mathscr C}(x) = w_{\mathscr C'}(z)$.
    \end{enumerate}
\end{Lemma}
\begin{proof}
    We may assume without loss of generality that $m_1 < m_2 < \cdots < m_k$ (along the leftmost chain in $S$). Suppose $T \in U_S(\mathscr C)$ contains the $k$-chain $C = C_1 \sqcup \cdots \sqcup C_k$. By Lemma~\ref{lemma:skewchainshifting}, $\aleph(T)$ must contain a chain $C'_i$ from $m_i'$ to $M_i'$ for all $i$. These chains must be disjoint: $C'_i$ does not intersect $C_i$ by construction, and $C'_{i+1} \in \mathscr C^{nw}_{m'_i,M'_i}$ cannot intersect $C_i$ since it contains no forced edge southwest from a left inward corner. Thus $C'_i$ and $C'_{i+1}$ are separated by $C_i$ and are therefore disjoint. It follows that $\aleph(T) \in D_S(\mathscr C')$. A similar argument shows that $\aleph^{-1}$ sends $D_S(\mathscr C')$ to $U_S(\mathscr C)$. Part (b) then follows from Corollaries~\ref{cor:measurepreserving} and \ref{cor:bijectionimpliesweight}.
\end{proof}

As another example, we show how one can apply this technique to derive a special case of the iterated rowmotion formula for rectangles from \cite{musikerroby}. 

\begin{Prop}
    Let $R = R_{r,s}$, and choose $p = (i,j) \in R$ with $1 \leq i < r$ and $1 \leq j < s$. Let $\mathscr C$ be the collection of $2$-chains $C_1 \sqcup C_2$, where $C_1$ is a chain from $(2,1)$ to $(i+1,j)$, and $C_2$ is a chain from $(1,2)$ to $(i,j+1)$.
    \begin{enumerate}[(a)]
        \item The bijection $\aleph$ restricts to a bijection from $U_R(\mathscr C)$ to $D_R(\mathscr C_{(1,1),p})$.
        \item Let $x \in \RR^R_+$ and $y = \psi(x)$. Then \[\rho^{-1}(y)_p = x_{11}x_{i+1,j+1} \frac{w_{\mathscr C}(x)}{w_{\mathscr C'}(x)},\]
        where $\mathscr C'$ is the set of chains from $(1,1)$ to $(i+1,j+1)$.
    \end{enumerate}
\end{Prop}
\begin{proof}
    Let $T \in U_R(\mathscr C)$ contain the $2$-chain $C_1 \sqcup C_2$ as in the statement of the proposition. By Lemma~\ref{lemma:chaincross}, the chain upward from $(1,1)$ in $\aleph(T)$ must lie southeast of $C_1$ and southwest of $C_2$, so it must pass through $p$. Thus $\aleph(T) \in D_R(\mathscr C_{(1,1),p})$. Similarly, if $T' \in D_R(\mathscr C_{(1,1),p})$ contains the chain $C'$ from $(1,1)$ to $p$, then the chains downward from $(i+1,j)$ and $(i,j+1)$ in $\aleph^{-1}(T)$ must lie northwest and northeast of $C'$ and so they must pass through $(2,1)$ and $(1,2)$, respectively, and be disjoint. Part (a) follows.

    For part (b), summing Proposition~\ref{prop:kchains} over $C \in \mathscr C$ and $C \in \mathscr C_{(1,1),p}$ gives
    \begin{align*}
    \mu_y(U_R(\mathscr C)) &= \frac{y_{11}}{y_{i+1,j}}\cdot \frac{y_{11}}{y_{i,j+1}} w_{\mathscr C}(x),\\
    \mu_y(D_R(\mathscr C_{(1,1),p}))&= \left(\frac{y_{11}}{y_{i+1,j}} + \frac{y_{11}}{y_{i,j+1}}\right) w_{\mathscr C_{(1,1),p}}(z),
    \end{align*}
    where $z = \tilde\rho^{-1}(x)$.
    By part (a) and Corollary~\ref{cor:measurepreserving}, these two quantities are equal. Equating them and rearranging gives
    \[w_{\mathscr C_{(1,1),p}}(z) = \frac{y_{11}}{y_{i+1,j}+y_{i,j+1}}w_{\mathscr C}(x).\]
    But the left hand side equals $\psi(z)_p = \rho^{-1}(y)_p$, and the fraction on the right hand side is \[\frac{y_{11}}{y_{i+1,j}+y_{i,j+1}} = \frac{x_{11}}{y_{i+1,j+1}/x_{i+1,j+1}} = \frac{x_{11}x_{i+1,j+1}}{w_{\mathscr C'}(x)}.\] Substituting these into the previous equation gives the result.
\end{proof}

Since the expression for $\rho^{-1}(y)_p$ obtained is written in terms of weights of chains with respect to $x$, it is not difficult to iterate this procedure to obtain the other cases of the iterated rowmotion formula from \cite{musikerroby}. (Since this is not relevant for our current work, we leave the details to the reader.)

\section{A map between the rectangle and trapezoid}
\label{section:equivariantMap}

In this section, we use the chain shifting lemmas (Lemmas~\ref{lemma:skewchainshifting} and \ref{Lemma:rtchainshifting}) to define a birational map between labelings of the rectangle $R_{r,s}$ and trapezoid $T_{r,s}$. In the tropical setting, this map will become a continuous, piecewise-linear, volume-preserving map between the chain polytopes of these two posets, which will in turn give a bijection between the plane partitions of $R_{r,s}$ and $T_{r,s}$ of height $\ell$. We will express this map as a composition of rowmotion-like maps between labelings of certain intermediate posets.

\subsection{Intermediate posets and maps}
The intermediate posets that we consider will all be induced subposets of the right trapezoid $RT_{r,s}$.

\begin{Def}
\label{Def:intermediatePosets}
    Let $k \leq s \leq r$ be positive integers. The $k$th \emph{intermediate poset} $I_k = I_{r,s,k}$ is the induced subposet of $RT_{r,s}$ on $T_{r,k} \cup [(k,k+1), (r+k-1, s)]$.
\end{Def}

See Figure~\ref{fig:intermediatePosets} for examples. 
Note that the leftmost minimal element of $I_k$ is $(k,1)$. 
One can easily verify that $I_1 = R_{r,s}$, $I_s = T_{r,s}$, and $|I_k| = rs$ for all $k$.

\begin{figure}
    \centering
\begin{tikzpicture}[scale = 0.5]
\draw [dashed] (0,0)--(3,3);
\draw [dashed] (-1,1)--(1,3);
\draw (1,3)--(2,4);
\draw [dashed] (-2,2)--(-1,3);
\draw (-1,3)--(1,5);
\draw (-3,3)--(0,6);
\draw (-3,5)--(-1,7);
\draw (-3,7)--(-2,8);

\draw [dashed] (0,0)--(-3,3);
\draw [dashed] (1,1)--(-1,3);
\draw (-1,3)--(-3,5);
\draw [dashed] (2,2)--(1,3);
\draw (1,3)--(-3,7);
\draw (3,3)--(-3,9);

\node[wW] at (0,0) {};
\node[wW] at (1,1) {};
\node[wW] at (2,2) {};
\node[wW] at (-1,1) {};
\node[wW] at (0,2) {};
\node[wW] at (-2,2) {};

\node[wB] at (3,3) {};
\node[wB] at (1,3) {};
\node[wB] at (2,4) {};
\node[wB] at (-1,3) {};
\node[wB] at (0,4) {};
\node[wB] at (1,5) {};
\node[wB] at (-3,3) {};
\node[wB] at (-2,4) {};
\node[wB] at (-1,5) {};
\node[wB] at (0,6) {};
\node[wB] at (-3,5) {};
\node[wB] at (-2,6) {};
\node[wB] at (-1,7) {};
\node[wB] at (-3,7) {};
\node[wB] at (-2,8) {};
\node[wB] at (-3,9) {};

\begin{scope}[shift={(7.5,0)}]
\draw [dashed] (0,0)--(3,3);
\draw [dashed] (-1,1)--(0,2);
\draw (0,2)--(1,3);
\draw [dashed] (1,3)--(2,4);
\draw (-2,2)--(1,5);
\draw (-3,3)--(0,6);
\draw (-3,5)--(-1,7);
\draw (-3,7)--(-2,8);

\draw [dashed] (0,0)--(-2,2);
\draw (-2,2)--(-3,3);
\draw [dashed] (1,1)--(0,2);
\draw (0,2)--(-3,5);
\draw (2,2)--(-3,7);
\draw [dashed] (3,3)--(1,5);
\draw (1,5)--(-2,8);
\draw [dashed] (-2,8)--(-3,9);

\node[wW] at (0,0) {};
\node[wW] at (1,1) {};
\node[wW] at (3,3) {};
\node[wW] at (-1,1) {};
\node[wW] at (2,4) {};
\node[wW] at (-3,9) {};

\node[wB] at (2,2) {};
\node[wB] at (0,2) {};
\node[wB] at (1,3) {};
\node[wB] at (-2,2) {};
\node[wB] at (-1,3) {};
\node[wB] at (0,4) {};
\node[wB] at (1,5) {};
\node[wB] at (-3,3) {};
\node[wB] at (-2,4) {};
\node[wB] at (-1,5) {};
\node[wB] at (0,6) {};
\node[wB] at (-3,5) {};
\node[wB] at (-2,6) {};
\node[wB] at (-1,7) {};
\node[wB] at (-3,7) {};
\node[wB] at (-2,8) {};
\end{scope}

\begin{scope}[shift={(15,0)}]
\draw [dashed] (0,0)--(3,3);
\draw (-1,1)--(2,4);
\draw (-2,2)--(1,5);
\draw (-3,3)--(0,6);
\draw (-3,5)--(-1,7);
\draw [dashed] (-3,7)--(-2,8);

\draw [dashed] (0,0)--(-1,1);
\draw (-1,1)--(-3,3);
\draw (1,1)--(-3,5);
\draw [dashed] (2,2)--(1,3);
\draw (1,3)--(-2,6);
\draw [dashed] (-2,6)--(-3,7);
\draw [dashed] (3,3)--(2,4);
\draw (2,4)--(-1,7);
\draw [dashed] (-1,7)--(-3,9);

\node[wW] at (0,0) {};
\node[wW] at (2,2) {};
\node[wW] at (3,3) {};
\node[wW] at (-3,7) {};
\node[wW] at (-2,8) {};
\node[wW] at (-3,9) {};

\node[wB] at (1,1) {};
\node[wB] at (-1,1) {};
\node[wB] at (0,2) {};
\node[wB] at (1,3) {};
\node[wB] at (2,4) {};
\node[wB] at (-2,2) {};
\node[wB] at (-1,3) {};
\node[wB] at (0,4) {};
\node[wB] at (1,5) {};
\node[wB] at (-3,3) {};
\node[wB] at (-2,4) {};
\node[wB] at (-1,5) {};
\node[wB] at (0,6) {};
\node[wB] at (-3,5) {};
\node[wB] at (-2,6) {};
\node[wB] at (-1,7) {};
\end{scope}

\begin{scope}[shift={(22.5,0)}]
\draw (0,0)--(3,3);
\draw (-1,1)--(2,4);
\draw (-2,2)--(1,5);
\draw (-3,3)--(0,6);
\draw [dashed] (-3,5)--(-1,7);
\draw [dashed] (-3,7)--(-2,8);

\draw (0,0)--(-3,3);
\draw (1,1)--(-2,4);
\draw [dashed] (-2,4)--(-3,5);
\draw (2,2)--(-1,5);
\draw [dashed] (-1,5)--(-3,7);
\draw (3,3)--(0,6);
\draw [dashed] (0,6)--(-3,9);

\node[wW] at (-3,5) {};
\node[wW] at (-2,6) {};
\node[wW] at (-1,7) {};
\node[wW] at (-3,7) {};
\node[wW] at (-2,8) {};
\node[wW] at (-3,9) {};

\node[wB] at (0,0) {};
\node[wB] at (1,1) {};
\node[wB] at (2,2) {};
\node[wB] at (3,3) {};
\node[wB] at (-1,1) {};
\node[wB] at (0,2) {};
\node[wB] at (1,3) {};
\node[wB] at (2,4) {};
\node[wB] at (-2,2) {};
\node[wB] at (-1,3) {};
\node[wB] at (0,4) {};
\node[wB] at (1,5) {};
\node[wB] at (-3,3) {};
\node[wB] at (-2,4) {};
\node[wB] at (-1,5) {};
\node[wB] at (0,6) {};
\end{scope}
\end{tikzpicture}
    \caption{The four intermediate posets $I_4, \dots, I_1$ lying inside of $RT_{4,4}$.}
    \label{fig:intermediatePosets}
\end{figure}
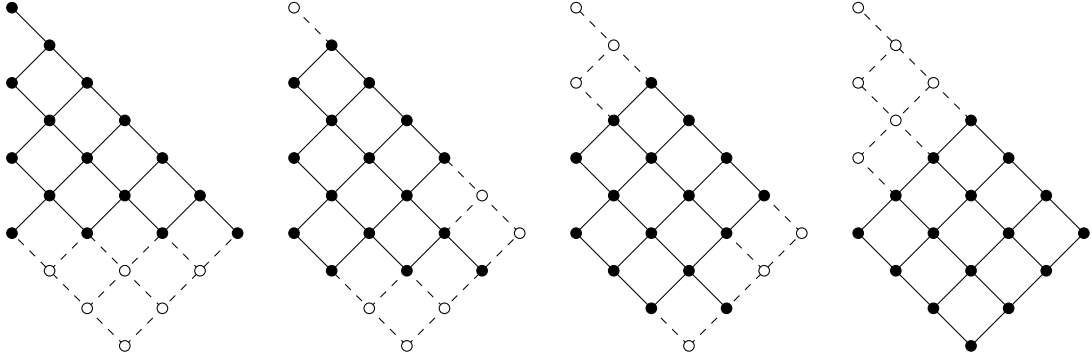

We now define maps $\themap_k \colon \RR_+^{I_{k+1}} \to \RR_+^{I_{k}}$ as follows. Consider the interval 
\[M_k = [(k,1),(r+k,k+1)] \subset RT_{r,s},\]
which is isomorphic to the smaller right trapezoid $RT_{r-k+1,k+1}$.
For any $x \in \RR_+^{I_{k+1}}$, let $\bar x \in \RR_+^{M_k}$ be the labeling obtained by restricting $x$ to $M_k \setminus \{(k,1)\} \subset I_{k+1}$ and setting $\bar x_{k,1}$ to be an arbitrary number $a \in \RR_+$ (say, $1$). Finally, let $\tilde\rho_k \colon \RR_+^{M_k} \to \RR_+^{M_k}$ be the antichain rowmotion map on $M_k$. (Recall that one can compute $\tilde\rho_k^{-1}$ with relative ease using Corollary~\ref{cor:edgeweight}.) Then we define $\themap_k(x) \in \RR_+^{I_{k}}$ by
\[\themap_k(x)_{ij} = \begin{cases} \tilde\rho_k^{-1}(\bar x)_{ij}&\text{if }(i,j) \in I_k \cap M_k = M_k \setminus \{(r+k,k+1)\}, \\ 
x_{i+1,j}&\text{if }(i,j) \in I_{k} \setminus M_k\text{ and }j > k+1, \\
x_{i,j+1}&\text{if }(i,j) \in I_{k} \setminus M_k\text{ and }i < k.\end{cases}\]

We say that the \emph{labels below $M_k$} (with $i < k$) are shifted southwest and the \emph{labels above $M_k$} (with $j > k+1$) are shifted southeast. By Proposition~\ref{prop:arbitrarya}, $\themap_k(x)$ does not depend on the choice of $a$. Similarly the inverse map $\themap_k^{-1}$ is also well-defined (apply $\tilde\rho_k$ on $M_k$ with an arbitrary label at $(r+k,k+1)$ and shift the labels outside $M_k$ appropriately). 

\begin{Ex}
Let $r=s=4$ and $k=2$, and consider $\themap_2\colon \mathbb{R}^{I_3} \to \mathbb{R}^{I_2}$ as shown in Figure~\ref{fig:psiMapExample}. Here $M_2 = [(2,1),(6,3)]$ is shown in red. Outside of $M_2$, the coordinates of the labeling shift parallel to the sides. Inside of $M_2$ we apply antichain rowmotion $\tilde \rho_2^{-1}$ (ignoring the labels outside of $M_2$).

See also Figure~\ref{fig:schematic} below for a schematic version of this diagram.
\end{Ex}

\begin{figure}
    \centering
\begin{tikzpicture}[scale = 0.55]

\begin{scope}[shift={(0,0)}]
\draw [dashed] (0,0)--(3,3);
\draw [dashed, red, very thick] (-1,1)--(0,2);
\draw [red, very thick] (0,2)--(1,3);
\draw [dashed] (1,3)--(2,4);
\draw [red, very thick] (-2,2)--(0,4);
\draw (0,4)--(1,5);
\draw [red, very thick] (-3,3)--(-1,5);
\draw (-1,5)--(0,6);
\draw [red, very thick] (-3,5)--(-2,6);
\draw (-2,6)--(-1,7);
\draw (-3,7)--(-2,8);

\draw [dashed] (0,0)--(-1,1);
\draw [dashed, red, very thick] (-1,1)--(-2,2);
\draw [red, very thick] (-2,2)--(-3,3);
\draw [dashed] (1,1)--(0,2);
\draw [red, very thick] (0,2)--(-3,5);
\draw (2,2)--(1,3);
\draw [red, very thick] (1,3)--(-3,7);
\draw [dashed] (3,3)--(1,5);
\draw (1,5)--(-2,8);
\draw [dashed] (-2,8)--(-3,9);

\node[wW] at (0,0) {};
\node[wW] at (1,1) {};
\node[wW] at (3,3) {};
\node[wW] at (2,4) {};
\node[wW] at (-3,9) {};

\node[wB] at (2,2) {};
\node[wB] at (1,5) {};
\node[wB] at (0,6) {};
\node[wB] at (-1,7) {};
\node[wB] at (-2,8) {};

\node[wR] at (-1,1) {};
\node[wR] at (0,2) {};
\node[wR] at (1,3) {};
\node[wR] at (-2,2) {};
\node[wR] at (-1,3) {};
\node[wR] at (0,4) {};
\node[wR] at (-3,3) {};
\node[wR] at (-2,4) {};
\node[wR] at (-1,5) {};
\node[wR] at (-3,5) {};
\node[wR] at (-2,6) {};
\node[wR] at (-3,7) {};

\node at (2.8,2) {$x_{13}$};
\node at (1.8,5) {$x_{34}$};
\node at (0.8,6) {$x_{44}$};
\node at (-0.2,7) {$x_{54}$};
\node at (-1.2,8) {$x_{64}$};
\end{scope}

\draw [-Stealth, ultra thick] (4.5,4.5)--(6.5,4.5);

\begin{scope}[shift={(11,0)}]
\draw [dashed] (0,0)--(3,3);
\draw [red, very thick] (-1,1)--(1,3);
\draw (1,3)--(2,4);
\draw [red, very thick] (-2,2)--(0,4);
\draw (0,4)--(1,5);
\draw [red, very thick] (-3,3)--(-1,5);
\draw (-1,5)--(0,6);
\draw [red, very thick] (-3,5)--(-2,6);
\draw (-2,6)--(-1,7);
\draw [dashed] (-3,7)--(-2,8);

\draw [dashed] (0,0)--(-1,1);
\draw [red, very thick] (-1,1)--(-3,3);
\draw (1,1)--(0,2);
\draw [red, very thick] (0,2)--(-3,5);
\draw [dashed] (2,2)--(1,3);
\draw [red, very thick] (1,3)--(-2,6);
\draw [dashed, red, very thick] (-2,6)--(-3,7);
\draw [dashed] (3,3)--(2,4);
\draw (2,4)--(-1,7);
\draw [dashed] (-1,7)--(-3,9);

\node[wW] at (0,0) {};
\node[wW] at (2,2) {};
\node[wW] at (3,3) {};
\node[wW] at (-3,7) {};
\node[wW] at (-2,8) {};
\node[wW] at (-3,9) {};

\node[wB] at (1,1) {};
\node[wB] at (2,4) {};
\node[wB] at (1,5) {};
\node[wB] at (0,6) {};
\node[wB] at (-1,7) {};

\node[wR] at (-1,1) {};
\node[wR] at (0,2) {};
\node[wR] at (1,3) {};
\node[wR] at (-2,2) {};
\node[wR] at (-1,3) {};
\node[wR] at (0,4) {};
\node[wR] at (-3,3) {};
\node[wR] at (-2,4) {};
\node[wR] at (-1,5) {};
\node[wR] at (-3,5) {};
\node[wR] at (-2,6) {};
\node[wR] at (-3,7) {};

\node at (1.8,1) {$x_{13}$};
\node at (2.8,4) {$x_{34}$};
\node at (1.8,5) {$x_{44}$};
\node at (0.8,6) {$x_{54}$};
\node at (-0.2,7) {$x_{64}$};
\end{scope}
\end{tikzpicture}
    \caption{A right trapezoid poset $M_2$ (red) inside of $RT_{4,4}$. When applying $\themap_2$, $\tilde\rho^{-1}$ is applied inside $M_2$ (ignoring all other coordinates) and then coordinates outside of $M_2$ shift parallel to the sides of $M_2$.}
    \label{fig:psiMapExample}
\end{figure}
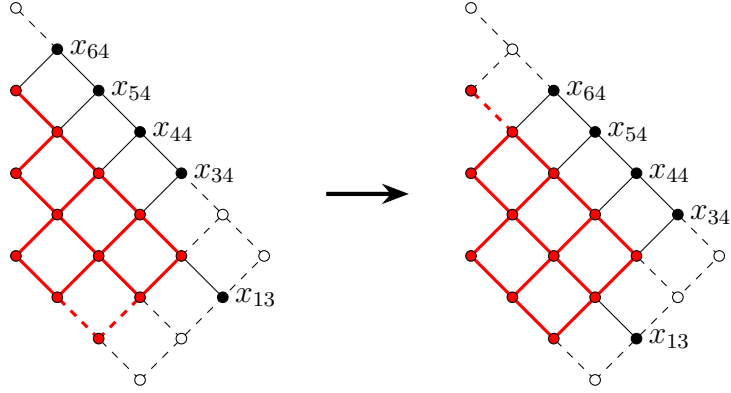

By composing the maps $\themap_k$, we arrive at a birational map \[\themap = \themap_{1} \circ \themap_{2} \circ \dots \circ \themap_{s-1}\]
from $\RR_+^{T_{r,s}}$ to $\RR_+^{R_{r,s}}$.

\begin{Ex}
    Figure~\ref{fig:zeta3} shows the result of applying $\themap = \themap_1 \circ \themap_2$ to a labeling $x \in \RR_+^T$ when $T = T_{3,3}$. Note that $x_{13} = c = \themap_2(x)_{12}$ as this label lies below $M_2$. Similarly $\themap_2(x)_{j+1,3} = \themap(x)_{j3}$ for $j=1,2,3$ as these labels lie above $M_1$.
\end{Ex}

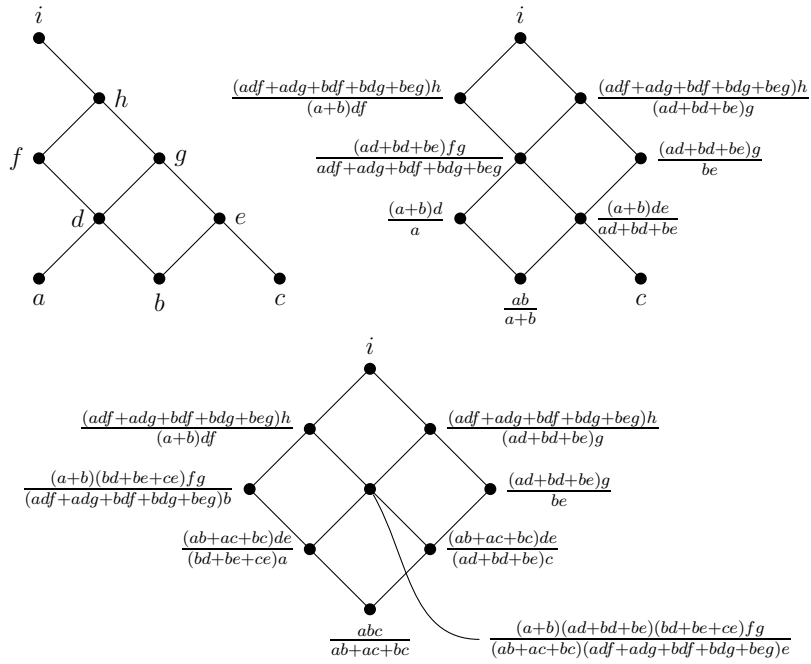
\begin{figure}
    \begin{tikzpicture}[scale=.8, transform shape]
    \begin{scope}
        \draw (2,0)--(3,1) (0,0)--(2,2) (0,2)--(1,3) (2,0)--(0,2) (4,0)--(0,4);
        \node[v,label=below:$a$] at (0,0){};
        \node[v,label=below:$b$] at (2,0){};
        \node[v,label=below:$c$] at (4,0){};
        \node[v,label=left:$d$] at (1,1){};
        \node[v,label=right:$e$] at (3,1){};
        \node[v,label=left:$f$] at (0,2){};
        \node[v,label=right:$g$] at (2,2){};
        \node[v,label=right:$h$] at (1,3){};
        \node[v,label=above:$i$] at (0,4){};
    \end{scope}
    \begin{scope}[shift={(7,1)}]
        \draw (1,-1)--(3,1) (0,0)--(2,2) (0,2)--(1,3) (1,-1)--(0,0) (3,-1)--(0,2) (3,1)--(1,3);
        \node[v,label=below:$\frac{ab}{a+b}$] at (1,-1){};
        \node[v,label=below:$c$] at (3,-1){};
        \node[v,label=left:$\frac{(a+b)d}{a}$] at (0,0){};
        \node[v,label=right:$\frac{(a+b)de}{ad+bd+be}$] at (2,0){};
        \node[v,label=left:$\frac{(ad+bd+be)fg}{adf+adg+bdf+bdg+beg}$] at (1,1){};
        \node[v,label=right:$\frac{(ad+bd+be)g}{be}$] at (3,1){};
        \node[v,label=left:$\frac{(adf+adg+bdf+bdg+beg)h}{(a+b)df}$] at (0,2){};
        \node[v,label=right:$\frac{(adf+adg+bdf+bdg+beg)h}{(ad+bd+be)g}$] at (2,2){};
        \node[v,label=above:$i$] at (1,3){};
    \end{scope}
    \begin{scope}[shift={(3.5,-3.5)}]
        \draw (2,-2)--(4,0) (1,-1)--(3,1) (0,0)--(2,2) (2,-2)--(0,0) (3,-1)--(1,1) (4,0)--(2,2);
        \node[v,label=below:$\frac{abc}{ab+ac+bc}$] at (2,-2){};
        \node[v,label=left:$\frac{(ab+ac+bc)de}{(bd+be+ce)a}$] at (1,-1){};
        \node[v,label=right:$\frac{(ab+ac+bc)de}{(ad+bd+be)c}$] at (3,-1){};
        \node[v,label=left:$\frac{(a+b)(bd+be+ce)fg}{(adf+adg+bdf+bdg+beg)b}$] at (0,0){};
        \node[v,label=left:$\frac{(adf+adg+bdf+bdg+beg)h}{(a+b)df}$] at (1,1){};
        \node[v,label=right:$\frac{(ad+bd+be)g}{be}$] at (4,0){};
        \node[v] (x) at (2,0){};
        \node[v,label=right:$\frac{(adf+adg+bdf+bdg+beg)h}{(ad+bd+be)g}$] at (3,1){};
        \node[v,label=above:$i$] at (2,2){};
        \node (l) at (6.5,-2.5){$\frac{(a+b)(ad+bd+be)(bd+be+ce)fg}{(ab+ac+bc)(adf+adg+bdf+bdg+beg)e}$};
        \draw (x) to [out = -60, in=180] (l);
    \end{scope}
    \end{tikzpicture}
    \caption{Applying $\themap_2$ and then $\themap_1$ to a labeling $x$ of $T_{3,3}$, resulting in the labeling $\themap(x)$ of $R_{3,3}$.}
    \label{fig:zeta3}
\end{figure}

The key property of $\themap$ that we will need to prove is that $\themap$ preserves the total weight of all maximal chains, which we will derive from the chain shifting lemma. However, this property does not hold for the intermediate maps $\themap_k$, so we will need to restrict to a certain special class of chains in the intermediate posets called polygonal chains.

\subsection{Polygonal chains}
We define a special collection of maximal chains within each intermediate poset which are in bijection with chains in the trapezoid and rectangle.

\begin{Def}
    A maximal chain $C \subset I_k \subset RT_{r,s}$ is \emph{polygonal} if $C$ intersects $L$ (the left border of $RT_{r,s}$) or if $(k,1) \in C$.
\end{Def}

Note that all chains in the trapezoid and rectangle are polygonal. Indeed, in the trapezoid $I_s$, all maximal chains contain the maximum element, which lies in $L$. Similarly, in the rectangle $I_1$, all maximal chains contain the minimum element $(1,1)$.

\begin{Ex}
The polygonal chains in $I_2 \subseteq RT_{3,3}$ are shown in Figure~\ref{fig:polygonalChains}. Note that $I_2$ has two other maximal chains that are not polygonal, as they start at $(1,2)$ and do not intersect $L$.
\end{Ex}

\begin{figure}
    \centering
\begin{tikzpicture}[scale = 0.65]

\foreach \x in {0,4,8,12,16,20}{
\draw (-1+\x,1)--(-2+\x,2);
\draw (1+\x,1)--(-2+\x,4);
\draw (1+\x,3)--(-1+\x,5);

\draw (-1+\x,1)--(1+\x,3);
\draw (-2+\x,2)--(0+\x,4);
\draw (-2+\x,4)--(-1+\x,5);

\draw [fill=black] (\x+1,1) circle [radius=0.13cm];

\draw [fill=black] (\x-1,1) circle [radius=0.13cm];
\draw [fill=black] (\x,2) circle [radius=0.13cm];
\draw [fill=black] (\x+1,3) circle [radius=0.13cm];

\draw [fill=black] (\x-2,2) circle [radius=0.13cm];
\draw [fill=black] (\x-1,3) circle [radius=0.13cm];
\draw [fill=black] (\x,4) circle [radius=0.13cm];

\draw [fill=black] (\x-2,4) circle [radius=0.13cm];
\draw [fill=black] (\x-1,5) circle [radius=0.13cm];
}

\draw [red, very thick] (-1,1)--(-2,2)--(-1,3)--(-2,4)--(-1,5);

\draw [fill=red] (-1,1) circle [radius=0.13cm];
\draw [fill=red] (-2,2) circle [radius=0.13cm];
\draw [fill=red] (-1,3) circle [radius=0.13cm];
\draw [fill=red] (-2,4) circle [radius=0.13cm];
\draw [fill=red] (-1,5) circle [radius=0.13cm];

\begin{scope}[shift={(4,0)}]
\draw [red, very thick] (-1,1)--(-2,2)--(-1,3)--(0,4)--(-1,5);

\draw [fill=red] (-1,1) circle [radius=0.13cm];
\draw [fill=red] (-2,2) circle [radius=0.13cm];
\draw [fill=red] (-1,3) circle [radius=0.13cm];
\draw [fill=red] (0,4) circle [radius=0.13cm];
\draw [fill=red] (-1,5) circle [radius=0.13cm];
\end{scope}

\begin{scope}[shift={(8,0)}]
\draw [red, very thick] (-1,1)--(0,2)--(-1,3)--(-2,4)--(-1,5);

\draw [fill=red] (-1,1) circle [radius=0.13cm];
\draw [fill=red] (0,2) circle [radius=0.13cm];
\draw [fill=red] (-1,3) circle [radius=0.13cm];
\draw [fill=red] (-2,4) circle [radius=0.13cm];
\draw [fill=red] (-1,5) circle [radius=0.13cm];
\end{scope}

\begin{scope}[shift={(12,0)}]
\draw [red, very thick] (-1,1)--(0,2)--(-1,3)--(-0,4)--(-1,5);

\draw [fill=red] (-1,1) circle [radius=0.13cm];
\draw [fill=red] (0,2) circle [radius=0.13cm];
\draw [fill=red] (-1,3) circle [radius=0.13cm];
\draw [fill=red] (0,4) circle [radius=0.13cm];
\draw [fill=red] (-1,5) circle [radius=0.13cm];
\end{scope}

\begin{scope}[shift={(16,0)}]
\draw [red, very thick] (-1,1)--(0,2)--(1,3)--(-0,4)--(-1,5);

\draw [fill=red] (-1,1) circle [radius=0.13cm];
\draw [fill=red] (0,2) circle [radius=0.13cm];
\draw [fill=red] (1,3) circle [radius=0.13cm];
\draw [fill=red] (0,4) circle [radius=0.13cm];
\draw [fill=red] (-1,5) circle [radius=0.13cm];
\end{scope}

\begin{scope}[shift={(20,0)}]
\draw [red, very thick] (1,1)--(0,2)--(-1,3)--(-2,4)--(-1,5);

\draw [fill=red] (1,1) circle [radius=0.13cm];
\draw [fill=red] (0,2) circle [radius=0.13cm];
\draw [fill=red] (-1,3) circle [radius=0.13cm];
\draw [fill=red] (-2,4) circle [radius=0.13cm];
\draw [fill=red] (-1,5) circle [radius=0.13cm];
\end{scope}

\end{tikzpicture}
    \caption{The six polygonal chains in $I_2 \subseteq RT_{3,3}$.}
    \label{fig:polygonalChains}
\end{figure}
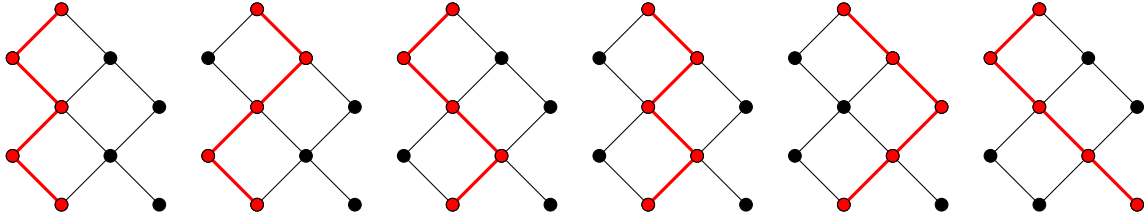

The following proposition relates the weights of polygonal chains under the maps $\themap_k$. Since the only complicated part of $\themap_k$ occurs inside $M_k$, it will follow directly from the chain shifting results for skew shapes and right trapezoids (Lemmas~\ref{lemma:skewchainshifting} and \ref{Lemma:rtchainshifting}).

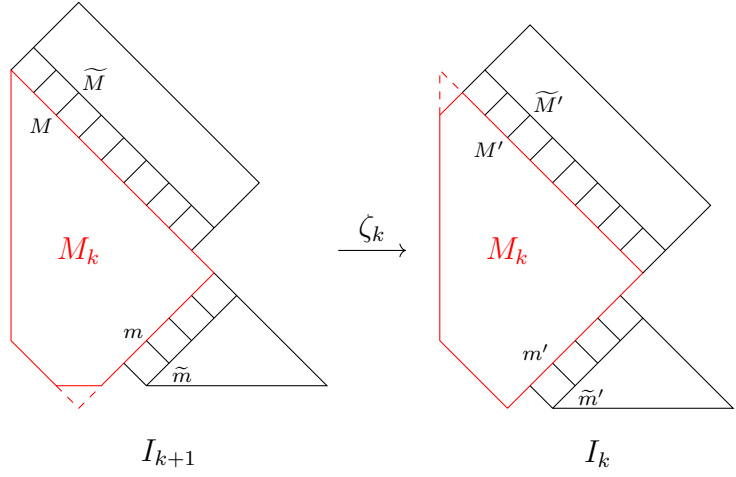
\begin{figure}
    \centering
    \begin{tikzpicture}[scale=.3]
    \begin{scope}
        \draw[red] (-1,1)--(-3,3)--(-3,15)--(6,6)--(1,1)--(-1,1);
        \draw[red,dashed] (-1,1)--(0,0)--(1,1);
        \node[red] at (0,7) {$M_k$};
        \draw (3,1)--(7,5)--(11,1)--(3,1); 
        \draw (6,8)--(-2,16)--(0,18)--(8,10)--(6,8); 
        \draw (3,1)--(2,2) (4,2)--(3,3) (5,3)--(4,4) (6,4)--(5,5) (7,5)--(6,6);
        \foreach \x in {1, ..., 9}
            \draw (6-\x,6+\x)--(7-\x,7+\x);
        \node[label={[label distance=-10pt]above left:$\scriptstyle m$}] at (3,3){};
        \node[label={[label distance=-10pt]below right:$\scriptstyle{\widetilde m}$}] at (4,2){};
        \node[label={[label distance=-10pt]below left:$\scriptstyle M$}] at (-1,13){};
        \node[label={[label distance=-10pt]above right:$\scriptstyle {\widetilde M}$}] at (0,14){};
        \node at (4,-2){$I_{k+1}$};
    \end{scope}
    \draw[->] (11.5,7)--(14.5,7);
    \node at (13,8){$\zeta_k$};
    \begin{scope}[shift={(19,0)}]
        \draw[red] (0,0)--(-3,3)--(-3,13)--(-2,14)--(6,6)--(0,0);
        \draw[red,dashed] (-3,13)--(-3,15)--(-2,14);
        \node[red] at (0,7) {$M_k$};
        \draw (2,0)--(6,4)--(10,0)--(2,0); 
        \draw (7,7)--(-1,15)--(1,17)--(9,9)--(7,7); 
        \draw (2,0)--(1,1) (3,1)--(2,2) (4,2)--(3,3) (5,3)--(4,4) (6,4)--(5,5);
        \foreach \x in {0, ..., 8}
            \draw (6-\x,6+\x)--(7-\x,7+\x);
        \node[label={[label distance=-10pt]above left:$\scriptstyle {m'}$}] at (2,2){};
        \node[label={[label distance=-10pt]below right:$\scriptstyle{\widetilde m'}$}] at (3,1){};
        \node[label={[label distance=-10pt]below left:$\scriptstyle {M'}$}] at (0,12){};
        \node[label={[label distance=-10pt]above right:$\scriptstyle {\widetilde M'}$}] at (1,13){};
        \node at (4,-2){$I_{k}$};
    \end{scope}
    \end{tikzpicture}
    \caption{Schematic of the map $\themap_k$. Polygonal chains in $I_{k+1}$ that intersect $M_k$ in a chain from $m$ to $M$ get shifted to polygonal chains in $I_k$ that intersect $M_k$ in a chain from $m'$ to $M'$.}
    \label{fig:schematic}
\end{figure}

\begin{Prop} \label{prop:intermediate}
    Let $\mathscr P_{k}$ be the collection of polygonal chains in $I_k$. For $x \in \RR_+^{I_{k+1}}$, let $z = \themap_k(x)$. Then $w_{\mathscr P_{k+1}}(x) = w_{\mathscr P_{k}}(z)$.
\end{Prop}
\begin{proof}
    Let $\mathscr P_{k+1} = \mathscr P_{k+1}^{(1)} \sqcup \mathscr P_{k+1}^{(2)}$, where $\mathscr P_{k+1}^{(1)}$ consists of the chains that contain $(k+1,1)$ and $\mathscr P_{k+1}^{(2)}$ contains the ones that do not. Similarly, write $\mathscr P_{k} = \mathscr P_{k}^{(1)} \sqcup \mathscr P_{k}^{(2)}$, where $\mathscr P_{k}^{(1)}$ consists of the chains that do not intersect $L$ and $\mathscr P_{k}^{(2)}$ contains the ones that do. (Note that all chains in $\mathscr P_{k+1}^{(2)}$ intersect $L$, while all chains in $\mathscr P_{k}^{(1)}$ contain $(k,1)$.)
    We will show that $w_{\mathscr P_{k+1}^{(i)}}(x) = w_{\mathscr P_{k}^{(i)}}(z)$ for $i=1,2$.

    For any chain $C \in \mathscr P_{k+1}$, $C \cap M_k$ is a chain with minimum $m$ and maximum $M$ such that either $m=(k+1,1)$ or $m = (k,j+1)$ for some $1\leq j \leq k$, and $M = (i+1,k+1)$ for some $k\leq i< r+k$.  

    Suppose first that $m = (k,j+1)$ and $M=(i+1,k+1)$. Let $m' = (k,j)$ and $M' = (i,k+1)$. Also let $\widetilde m$ and $\widetilde m'$ be the southeast neighbors of $m$ and $m'$, and let $\widetilde M$ and $\widetilde M'$ be the northeast neighbors of $M$ and $M'$ (if they exist). See Figure~\ref{fig:schematic}.
    
    Every chain in $\mathscr P_{k+1}^{(2)}$ is the union of a chain in some $\mathscr C^L_{m,M}$, a chain down from $\widetilde m$, and a chain up from $\widetilde M$, while every chain in $\mathscr P_{k}^{(2)}$ is the union of a chain in some $\mathscr C^L_{m',M'}$, a chain down from $\widetilde m'$, and a chain up from $\widetilde M'$. By Lemma~\ref{Lemma:rtchainshifting} and the fact that $\themap_k$ shifts labels outside $M_k$,
    \begin{align*}
        w_{\mathscr P_{k+1}^{(2)}}(x) &= \sum_{m,M} w_{\mathscr C^L_{m,M}}(x) \cdot \psi_{I_{k+1}}(x)_{\widetilde m}\psi_{I_{k+1}}^*(x)_{\widetilde M}\\
        &= \sum_{m,M} w_{\mathscr C^{L}_{m',M'}}(z)  \cdot \psi_{I_k}(z)_{\widetilde m'}\psi_{I_k}^*(z)_{\widetilde M'}= w_{\mathscr P_k^{(2)}}(z).
    \end{align*}    
    (By convention we set the value of $\psi$ or $\psi^*$ to be $1$ if the label does not exist.)

    A similar argument using Lemma~\ref{lemma:skewchainshifting} shows that $w_{\mathscr P_{k+1}^{(1)}}(x) = w_{\mathscr P_k^{(1)}}(z)$.
\end{proof}

As a remark, note that since in Lemmas~\ref{lemma:skewchainshifting} and \ref{Lemma:rtchainshifting}, the two collections of chains always have the same size, the technique in the above proof also shows that the intermediate posets $I_k$ all contain the same number of polygonal chains.

It is now simple to deduce the following theorem.

\begin{Th}\label{thm:main}
    Let $\mathscr C$ and $\mathscr C'$ be the sets of all maximal chains in $T_{r,s}$ and $R_{r,s}$, respectively. Then for all $x \in \RR_+^{T_{r,s}}$, $w_{\mathscr C}(x) = w_{\mathscr C'}(\themap(x))$.
\end{Th}
\begin{proof}
Apply Proposition~\ref{prop:intermediate} to $\themap = \themap_{1} \circ \themap_{2} \circ \dots \circ \themap_{s-1}$.
\end{proof}

\subsection{Polytopes and plane partitions} \label{sec:polytopes}
We now examine the consequences of Proposition~\ref{prop:intermediate} and Theorem~\ref{thm:main} in the piecewise-linear case.

In this section, we will take all maps to be their piecewise-linear counterparts. In particular, the definition of $\themap_k$ inside $M_k$ utilizes the map $\tilde \rho^{-1}$ on $M_k$. By tropicalizing Corollary~\ref{cor:edgeweight}, we can compute $z = \tilde\rho^{-1}(x)$ as
\[
    z_p = - \max_{q \gtrdot p} \{y_p-y_q\} = \min_{q \gtrdot p} \{y_q\}-y_p, \label{eq:yz} \tag{$\dagger$}
\]
where
\[y_p = \phibar(x)_p = \max_{\hat 0 \lessdot q_1 \lessdot \dots \lessdot q_n = p} \sum_i x_{q_i}. \label{eq:xy} \tag{$\ddagger$}\]

Note that if $x$ has nonnegative coordinates, then so does $z$ (except at a maximal element), from which it follows that $\zeta_k$ also preserves nonnegativity. (This nonnegativity property also follows from the fact that piecewise-linear toggles map the order polytope to itself.) It is also clear that $\themap_k$ is \emph{lattice-preserving} in that it sends lattice points to lattice points.

\begin{Ex}
    An example calculation of $\themap$ when $r=5$ and $s=4$ is given in Figure~\ref{fig:plexample}. Each $\zeta_k$ can be computed by applying \eqref{eq:xy} and then \eqref{eq:yz} on $M_k$. (The minimum element of $M_k$ is arbitrarily given the label $0$, and the label of the maximum element of $M_k$ is discarded at the end). The entries outside of $M_k$ are shifted downward appropriately.
\end{Ex}

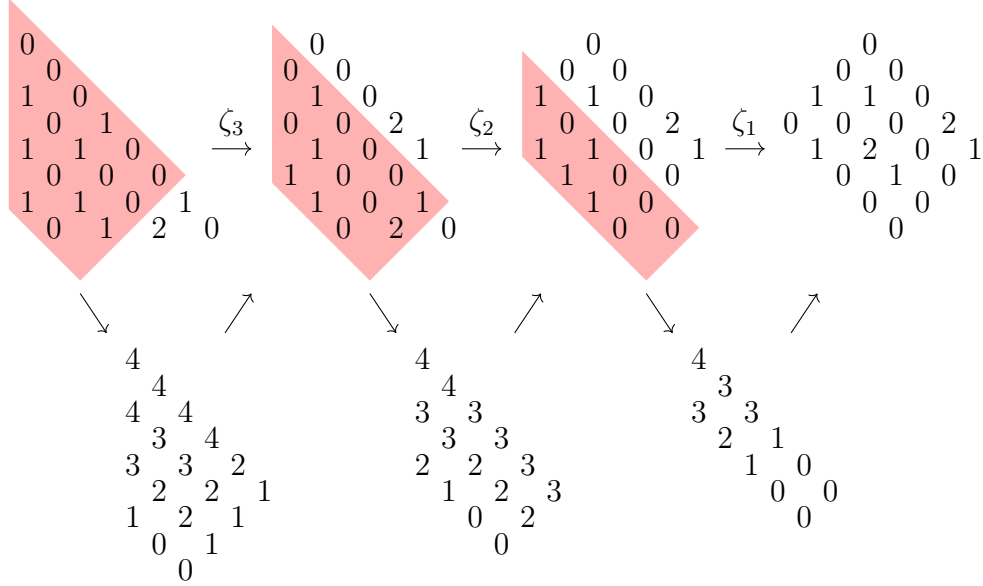
\begin{figure}
\centering
\begin{tikzpicture}[scale=.35]
    \tikzstyle{s}=[line width = .5cm, color=red!30]
    \begin{scope}
    \filldraw[s](-2,4)--(-4,6)--(-4,12)--(1,7)--cycle;
    \foreach \i/\j/\l in {4/1/0,5/1/1,3/2/1,4/2/1,5/2/0,6/2/1,2/3/2,3/3/0,4/3/0,5/3/1,6/3/0,7/3/1,1/4/0,2/4/1,3/4/0,4/4/0,5/4/1,6/4/0,7/4/0,8/4/0}
        \node at (\j-\i,\i+\j){$\l$};
    \end{scope}
    \draw[->] (-2,2.5)--(-1,1);
    \draw[->] (3.5,1)--(4.5,2.5);
    \begin{scope}[shift={(4,-12)}]
    \foreach \i/\j/\l in {3/1/0,4/1/0,5/1/1,3/2/1,4/2/2,5/2/2,6/2/3,3/3/1,4/3/2,5/3/3,6/3/3,7/3/4,3/4/1,4/4/2,5/4/4,6/4/4,7/4/4,8/4/4}
        \node at (\j-\i,\i+\j){$\l$};
    \end{scope}
    \draw[->](3,8)--(4.5,8);
    \node at (3.75,9){$\themap_3$};
    \begin{scope}[shift={(10,1)}]
    \filldraw[s](-1,3)--(-4,6)--(-4,10)--(1,5)--cycle;
    \foreach \i/\j/\l in {3/1/0,4/1/1,5/1/1,2/2/2,3/2/0,4/2/0,5/2/1,6/2/0,1/3/0,2/3/1,3/3/0,4/3/0,5/3/0,6/3/1,7/3/0,3/4/1,4/4/2,5/4/0,6/4/0,7/4/0}
        \node at (\j-\i,\i+\j){$\l$};
    \end{scope}
    \draw[->] (9,2.5)--(10,1);
    \draw[->] (14.5,1)--(15.5,2.5);
    \begin{scope}[shift={(15,-10)}]
    \foreach \i/\j/\l in {2/1/0,3/1/0,4/1/1,5/1/2,2/2/2,3/2/2,4/2/2,5/2/3,6/2/3,2/3/3,3/3/3,4/3/3,5/3/3,6/3/4,7/3/4}
        \node at (\j-\i,\i+\j){$\l$};
    \end{scope}
    \draw[->](12.5,8)--(14,8);
    \node at (13.25,9){$\themap_2$};
    \begin{scope}[shift={(19.5,2)}]
    \filldraw[s](0,2)--(-4,6)--(-4,8)--(1,3)--cycle;
    \foreach \i/\j/\l in {2/1/0,3/1/1,4/1/1,5/1/1,1/2/0,2/2/0,3/2/0,4/2/1,5/2/0,6/2/1,2/3/0,3/3/0,4/3/0,5/3/1,6/3/0,2/4/1,3/4/2,4/4/0,5/4/0,6/4/0}
        \node at (\j-\i,\i+\j){$\l$};
    \end{scope}
    \draw[->] (19.5,2.5)--(20.5,1);
    \draw[->] (25,1)--(26,2.5);
    \begin{scope}[shift={(25.5,-8)}]
    \foreach \i/\j/\l in {1/1/0,2/1/0,3/1/1,4/1/2,5/1/3,1/2/0,2/2/0,3/2/1,4/2/3,5/2/3,6/2/4}
        \node at (\j-\i,\i+\j){$\l$};
    \end{scope}
    \draw[->](22.5,8)--(24,8);
    \node at (23.25,9){$\themap_1$};
    \begin{scope}[shift={(29,3)}]
    \foreach \i/\j/\l in {1/1/0,2/1/0,3/1/0,4/1/1,5/1/0,1/2/0,2/2/1,3/2/2,4/2/0,5/2/1,1/3/0,2/3/0,3/3/0,4/3/1,5/3/0,1/4/1,2/4/2,3/4/0,4/4/0,5/4/0}
        \node at (\j-\i,\i+\j){$\l$};
    \end{scope}
\end{tikzpicture}
\caption{Example calculation of $\themap$ for $r=5$ and $s=4$. The top row shows the labelings of the intermediate posets obtained from $x$ when applying $\themap_3$, $\themap_2$, and $\themap_1$ with entries in $M_k$ highlighted. The vertical maps show applications of \eqref{eq:xy} and \eqref{eq:yz} on $M_k$.}
\label{fig:plexample}
\end{figure}

Just as we defined the chain polytopes of the rectangle and trapezoid, we can define polygonal chain polytopes for our intermediate posets $I_k \subseteq RT_{r,s}$.

\begin{Def}
    The \emph{polygonal chain polytope} $\widetilde{\mathcal C}(I_k) \subseteq \mathbb{R}^{I_k}$ is the set of all $\RR$-labelings $x = (x_p)_{p \in I_k}$ such that $x_p \geq 0$ for all $p \in I_k$, and $\sum_{p \in C} x_p \leq 1$ for all polygonal chains $C \subset I_k$.
\end{Def}

When $I_k$ is either $R_{r,s}$ or $T_{r,s}$, the polygonal chain polytope coincides with the chain polytope, but in general $\widetilde{\mathcal C}(I_k)$ will be larger than the chain polytope of $I_k$.

\begin{Ex}
Consider the poset $I_2 \subseteq RT_{3,3}$ pictured in Figure~\ref{fig:polygonalChains}. The polygonal chain polytope $\widetilde {\mathcal C}(I_2)$ is defined by the inequalities $x_p \geq 0$ for all $p$, and
\begin{align*}
    &x_{21}+x_{31}+x_{32}+x_{42}+x_{43} \leq 1, &x_{21}+x_{31}+x_{32}+x_{33}+x_{43} \leq 1, \\
    &x_{21}+x_{22}+x_{32}+x_{42}+x_{43} \leq 1, &x_{21}+x_{22}+x_{32}+x_{33}+x_{43} \leq 1, \\
    &x_{21}+x_{22}+x_{23}+x_{33}+x_{43} \leq 1, &x_{12}+x_{22}+x_{32}+x_{42}+x_{43} \leq 1.
\end{align*}
\end{Ex}

Although $\widetilde{\mathcal C}(I_k)$ is a lattice polytope when $k = 1$ or $k = s$ (when it is an ordinary chain polytope), this is not true in general (for instance, when $r=s=4$ and $k=3$). Nevertheless, for fixed $r$ and $s$, these polytopes all have the same volume and Ehrhart polynomial by Theorem~\ref{thm:polytope} below. In particular, if $\widetilde{\mathcal C}(I_k)$ is not a lattice polytope, then it exhibits \emph{period collapse} of its Ehrhart quasi-polynomial---see, for instance, \cite{haasemcallister,mcallisterwoods} for some discussion of this phenomenon.

\begin{Th} \label{thm:polytope}
    The map $\themap_k \colon \RR^{I_{k+1}} \to \RR^{I_{k}}$ defines a continuous, piecewise-linear, and lattice-preserving bijection from $\ell \cdot \widetilde {\mathcal C}(I_{k+1})$ to $\ell \cdot \widetilde {\mathcal C}(I_{k})$ for all $\ell \in \ZZ_{\geq 0}$.

    Hence, for fixed $r$ and $s$, the rational polytopes $\widetilde{\mathcal C}(I_k)$ exhibit Ehrhart quasi-polynomial period collapse and share the same Ehrhart polynomial for all $k$.
\end{Th}
\begin{proof}
    As seen above, $\themap_k$ is a continuous, piecewise-linear, lattice-preserving map that preserves nonnegativity. The tropicalized version of Proposition~\ref{prop:intermediate} states that for any $x \in \RR^{I_{k+1}}$ and $z = \themap_k(x)$,
    \[\max_{C \in \mathscr P_{k+1}} \sum_{p \in C} x_p = \max_{C \in \mathscr P_{k}} \sum_{p \in C} z_p.\]
    In particular, this quantity is at most $\ell$ if and only if $x \in \ell \cdot \widetilde {\mathcal C}(I_{k+1})$ by the left hand side, but also if and only if $z \in \ell \cdot \widetilde {\mathcal C}(I_{k})$ by the right hand side. The result follows.
\end{proof}

The following corollary is immediate.

\begin{Cor}
    The continuous, piecewise-linear map $\psi \circ \themap \circ \psi^{-1} \colon \RR^{T_{r,s}} \to \RR^{R_{r,s}}$ defines a bijection between plane partitions of $R_{r,s}$ and $T_{r,s}$ of height $\ell$ for all $\ell \in \ZZ_{\geq 0}$.
\end{Cor}
\begin{proof}
    By Theorem~\ref{thm:polytope}, the composition $\themap = \themap_1 \circ \cdots \circ \themap_k$ defines a continuous, piecewise-linear bijection from the lattice points in $\ell \cdot \mathcal C(T_{r,s})$ to the lattice points in $\ell \cdot \mathcal C(R_{r,s})$. But the inverse transfer map $\psi$ is a continuous, piecewise-linear bijection from the lattice points in $\ell \cdot \mathcal C(P)$ to the lattice points in $\ell \cdot \mathcal O(P)$ (which are plane partitions of $P$ of height at most $\ell$) for any poset $P$, as shown by Stanley \cite{stanley2}.
\end{proof}

\begin{Ex}
    Applying $\psi$ to both $x \in \RR^{T_{r,s}}$ and $\themap(x) \in \RR^{R_{r,s}}$ gives the plane partitions in Figure~\ref{fig:planepartition}, which are related by the bijection $\psi \circ \themap \circ \psi^{-1}$. As required, both plane partitions have the same height.
\end{Ex}

\begin{figure}
    \centering
    \begin{tikzpicture}[scale=.4]
        \begin{scope}
        \foreach \i/\j/\l in {4/1/0,5/1/1,3/2/1,4/2/2,5/2/2,6/2/3,2/3/2,3/3/2,4/3/2,5/3/3,6/3/3,7/3/4,1/4/0,2/4/3,3/4/3,4/4/3,5/4/4,6/4/4,7/4/4,8/4/4}
        \node at (\j-\i,\i+\j){$\l$};
        \end{scope}
        \draw[->] (3.5,8)--(8.5,8);
        \node at (6,8.8){$\psi \circ \themap \circ \psi^{-1}$};
        \begin{scope}[shift={(14,3)}]
        \foreach \i/\j/\l in {1/1/0,2/1/0,3/1/0,4/1/1,5/1/1,1/2/0,2/2/1,3/2/3,4/2/3,5/2/4,1/3/0,2/3/1,3/3/3,4/3/4,5/3/4,1/4/1,2/4/3,3/4/3,4/4/4,5/4/4}
        \node at (\j-\i,\i+\j){$\l$};
        \end{scope}
    \end{tikzpicture}
    \caption{Example of the bijection $\psi \circ \themap \circ \psi^{-1}$ obtained by applying $\psi$ to the labelings in Figure~\ref{fig:plexample}. Note that both plane partitions have the same height.}
    \label{fig:planepartition}
\end{figure}
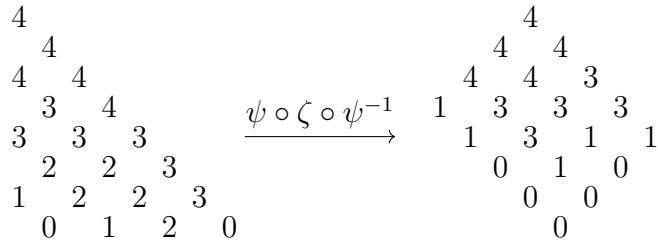

\section{Rowmotion equivariance} \label{sec:rowmotionequivariance}

In this section, we will show that the map $\themap$ defined in the previous section is equivariant with respect to the action of rowmotion $\tilde \rho$ (or, equivalently, that $\psi \circ \themap \circ \psi^{-1}$ is equivariant with respect to $\rho$). To do this, we will define a modified version of rowmotion on the intermediate posets $I_k$ that is respected by the maps $\themap_k$.

\subsection{Polygonal chain rowmotion}

As in the previous section, let $k \leq s \leq r$ and consider the intermediate posets $I_k \subset RT_{r,s}$. Let $\mathscr P_k$ denote the set of (maximal) polygonal chains in $I_k$, and let $\mathscr P_k(p)$ denote the subset of those chains that contain $p$.

\begin{Def}
\label{Def:polygonalChainToggle}
    For any $p \in I_k$, the \emph{(birational) polygonal toggle} $\tau'_p\colon \mathbb{R}_+^{I_k} \to \mathbb{R}_+^{I_k}$ is the map that changes the $p$-coordinate of $x \in \mathbb{R}_+^{I_k}$ by
    \[x_p \mapsto \left(\sum_{C \in \mathscr P_k(p)} \prod_{q \in C} x_q\right)^{-1} = w_{\mathscr P_k(p)}(x)^{-1},\] while keeping all other coordinates fixed.
    
    The \emph{(birational) polygonal rowmotion} map $\spicyrho_k:\mathbb{R}_+^{I_k} \to \mathbb{R}_+^{I_k}$ is the composition
    \[\spicyrho_k = \tau'_{L^{-1}(|I_k|)} \circ \dots \circ \tau'_{L^{-1}(1)},\]
    for any linear extension $L$ of $I_k$.
\end{Def}

As with ordinary toggles, it is easy to verify that the polygonal toggles $\tau'_p$ and $\tau'_q$ commute if $p$ and $q$ are incomparable (or more generally, if no polygonal chain contains both $p$ and $q$). It follows that $\spicyrho_k$ is well-defined.

Note that Definition~\ref{Def:polygonalChainToggle} differs from the definition of the usual antichain toggle only in that the sum is taken only over polygonal chains rather than all maximal chains. Therefore, in the case when $k=1$ (when $I_k = R_{r,s}$) or $k=s$ (when $I_k = T_{r,s}$), we have $\spicyrho_k = \tilde \rho_k$ since $\mathscr P_k$ is just the set of all maximal chains.

It is important to note that unlike ordinary rowmotion, there does not appear to be a nice ``order'' version of $\spicyrho$ (as $\rho$ is to $\tilde\rho$). The reason for this is that the inverse transfer and dual transfer maps $\phibar$ and $\phibar^*$ are particularly well suited for working with the set of all maximal chains. Indeed, these two maps enumerate the ``bottom'' and ``top'' parts of maximal chains through $p$, and any bottom part can be combined with any top part to yield a maximal chain. However, the same is not true for polygonal chains: the bottom part of one polygonal chain may not be compatible with the top part of another. Nevertheless, we will see that not all hope is lost because interchangeable parts of chains will satisfy an appropriate chain shifting lemma when considered as a group.

Our main result for this section will be to prove the following theorem. 

\begin{Th} \label{thm:equivariance}
    The maps $\themap_k \colon \RR_+^{I_{k+1}} \to \RR_+^{I_{k}}$ are equivariant with respect to the action of polygonal rowmotion:
    \[\themap_k \circ \spicyrho_{k+1} = \spicyrho_{k} \circ \themap_k.\]
\end{Th}

The following corollaries will then be immediate.
\begin{Cor}
    Let $T = T_{r,s}$ and $R=R_{r,s}$ be the rectangle and trapezoid poset. Then the map $\themap \colon \RR_+^{T} \to \RR_+^{R}$ is equivariant with respect to birational (antichain) rowmotion:
    \[\themap \circ \tilde\rho_T = \tilde\rho_R \circ \themap.\]
    In particular, birational rowmotion on the trapezoid ($\tilde \rho_T$ or $\rho_T$) has order $r+s$.
\end{Cor}
\begin{proof}
   Since the maps $\themap_k$ are equivariant with respect to polygonal rowmotion by Theorem~\ref{thm:equivariance}, so is their composition $\themap$. The first part follows since $\spicyrho = \tilde\rho$ on both $T$ and $R$. The second claim follows since birational rowmotion on the rectangle has order $r+s$ as shown by Grinberg and Roby \cite{grinbergroby2}.
\end{proof}

\begin{Cor}
    Polygonal rowmotion $\spicyrho_k$ has order $r+s$ on $I_k = I_{r,s,k}$.
\end{Cor}
\begin{proof}
    Similarly follows from Theorem~\ref{thm:equivariance} since birational rowmotion on the rectangle $R_{r,s}$ has order $r+s$.
\end{proof}

\subsection{Partial transfer maps}

For any element $p \in I_k$ and $x \in \RR_+^{I_k}$, $\phibar(x)_p$ is the total weight (with respect to $x$) of $\mathscr C_k(p)$, the set of all saturated chains from a minimal element of $I_k$ to $p$. In order to work with polygonal chains, we will need to split these chains into two types. Specifically, we partition $\mathscr C_k(p)$ into sets
\begin{align*}
    \mathscr C_k^{(1)}(p) &= \{C \in \mathscr C_k(p) \mid (k,1) \in C \text{ or } C \cap L \neq \varnothing\},\\
    \mathscr C_k^{(2)}(p) &= \{C \in \mathscr C_k(p) \mid (k,1) \notin C \text{ and } C \cap L = \varnothing\},
\end{align*}
and let $\phibar_i(x)_p$ denote the total weight of $\mathscr C_k^{(i)}(p)$ (with respect to $x$) for $i=1,2$. If $p$ is nonminimal, then 
$\phibar_i(x)_p = x_p\sum_{q \lessdot p} \phibar_i(x)_q$ if $p \notin L$, while $\phibar_1(x)_p = x_p(\phibar_1(x)_q+\phibar_2(x)_q)$ and $\phibar_2(x)_p = 0$ if $q \lessdot p \in L$.

Similarly, $\phibar^*(x)_p$ is the total weight of $\mathscr C^*_k(p)$, the set of all saturated chains from $p$ to the maximum element of $I_k$. We partition $\mathscr C^*_k(p)$ into sets
\begin{align*}
    \mathscr C_k^{*,(1)}(p) &= \{C \in \mathscr C^*_k(p) \mid C \cap L \neq \varnothing\},\\
    \mathscr C_k^{*,(2)}(p) &= \{C \in \mathscr C^*_k(p) \mid C \cap L = \varnothing\},
\end{align*}
and let $\phibar^*_i(x)_p$ denote the total weight of $\mathscr C_k^{*,(i)}(p)$ for $i=1,2$. If $p$ is nonmaximal, then 
$\phibar_i^*(x)_p = x_p\sum_{q \gtrdot p} \phibar_i^*(x)_q$ if $p \notin L$, while $\phibar_1^*(x)_p = x_p(\phibar_1^*(x)_q+\phibar_2^*(x)_q)$ and $\phibar_2^*(x)_p = 0$ if $q \gtrdot p \in L$.

Observe that if $C$ is any polygonal chain containing $p$, then $C$ can be written uniquely as a union of a chain in $\mathscr C_k^{(i)}(p)$ and a chain in $\mathscr C_k^{*,(j)}(p)$ for some pair $(i,j) \neq (2,2)$. (Combining chains in $\mathscr C_k^{(2)}(p)$ and $\mathscr C_k^{*,(2)}(p)$ yields a non-polygonal maximal chain.)

The following proposition relates these partial transfer maps to polygonal rowmotion. The proof is analogous to that of Proposition~\ref{prop:rhotilde}.

\begin{Prop} \label{prop:specialrho}
    For any $p \in I_k$ and $x \in \RR_+^{I_k}$,
    \[
    \sum_{(i,j) \neq (2,2)} \phibar^*_i(x)_p \phibar_j(\tilde\varrho(x))_p = 1.
    \]
\end{Prop}
\begin{proof}
    In the computation of $\tilde\varrho(x)$, the polygonal toggles in $\tilde\varrho$ occur from bottom to top. Thus immediately before the toggle $\tau'_p$ is performed, all of the elements less than $p$ have been toggled and none of the elements at $p$ or above have been toggled. Then the total weight of all polygonal chains through $p$ before the toggle at $p$ is
    \[
    \sum_{(i,j) \neq (2,2)} \phibar^*_i(x)_p \cdot \phibar_j(\tilde\varrho(x))_p \cdot (\tilde\varrho(x)_p)^{-1}.
    \]
    By the definition of $\tau_p'$, this weight equals the inverse of the new coordinate at $p$ after the toggle, that is, $(\tilde\varrho(x)_p)^{-1}$. The result follows easily.
\end{proof}

\subsubsection{Chain shifting}
To study the behavior of the partial transfer maps under the action of $\themap_k$, we will need to prove appropriate chain shifting lemmas for them. Luckily, since $\themap_k$ mostly consists of the rowmotion map $\tilde \rho^{-1}$ on $M_k \subset I_{k+1}$, it is straightforward to adapt the results from Sections~\ref{sec:partialchains} and \ref{sec:partialchainsintrapezoids}. These results will be extended from $M_k$ to all of $I_k$ in the same way as the proof of Proposition~\ref{prop:intermediate}: by factoring out the parts of each chain that lie in $M_k$, we can apply a chain shifting result to this part, while the portions of the chains lying outside of $M_k$ will be shifted by the definition of $\themap_k$ outside of $M_k$.

Recall that for $x \in \RR_+^{I_{k+1}}$, we define $\bar x \in \RR_+^{M_k}$ to be the labeling obtained by restricting $x$ to the coordinates in $M_k$ and setting $\bar x_{k,1} = 1$.

\begin{Prop} \label{prop:outside1}
    Let $x \in \RR_+^{I_{k+1}}$ and $z = \themap_k(x) \in \RR_+^{I_{k}}$. Let $p \in I_k$ lie below $M_k$ and let $q \in I_{k+1}$ be its northeast neighbor. Then $\psi_1(z)_p = 0$, $\psi_2(z)_p = \psi_2(x)_q$, and $\psi^*_1(z)_p = \psi_1^*(x)_q$.
\end{Prop}
\begin{proof}
    Clearly $\psi_1(z)_p = 0$ since $\mathscr C_k^{(1)}(p)$ is empty. Since chains in $\mathscr C_k^{(2)}(p)$ are southwest shifts of chains in $\mathscr C_{k+1}^{(2)}(q)$, and the labels shift similarly, we have $\psi_2(z)_p = \psi_2(x)_q$. The final claim follows from Lemma~\ref{Lemma:rtchainshifting} as in the proof of Proposition~\ref{prop:intermediate}.
\end{proof}

We can likewise prove the following dual version.

\begin{Prop} \label{prop:outside2}
    Let $x \in \RR_+^{I_{k+1}}$ and $z = \themap_k(x) \in \RR_+^{I_{k}}$. Let $p \in I_k$ lie above $M_k$ and let $q \in I_{k+1}$ be its northwest neighbor. Then $\psi^*_1(z)_p = 0$, $\psi^*_2(z)_p = \psi^*_2(x)_q$, and $\psi_1(z)_p = \psi_1(x)_q$.
\end{Prop}
\begin{proof}
    Analogous to the proof of Proposition~\ref{prop:outside1}.
\end{proof}

For the labels inside $M_k$, the chain shifting is slightly more complicated. For $p \in M_k$, recall the definition of $se(p)$ and $sw(p)$ (defined as subsets of $M_k$) from Section~\ref{sec:partialchains}.

\begin{Prop} \label{prop:partialshiftdown}
    Let $x \in \RR_+^{I_{k+1}}$, $y = \phibar(\bar x) \in \RR_+^{M_k}$, and $z = \themap_k(x) \in \RR_+^{I_{k}}$. Then for $p \in M_k \setminus \{(k,1)\}$, 
    \begin{align*}
        \phibar_1(x)_p &= \sum_{u \in se(p)} \sum_{\substack{u' \gtrdot u\\u' \in M_k}} \frac{y_p}{y_{u'}} \cdot \phibar_1(z)_{u},\\ 
        \phibar_2(x)_p &= \sum_{u \in sw(p)} \sum_{\substack{u' \gtrdot u\\u' \in M_k}} \frac{y_p}{y_{u'}} \cdot (\phibar_1(z)_{u}+\phibar_2(z)_{u}).
    \end{align*}
\end{Prop}
\begin{proof}
    For the first equation, setting $m=(k,2)$, $n=(k+1,1)$, and $m'=n'=(k,1)$ in the first and third equation of Lemma~\ref{lemma:trapezoidshift}(b) and summing gives
    \[w_{\mathscr C_{(k+1,1),p}(x)}+w_{\mathscr C^L_{(k,2),p}}(x) = \sum_{u \in se(p)} \sum_{\substack{u' \gtrdot u\\u' \in M_k}} \frac{y_p}{y_{u'}} \cdot w_{\mathscr C_{(k,1),u}}(z).\]
    For any other minimal element $(k+2-j,j) \in I_{k+1}$ for $j > 2$, applying the argument of Proposition~\ref{prop:intermediate} using the first equation of Lemma~\ref{lemma:trapezoidshift}(b) gives
    \[w_{\mathscr C^L_{(k+2-j,j),p}}(x) = \sum_{u \in se(p)} \sum_{\substack{u' \gtrdot u\\u' \in M_k}} \frac{y_p}{y_{u'}} \cdot w_{\mathscr C^L_{(k+1-j,j),u}}(z).\]
    Summing over all minimal elements of $I_{k+1}$ gives the desired equation for $\psi_1(x)_p$.
    
    The second equation follows by similarly applying the argument of Proposition~\ref{prop:intermediate} to the second equation of Lemma~\ref{lemma:trapezoidshift}(b).
\end{proof}

We can simplify this expression if we are willing to use subtraction.

\begin{Prop} \label{prop:partialshiftup}
    Let $x \in \RR_+^{I_{k+1}}$, $y = \phibar(\bar x) \in \RR_+^{M_k}$, and $z = \themap_k(x) \in \RR_+^{I_{k}}$. Suppose $p,q,u \in M_k$ such that $p \gtrdot u$ and $q \gtrdot u$ with $p$ to the left of $q$. Then
    \begin{align*}
        \phibar_1(x)_p \cdot \frac{y_q}{y_p+y_q} - \phibar_1(x)_q \cdot \frac{y_p}{y_p+y_q} &= \phibar_1(z)_u,\\
        \phibar_2(x)_p \cdot \frac{y_q}{y_p+y_q} - \phibar_2(x)_q \cdot \frac{y_p}{y_p+y_q} &= -(\phibar_1(z)_u+\phibar_2(z)_u).
    \end{align*}
\end{Prop}
\begin{proof}
    By construction, $se(p) = se(q) \cup \{u\}$. Thus by Proposition~\ref{prop:partialshiftdown} 
    \[\frac{\psi_1(x)_p}{y_p} - \frac{\psi_1(x)_q}{y_q} = \left(\frac{1}{y_p}+\frac{1}{y_q}\right)\psi_1(z)_u.\]
    Rearranging gives the first desired equation. The other follows similarly using the fact that $sw(q) = sw(p) \cup \{u\}$.
\end{proof}

Using a dual argument, one can likewise show the following two results by appealing to Lemma~\ref{lemma:trapezoidshiftdual}.

\begin{Prop} \label{prop:partialshiftdowndual}
    Let $x \in \RR_+^{I_{k+1}}$, $y = \phibar(\bar x) \in \RR_+^{M_k}$, and $z = \themap_k(x) \in \RR_+^{I_{k}}$. Then for $p \in M_k \setminus \{(r+k,k+1)\}$, 
    \begin{align*}
        \phibar^*_1(z)_p &= \sum_{v \in ne(p)} \sum_{\substack{v' \lessdot v\\v' \in M_k}} \frac{y_{v'}}{y_{p}} \cdot \phibar^*_1(x)_{v},\\
        \phibar^*_2(z)_p &= \sum_{v \in nw(p)} \sum_{\substack{v' \lessdot v\\v' \in M_k}} \frac{y_{v'}}{y_{p}} \cdot (\phibar_1^*(x)_{v}+\phibar_2^*(x)_{v}).
    \end{align*}
\end{Prop}

\begin{Prop} \label{prop:partialshiftupdual}
    Let $x \in \RR_+^{I_{k+1}}$, $y = \phibar(\bar x) \in \RR_+^{M_k}$, and $z = \themap_k(x) \in \RR_+^{I_{k}}$. Suppose $p,q,v \in M_k$ such that $p \lessdot v$ and $q \lessdot v$ with $p$ to the left of $q$. Then
    \begin{align*}
        \phibar_1^*(z)_p \cdot \frac{y_p}{y_p+y_q} - \phibar_1^*(z)_q \cdot \frac{y_q}{y_p+y_q} &= \phibar_1^*(x)_v,\\
        \phibar_2^*(z)_p \cdot \frac{y_p}{y_p+y_q} - \phibar_2^*(z)_q \cdot \frac{y_q}{y_p+y_q} &= -(\phibar_1^*(x)_v+\phibar_2^*(x)_v).
    \end{align*}
\end{Prop}

\subsection{Main calculation}
The main calculation needed for the proof of Theorem~\ref{thm:equivariance} is the following result, which expresses the difference between a label before and after applying $\themap_k \circ \spicyrho$ in terms of certain weights of chains. Since $\spicyrho^{-1}$ and $\themap_k^{-1}$ are obtained by reversing the order of the toggles in $\spicyrho$ and $\themap_k^{-1}$, the dual version of this argument will give us the analogous result for $\themap_k^{-1} \circ \spicyrho^{-1}$ as well.

\begin{Prop} \label{prop:x-z}
    Let $x \in \RR_+^{I_{k+1}}$, $x' = \spicyrho(x)$, and $z' = \themap_k(x') \in \RR_+^{I_k}$. For any nonminimal and nonmaximal $p \in M_k$, let $v_1$ and $v_2$ be the northwest and northeast neighbors of $p$ in $I_{k+1}$, and let $u_1$ and $u_2$ be the southwest and southeast neighbors of $p$ in $I_k$ (if they exist). Then
    \begin{align*}
        x_p^{-1} - {z'_p}^{-1} = \sum_{t=1}^2 (-1)^{t-1}(\psi_1^*(x)_{v_t} \psi_2(z')_{u_t} - \psi_2^*(x)_{v_t}\psi_1(z')_{u_t}).
    \end{align*}
    (If $u_t$ or $v_t$ does not exist, then remove the $t$th term from the sum.)
\end{Prop}
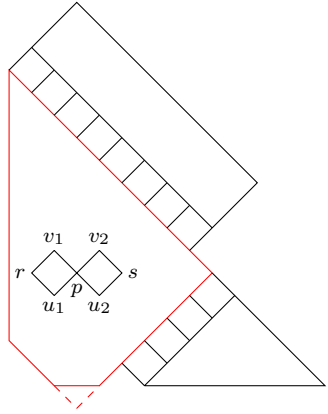
\begin{figure}
    \centering
    \begin{tikzpicture}[scale=.3]
        \draw[red] (-1,1)--(-3,3)--(-3,15)--(6,6)--(1,1)--(-1,1);
        \draw[red,dashed] (-1,1)--(0,0)--(1,1);
        \draw (-1,5)--(1,7) (-1,7)--(1,5) (-1,5)--(-2,6)--(-1,7) (1,5)--(2,6)--(1,7);
        \node[label={[label distance=-4pt]below:$\scriptstyle p$}] at (0,6){};
        \node[label={[label distance=-6pt]below:$\scriptstyle u_1$}] at (-1,5){};
        \node[label={[label distance=-6pt]below:$\scriptstyle u_2$}] at (1,5){};
        \node[label={[label distance=-6pt]above:$\scriptstyle v_1$}] at (-1,7){};
        \node[label={[label distance=-6pt]above:$\scriptstyle v_2$}] at (1,7){};
        \node[label={[label distance=-6pt]left:$\scriptstyle r$}] at (-2,6){};
        \node[label={[label distance=-6pt]right:$\scriptstyle s$}] at (2,6){};
        \draw (3,1)--(7,5)--(11,1)--(3,1); 
        \draw (6,8)--(-2,16)--(0,18)--(8,10)--(6,8); 
        \draw (3,1)--(2,2) (4,2)--(3,3) (5,3)--(4,4) (6,4)--(5,5) (7,5)--(6,6);
        \foreach \x in {1, ..., 9}
            \draw (6-\x,6+\x)--(7-\x,7+\x);
    \end{tikzpicture}
    \caption{The generic case of Proposition~\ref{prop:x-z}.}
    \label{fig:uv}
\end{figure}
\begin{proof}
    If $p \notin L$, then dividing Proposition~\ref{prop:specialrho} by $x_p$ gives
    \[x_p^{-1} = \sum_{(i,j) \neq (2,2)} \frac{\psi_i^*(x)_p}{x_p} \psi_j(x')_p = \sum_{(i,j) \neq (2,2)} \sum_{v \gtrdot p} \psi_i^*(x)_v \psi_j(x')_p \tag{$*$} \label{eq:xp-1}.\]
    If instead $p \in L$, then $\psi_2^*(x)_p = \psi_2(x')_p = 0$, so 
    \[x_p^{-1} = \frac{\psi_1^*(x)_p}{x_p} \psi_1(x')_p = (\psi_1^*(x)_{v_2} + \psi_2^*(x)_{v_2})\psi_1(x')_p,\] which also implies \eqref{eq:xp-1} since $\psi_2(x')_p=0$.
    
    Let $y' = \psi(\bar x') \in \RR_+^{M_k}$. We have by Corollary~\ref{cor:edgeweight} and Proposition~\ref{prop:specialrho} that
    \[{z_p'}^{-1} = \sum_{\substack{v \gtrdot p\\v \in M_k}} \frac{y'_p}{y'_v} \sum_{(i,j) \neq (2,2)}  \psi_i^*(x)_v {\psi_j(x')_v}.\]
    Thus subtracting from \eqref{eq:xp-1} gives
    \[x_p^{-1}-{z'_p}^{-1} = \sum_{(i,j) \neq (2,2)}\Bigg(\sum_{\substack{v \gtrdot p\\v \in M_k}}  \psi_i^*(x)_v (\psi_j(x')_p - \frac{y'_p}{y'_v} \psi_j(x')_v) + \sum_{\substack{v \gtrdot p\\v \notin M_k}} \psi_i^*(x)_v \psi_j(x')_p\Bigg). \]
    
    We claim that when $v_t \in M_k$,
    \begin{multline}\sum_{(i,j) \neq (2,2)} \psi_i^*(x)_{v_t} (\psi_j(x')_p - \frac{y'_p}{y'_{v_t}} \psi_j(x')_{v_t}) \\
    = \begin{cases}
        (-1)^{t-1}(\psi_1^*(x)_{v_t} \psi_2(z')_{u_t} - \psi_2^*(x)_{v_t} \psi_1(z')_{u_t}) &\text{if $u_t$ exists,}\\
        0&\text{otherwise.}
        \end{cases}\tag{$**$} \label{eq:psistarpsi}\end{multline}
    In fact, this will imply the result: this is clear if each $v_t$ that exists lies in $M_k$. However, it can occur that $v_2 \notin M_k$ when $p$ lies on the northeast boundary of $M_k$. But then $\psi_1^*(x)_{v_2} = 0$, while $\psi_1(x')_{p} = \psi_1(z')_{u_2}$ by Proposition~\ref{prop:partialshiftdown}. These imply that the extra term $\sum_{i,j \neq (2,2)} \psi_i^*(x)_{v_2} \psi_j(x')_{p}$ matches the right hand side of \eqref{eq:psistarpsi}, as needed.
    
    We now prove the claim. If $t=1$, the generic case is when $v_1$ has a southwest neighbor $r \in M_k$. (See Figure~\ref{fig:uv}.) Then 
    \begin{align*}
        \psi_j(x')_p - \frac{y'_p}{y'_{v_1}} \psi_j(x')_{v_1} &= \psi_j(x')_p - \frac{y'_p}{x'_{v_1}(y'_{r}+y'_p)}\cdot x'_{v_1} (\psi_j(x')_r+\psi_j(x')_p)\\
        &=\psi_j(x')_p \frac{y'_r}{y'_r+y'_p} - \psi_j(x')_r \frac{y'_p}{y'_r+y'_p}\\
        &=\begin{cases}-\psi_1(z')_{u_1} &\text{if }j=1,\\
        \psi_1(z')_{u_1} + \psi_2(z')_{u_1} &\text{if }j=2,
        \end{cases}
    \end{align*}
    by Proposition~\ref{prop:partialshiftup}. This easily implies \eqref{eq:psistarpsi}. A similar argument shows that \eqref{eq:psistarpsi} holds when $t=2$ and $v_2$ has southeast neighbor $s$ that lies in $M_k$. 

    If instead $t=1$ but $v_1$ has no southwest neighbor, then $\psi_2(x')_{v_1}=0$ since $v_1$ either lies in $L$ or on the southwest boundary of $M_k$. In either case, $\psi_1(x')_{v_1} = x'_{v_1}(\psi_1(x')_p+\psi_2(x')_p)$ (for if $v_1$ lies on the southwest edge of $M_k$, then so does $p$, so $\psi_2(x')_p=0$). These imply (using $y'_{v_1} = x'_{v_1}y'_p$) that
    \[\psi_j(x')_p - \frac{y'_p}{y'_{v_1}}\psi_j(x')_{v_1} = \begin{cases}
        -\psi_2(x')_p& \text{if }j=1,\\
        \psi_2(x')_p&\text{if }j=2.
    \end{cases}\]
    Then
    \[\sum_{(i,j) \neq (2,2)} \psi_i^*(x)_{v_1} (\psi_j(x')_p - \frac{y'_p}{y'_{v_1}} \psi_j(x')_{v_1}) = -\psi_2^*(x)_{v_1} \psi_2(x')_{p}.\]
    If $u_1$ exists, then it must lie in $L$, so $\psi_2(z')_{u_1}=0$. Then $\psi_2(x')_p = \psi_1(z')_{u_1}$ by Proposition~\ref{prop:partialshiftdown} if $u_1$ exists, while $\psi_2(x')_p = 0$ if $u_1$ does not exist. In either case, we again obtain \eqref{eq:psistarpsi}. A similar argument applies when $t=2$ and $v_2$ has no southeast neighbor.

    The final case is when $t=2$ and $v_2$ has a southeast neighbor $s \in I_{k+1}$ that does not lie in $M_k$. Then $p$ and $v_2$ lie on the southeast boundary of $M_k$. Therefore $\psi_1(x')_p=\psi_1(x')_{v_2}=0$ and $\psi_2(x')_{v_2} = x'_{v_2}(\psi_2(x')_p+\psi_2(x')_s)$. Using $y'_{v_2} = x'_{v_2}y'_p$, we can compute
    \[\sum_{(i,j) \neq (2,2)} \psi_i^*(x)_{v_2} (\psi_j(x')_p - \frac{y'_p}{y'_{v_2}} \psi_j(x')_{v_2}) = -\psi_1^*(x)_{v_2} \psi_2(x')_{s}.\] But since $\themap_k$ shifts labels outside $M_k$, this equals $-\psi_1^*(x)_{v_2}\psi_2(z')_{u_2}$, which implies \eqref{eq:psistarpsi} (as $\psi_1(z')_{u_2}=0$). This completes the proof of the claim and the result.
\end{proof}

Using a dual argument, we can similarly prove the following result.

\begin{Prop} \label{prop:x-zdual}
    Let $z'' \in \RR_+^{I_k}$, $z = \spicyrho^{-1}(z'')$, and $x = \zeta_k^{-1}(z) \in \RR_+^{I_{k+1}}$.    
    For any nonminimal and nonmaximal $p \in M_k$, let $v_1$ and $v_2$ be the northwest and northeast neighbors of $p$ in $I_{k+1}$, and let $u_1$ and $u_2$ be the southwest and southeast neighbors of $p$ in $I_k$ (if they exist). Then
    \begin{align*}
        {z''_p}^{-1} - x_p^{-1} = \sum_{i=1}^2 (-1)^{i-1}(\psi_1(z'')_{u_i} \psi_2^*(x)_{v_i} - \psi_2(z'')_{u_i}\psi_1^*(x)_{v_i}).
    \end{align*}
    (If $u_i$ or $v_i$ does not exist, then remove the $i$th term from the sum.)
\end{Prop}
\begin{proof}
    This follows from an analogous argument to Proposition~\ref{prop:x-z} by replacing $x$ and $z'$ with $z''$ and $x$, respectively, as well as flipping the poset vertically (thereby switching the roles of $\psi_i^*$ and $\psi_i$, north and south, $v_i$ and $u_i$, $y'_p$ and $y_p^{-1}$, and so forth).
\end{proof}

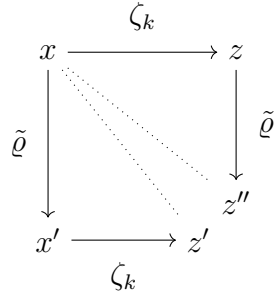
\begin{figure}
    \begin{tikzpicture}
    \node(x) at (0,2.5){$x$};
    \node(x') at (0,0){$x'$};
    \node(z) at (2.5,2.5){$z$};
    \node(z') at (2,0){$z'$};
    \node(z'') at (2.5,.5){$z''$};
    \draw[->] (x) to node[midway,label=left:$\spicyrho$]{} (x');
    \draw[->] (x') to node[midway,label=below:$\themap_k$]{} (z');
    \draw[->] (x) to node[midway,label=above:$\themap_k$]{} (z);
    \draw[->] (z) to node[midway,label=right:$\spicyrho$]{} (z'');
    \draw[dotted](x)--(z') (x)--(z'');
    \end{tikzpicture}
    \caption{A diagram showing the relationship between the labelings in Propositions~\ref{prop:x-z} and \ref{prop:x-zdual}. The labelings on the left are on $I_{k+1}$ while those on the right are on $I_k$.}
    \label{fig:cd}
\end{figure}

Figure~\ref{fig:cd} shows the relationship between the labelings in Propositions~\ref{prop:x-z} and \ref{prop:x-zdual}. Note that the equations  resulting from these two propositions are very similar: in fact, if $z''$ and $z'$ agree everywhere below $p$, then they immediately imply $z''_p = z'_p$ as well. This will serve as the main step of the induction in the proof of Theorem~\ref{thm:equivariance}.

\subsection{Proof of equivariance}

All that remains is to combine these pieces to give the final proof that the maps $\themap_k$ are equivariant with respect to polygonal rowmotion.

\begin{proof}[Proof of Theorem~\ref{thm:equivariance}]
Let $x \in \RR_+^{I_{k+1}}$, $x' = \spicyrho(x)$, $z' = \themap_k(x')$, $z = \themap_k(x)$, and $z'' = \spicyrho(z)$. (See Figure~\ref{fig:cd}.) We will prove that $z'_p = z''_p$ for all $p \in I_k$ by induction on $p$, so assume that $z'$ and $z''$ agree at all labels below $p$.

First suppose $p \in I_k$ lies below $M_k$, and let $q \in I_{k+1}$ be its northeast neighbor, so that $z'_p = x'_q$ by the definition of $\themap_k$. We then have by Propositions~\ref{prop:specialrho} and \ref{prop:outside1} that
\[{z'_p}^{-1} = {x'_q}^{-1} = \psi_1^*(x)_q \cdot \frac{\psi_2(x')_q}{x'_q} = \psi_1^*(z)_p \cdot \frac{\psi_2(z')_p}{z'_p}.\]
The second factor depends only on the labels of $z'$ below $p$, so by the inductive hypothesis, we can replace $z'$ with $z''$ to obtain 
\[\psi_1^*(z)_p \cdot \frac{\psi_2(z')_p}{z'_p} = \psi_1^*(z)_p \cdot \frac{\psi_2(z'')_p}{z''_p} = {z''_p}^{-1}\]
(using Proposition~\ref{prop:specialrho} applied to $z$), so that $z'_p = z''_p$.

Next, suppose $p \in M_k$. If $p=(k,1)$, the minimal element of $M_k$, then by the definition of $\spicyrho(z)$, ${z''_{k1}}^{-1}$ is the total weight of all polygonal chains through $p$ in $I_k$ with respect to $z$. By the proof of Proposition~\ref{prop:intermediate}, this is equal to the total weight of all polygonal chains through $(k+1,1)$ or $(k,2)$ in $I_{k+1}$ with respect to $x$. But by the definition of $\spicyrho(x)$, this equals $(x'_{k+1,1})^{-1} + (x'_{k2})^{-1}$, which equals ${z'_{k1}}^{-1}$ by the definition of $\zeta_k(x')$. It follows that $z'_{k1} = z''_{k1}$. If $p \neq (k,1)$, then comparing Propositions~\ref{prop:x-z} and \ref{prop:x-zdual}, we see that the inductive hypothesis implies $x_p^{-1}-{z'_p}^{-1}= x_p^{-1}-{z''_p}^{-1}$, so again $z'_p = z''_p$.

Finally, suppose $p \in I_k$ lies above $M_k$, and let $q \in I_{k+1}$ be its northwest neighbor, so that $z'_p = x'_q$. As in the first case above, we have by Propositions~\ref{prop:specialrho} and \ref{prop:outside2} that
\begin{align*}
    {z'_p}^{-1} = {x'_q}^{-1} = \psi_2^*(x)_q \cdot \frac{\psi_1(x')_q}{x'_q}
    &= \psi_2^*(z)_p \cdot \frac{\psi_1(z')_p}{z'_p}\\
    &= \psi_2^*(z)_p \cdot \frac{\psi_1(z'')_p}{z''_p} = {z''_p}^{-1},
\end{align*}
completing the proof.
\end{proof}

\section{Future directions} \label{sec:conclusion}

In this section we will discuss some potential directions for future study.

\subsection{Combinatorial properties}
In this work, we have defined a birational map $\themap$ between labelings of $T_{r,s}$ and $R_{r,s}$ and shown that it can be used to relate both plane partitions and rowmotion on these two posets. It would be interesting to study the combinatorial properties of this map, such as its effect on various combinatorial statistics such as described by Hopkins \cite{hopkins2}, or its relation to other combinatorial bijections between plane partitions such as the ones given in \cite{elizalde,hamakerpatriaspechenikwilliams}.

In prior work, Musiker and Roby \cite{musikerroby} give a combinatorial formula for iterated rowmotion on $R_{r,s}$. This formula was studied further by the present authors in \cite{johnsonliu}, where it is related to the Lindstr\"om-Gessel-Viennot lemma, Dodgson condensation, the octahedron recurrence, and birational RSK. Can a similar formula also be derived for iterated rowmotion on $T_{r,s}$ in terms of lattice paths, arborescences, or related combinatorial objects? It may also turn out that, even without explicit formulas, the description of $\themap$ we give here may be enough to prove dynamical properties of rowmotion on $T_{r,s}$ which have been proved for $R_{r,s}$ such as \emph{homomesy phenomena} \cite{propproby}.

\subsection{Other shapes and types of rowmotion}

Some of the results derived above apply to posets shaped like general skew shapes or other saturated subposets of $\ZZ \times \ZZ$. Can the relationship between $T_{r,s}$ and $R_{r,s}$ be generalized to other pairs of posets? It is worth noting that for most posets, rowmotion does not have finite order. However, this does not preclude the possibility of a birational map between two shapes with similar properties with respect to plane partitions and rowmotion.

For the intermediate posets $I_k$ we consider above, we define a special polygonal rowmotion $\spicyrho$ that has finite order. More generally, we could define an analogue of rowmotion where we replace $I_k$ with a general poset $P$ and the polygonal chains with any subset of chains. We might call this general notion \emph{restricted chain rowmotion}. As is the case with classical rowmotion, typically restricted chain rowmotion is poorly behaved. Moreover, as with polygonal chain rowmotion, there may be no order rowmotion analogue nor a natural transfer map for working with it. Our results give the first nontrivial instances of restricted chain rowmotion exhibiting desirable dynamical properties. It would be interesting to study this notion further for other posets, particularly other saturated subposets of $\ZZ \times \ZZ$ (such as \emph{moon polyominoes}, which were studied in related work by the present authors \cite{johnsonliu2}).

In \cite{galashinpylyavskyy}, Galashin and Pylyavskyy generalize the notion of rowmotion to \emph{$R$-systems}, in particular investigating the notions of \emph{singularity confinement} and \emph{algebraic entropy} to quantify their complexity. Their study suggests in particular that rowmotion on certain \emph{octagonal} posets (of which rectangles and trapezoids are special cases) may also exhibit relatively nice dynamical properties, so it would be interesting to study these posets or other related $R$-systems using these ideas. (It is worth noting that the polygonal rowmotion $\spicyrho$ that we define here cannot obviously be expressed in terms of an $R$-system.)

\subsection{Polytopes and combinatorial mutation}

The tropicalizations of the maps $\themap_k$ are piecewise-linear, continuous, volume-preserving maps that preserve the integer lattice. Since they are constructed out of toggles, they are essentially also (up to duality) examples of \emph{combinatorial mutations} as defined in \cite{akhtarcoatesgalkinkasprzyk}. (For other related work, see for instance \cite{ardilabliemsalazar,berensteinzelevinsky,rietschwilliams}.) In our work, we prove that the polygonal chain polytopes of each $I_k$ have the same Ehrhart polynomial using such combinatorial mutations, which, in particular, implies that they exhibit Ehrhart quasi-polynomial \emph{period collapse} if they are not lattice polytopes. (See \cite{haasemcallister,mcallisterwoods} for some discussion of this phenomenon.) Are there other nice examples of polytopes arising in combinatorics that can be shown to be Ehrhart equivalent or exhibit period collapse in a similar fashion? In particular, it would be interesting if such a map between polytopes could be constructed as the tropicalization of a birational map as we have done here.

\section{Acknowledgments}

The authors would like to thank Darij Grinberg, Sam Hopkins, Tom Roby, and Sylvester Zhang for interesting conversations.

\bibliographystyle{acm}
\bibliography{citations}

\end{document}